\documentclass[10pt]{amsart}

\usepackage{amsmath,amssymb,amsthm}
\usepackage{indentfirst,color}
\usepackage[colorlinks,citecolor=red,linkcolor=blue,urlcolor=cyan]{hyperref}
\usepackage{newtxtext}

\usepackage[reversemp,paperwidth=155mm,paperheight=235mm,%
top={19.5mm},headheight={5.5pt},headsep={5.6mm},text={119mm,194.25mm},%
marginparsep=5mm,marginparwidth=12mm,bindingoffset=6mm,footskip=10mm]{geometry}

\theoremstyle{plain}%
\newtheorem{theorem}{Theorem}[section]%
\newtheorem{proposition}[theorem]{Proposition}%
\newtheorem{lemma}[theorem]{Lemma}%
\newtheorem{corollary}[theorem]{Corollary}%

\theoremstyle{definition}%
\newtheorem{definition}[theorem]{Definition}%
\theoremstyle{remark}%
\newtheorem{remark}[theorem]{Remark}%
\newtheorem{example}[theorem]{Example}%

\numberwithin{equation}{section}%

\newcommand{\BB}{\mathbb{B}}%
\newcommand{\CC}{\mathbb{C}}%
\newcommand{\DD}{\mathbb{D}}%
\newcommand{\NN}{\mathbb{N}}%
\newcommand{\QQ}{\mathbb{Q}}%
\newcommand{\RR}{\mathbb{R}}%
\newcommand{\ZZ}{\mathbb{Z}}%

\newcommand{\calA}{\mathcal{A}}%
\newcommand{\calD}{\mathcal{D}}%
\newcommand{\calH}{\mathcal{H}}%
\newcommand{\calI}{\mathcal{I}}%
\newcommand{\calJ}{\mathcal{J}}%
\newcommand{\calO}{\mathcal{O}}%
\newcommand{\frakm}{\mathfrak{m}}%

\newcommand{\Id}{{\textup{Id}}}%
\newcommand{\one}{{\textup{\bf{1}}}}%
\newcommand{\supp}{{\textup{supp}\,}}%
\newcommand{\loc}{{\textup{loc}}}%
\newcommand{\psh}{{\textup{psh}}}%
\newcommand{\Nak}{{\textup{Nak}}}%
\newcommand{\reg}{{\textup{reg}}}%
\newcommand{\sing}{{\textup{sing}}}%

\newcommand{\inner}[1]{\langle#1\rangle}%
\newcommand{\ddbar}{\sqrt{-1}\partial\bar{\partial}}%
\newcommand{\dwdbar}[1]{\sqrt{-1}\partial#1\wedge\bar{\partial}#1}%
\newcommand{\DkS}{{D^k\backslash S}}%
\newcommand{\DkSn}{{D^k\backslash S_{k,\nu}}}%
\newcommand{\Bergman}{A^2(\Omega,K_\Omega\otimes E)}%
\newcommand{\BergmanN}{A^2(\widetilde{\Omega},K_{\widetilde{\Omega}}\otimes\widetilde{E})}%

\newcommand{\Vol}{\textup{Vol}}%
\newcommand{\re}{\operatorname{Re}}%
\newcommand{\pd}{\partial}%
\newcommand{\eps}{\varepsilon}%
\renewcommand{\geq}{\geqslant}%

\begin{document}

\title{Optimal $L^2$ Extensions of Openness Type}

\author{Wang Xu}
\address{School of Mathematical Sciences, Peking University, Beijing 100871, China}
\email{xuwang@amss.ac.cn}

\author{Xiangyu Zhou}
\address{Institute of Mathematics, AMSS, and Hua Loo-Keng Key Laboratory of Mathematics, Chinese Academy of Sciences, Beijing 100190, China}
\email{xyzhou@math.ac.cn}

\thanks{The second author is partially supported by National Key R\&D Program of China (No. 2021YFA1003100) and NSFC grants (No. 11688101 and No. 12288201).}

\begin{abstract}
	We study the following optimal $L^2$ extension problem of openness type: given a complex manifold $M$, a closed subvariety $S\subset M$ and a holomorphic vector bundle $E\rightarrow M$, for any $L^2$ holomorphic section $f$ defined on some open neighborhood $U$ of $S$, find an $L^2$ holomorphic section $F$ on $M$ such that $F|_S = f|_S$, and the $L^2$ norm of $F$ on $M$ is optimally controlled by the $L^2$ norm of $f$ on $U$.
	
	Answering the above problem, we prove an optimal $L^2$ extension theorem of openness type on weakly pseudoconvex K\"ahler manifolds, which generalizes a couple of known results on such a problem. Moreover, we prove a product property for certain minimal $L^2$ extensions and give an alternative proof to a version of the above $L^2$ extension theorem. We also present some applications to the usual optimal $L^2$ extension problem and the equality part of Suita's conjecture.
\end{abstract}

\subjclass[2020]{32D15, 32A36, 32U05, 14C30, 32W05}

\footnote{Due to length considerations, this article was divided into two parts during the submission. The first part (Sections 1--6) was \href{https://doi.org/10.1007/s00208-023-02774-9}{published online} by \textit{Mathematische Annalen} in January 2024, and the second part (Section 7) was accepted by \textit{Annales l'Institut Fourier} in April 2024.}

\maketitle

{\small\tableofcontents}

\section{Introduction}

Constructing holomorphic objects is a fundamental problem in complex analysis of several variables. In this article, we focus on the extension of holomorphic sections of vector bundles from subvarieties. The well-known Oka-Cartan theory guarantees the existence of holomorphic extensions on Stein manifolds.
In practice, we are interested in holomorphic extensions satisfying additional conditions, the one that attracts the most attention is the following $L^2$ extension problem (see \cite{DemaillyBook}):
\begin{quotation}
	\itshape Let $M$ be a complex manifold, $S\subset M$ be a closed submanifold and $E\to M$ be a Hermitian holomorphic vector bundle; suppose the triple $(M,S,E)$ satisfies some reasonable conditions; given a holomorphic section $f\in\Gamma(S,E|_S)$ satisfying suitable $L^2$ conditions, find a holomorphic extension $F\in\Gamma(M,E)$ of $f$, together with a good or optimal $L^2$ estimate for $F$ in terms of $f$.
\end{quotation}

After a series of important works due to H\"ormander, Bombieri, Skoda, Ohsawa-Takegoshi, Berndtsson, Demailly, Manivel, McNeal-Varolin, Ohasawa, Siu, Guan-Zhou-Zhu et al, one can get a satisfactory answer to the above problem and reach the optimal $L^2$ extension theorem (see \cite{Blocki2013, GuanZhou2012, GuanZhou2015}).

The Ohsawa-Takegoshi $L^2$ extension theorem \cite{OhsawaTakesoshi1987} and its various generalizations provide solutions to the existence part of the above problem.
These $L^2$ extension theorems have many important applications, e.g. boundary estimate of Bergman kernels, approximation of psh functions (Demailly \cite{Demailly1992}), invariance of plurigenera (Siu \cite{SiuPlurigenera}), strong openness property of multiplier ideal sheaves (Guan-Zhou \cite{GuanZhouSOC}).

The second part of the above problem is to find $L^2$ extensions with optimal estimates. Motivated by \cite{ZhuGuanZhou2012}, Blocki \cite{Blocki2013} proved an optimal $L^2$ extension theorem on pseudoconvex domains (in the setting of Ohsawa-Takegoshi \cite{OhsawaTakesoshi1987}), Guan-Zhou \cite{GuanZhou2012} proved an optimal $L^2$ extension theorem with negligible weight on Stein manifolds (in the setting of Ohsawa \cite{Ohsawa3}).
As a consequence, they solved the inequality part of Suita's conjecture \cite{Suita1972}. In \cite{GuanZhou2015}, a general optimal $L^2$ extension theorem and its geometric meaning were established. Consequently, the full version of Suita's conjecture (the inequality part and the equality part) was solved.
The geometric meaning of the optimal $L^2$ extension theorem plays a key role in Hacon-Popa-Schnell's recent work \cite{HPS} on the Iitaka conjecture.
Further developments on optimal $L^2$ extensions include \cite[etc]{BerndtssonLempert2016, Cao2017, ZhouZhu2018}. For more historical remarks and applications of (optimal) $L^2$ extension theorems, the reader is referred to \cite{OhsawaL2Book, OhsawaSurvey, Zhou:Abel}.\\

In this article, we study a closely related problem:
\begin{quotation}
	\itshape Let $M$ be a complex manifold, $S\subset M$ be a closed analytic subset, $E\to M$ be a Hermitian holomorphic vector bundle and $U$ be an open neighborhood of $S$; suppose $(M,S,E,U)$ satisfies some reasonable conditions; given an $L^2$ holomorphic section $f\in\Gamma(U,E)$, find an $L^2$ holomorphic section $F\in\Gamma(M,E)$ such that $F|_S = f|_S$ and the $L^2$ norm of $F$ on $M$ is uniformly or optimally controlled by the $L^2$ norm of $f$ on $U$.
\end{quotation}

To distinguish with the usual $L^2$ extension problem, we shall call it $L^2$ extension problem of \textbf{openness type}.
Such a problem is naturally raised in the literature, e.g. Jennane \cite{Jennane1978} and Demailly \cite{Demailly1982Ext,Demailly1982}.

For the convenience of readers, let us recall a couple of known results on this problem. We have the following $L^2$ extension theorem of openness type from Demailly's textbook (with some simplifications):

\begin{theorem}[{see \cite[Chapter VIII.7]{DemaillyCADG}}] \label{Thm:DemExt}
	Let $(\Omega,\omega)$ be a weakly pseudoconvex K\"ahler manifold and $(E,h)$ be a Hermitian holomorphic line bundle over $\Omega$ whose curvature is semi-positive. Let $w=(w_1,\ldots,w_p)$ be a tuple of holomorphic functions on $\Omega$ $(1\leqslant p\leqslant \dim\Omega)$. Let
	\begin{equation*}
		S = \{x\in\Omega:w(x)=0\} \quad\text{and}\quad U = \{x\in\Omega:|w(x)|<1\}.
	\end{equation*}
	Assume that $dw_1\wedge\cdots\wedge dw_p\neq0$ generically on $S$. For any holomorphic section $f\in\Gamma(U,K_\Omega\otimes E)$ satisfying $\int_U|f|_{\omega,h}^2dV_\omega<+\infty$ and any $\eps>0$, there exists a holomorphic section $F\in\Gamma(\Omega,K_\Omega\otimes E)$ such that $F|_S=f|_S$ and
	\begin{equation}
		\int_\Omega \frac{|F|_{\omega,h}^2}{(1+|w|^2)^{p+\eps}} dV_\omega
		\leqslant \left(1+\frac{(p+1)^2}{\eps}\right) \int_U |f|_{\omega,h}^2 dV_\omega.
	\end{equation}
\end{theorem}

Recall that, a complex manifold $\Omega$ is said to be \textit{weakly pseudoconvex} if there exists a smooth psh exhaustion function on $\Omega$. Theorem \ref{Thm:DemExt} contains the H\"ormander-Bombieri-Skoda theorem, which is an early result on $L^2$ extension problems:
let $\Omega\subset\CC^n$ be a pseudoconvex domain and $\varphi\in\psh(\Omega)$, if $e^{-\varphi}$ is integrable on $\BB^n(z_0;r)\subset\Omega$, then there exists a holomorphic function $f\in\calO(\Omega)$ so that $f(z_0)=1$ and $\int_\Omega |f(z)|^2e^{-\varphi} (r^2+|z|^2)^{-n-\eps} d\lambda < +\infty$.

The following theorem of Blocki \cite{Blocki2014} is a special \textit{optimal} $L^2$ extension theorem of openness type, in which the analytic subset is a single point.

\begin{theorem} \label{Thm:BlockiExt}
	Let $\Omega$ be a bounded pseudoconvex domain in $\CC^n$. Given $z_0\in\Omega$ and $a\in\RR_+$, let $U = \{z\in\Omega: G_\Omega(z,z_0)<-a\}$. Then for any $g\in A^2(U)$, one has
	\begin{equation} \label{Eq:Blocki1}
		\frac{|g(z_0)|^2}{B_\Omega(z_0)} \leqslant e^{2na} \int_U |g|^2 d\lambda.
	\end{equation}
	Equivalently, for any $f\in A^2(U)$, there exists an $F\in A^2(\Omega)$ such that
	\begin{equation}
		F(z_0)=f(z_0) \quad\text{and}\quad \int_\Omega |F|^2 d\lambda \leqslant e^{2na} \int_U |f|^2 d\lambda.
	\end{equation}
\end{theorem}

Here, $A^2$ denote the Bergman spaces, $B_\Omega$ denotes the Bergman kernel and $G_\Omega$ is the pluricomplex Green function. Let $g\equiv1$, then \eqref{Eq:Blocki1} implies a sharp lower bound for Bergman kernels.
If $n=1$, letting $a\to+\infty$, Blocki \cite{Blocki2014} obtained an alternative proof to the inequality part of Suita's conjecture.
We remark that Blocki-Pflug \cite{BlockiPflug1998} and Herbort \cite{Herbort1999} used a coarse version of Theorem \ref{Thm:BlockiExt} to prove that every bounded hyperconvex domain in $\CC^n$ is Bergman complete. \\

After \cite{ZhuGuanZhou2012}, \cite{Blocki2013} and \cite{GuanZhou2012}, there are many studies on optimal $L^2$ extension theorem and its applications. However, little is known about general optimal $L^2$ extension problems of openness type. The first main result of the present article is the following optimal $L^2$ extension theorem of openness type which answers the problem.

\begin{theorem} \label{MainThm:Opt}
	Let $(\Omega,\omega)$ be a weakly pseudoconvex K\"ahler manifold and $(E,h)$ be a Hermitian holomorphic vector bundle over $\Omega$. Let $\psi<0$ and $\varphi$ be quasi-psh functions on $\Omega$. Suppose there are continuous real $(1,1)$-forms $\gamma\geqslant0$ and $\rho$ on $\Omega$ such that
	\begin{equation*}
		\ddbar\psi\geqslant\gamma, \quad \ddbar\varphi\geqslant\rho \quad\text{and}\quad
		\sqrt{-1}\Theta(E,h)+(\gamma+\rho)\otimes\Id_E \geqslant_\Nak 0.
	\end{equation*}
	Given $a\in\RR_+$, let $\Omega_a:=\{z\in\Omega:\psi(z)<-a\}$. Then for any holomorphic section $f\in\Gamma(\Omega_a,K_\Omega\otimes E)$ satisfying $\int_{\Omega_a} |f|_{\omega,h}^2e^{-\varphi} dV_\omega < +\infty$, there exists a holomorphic section $F\in\Gamma(\Omega,K_\Omega\otimes E)$ such that
	\begin{equation} \label{Eq:EqualwrtCalI-intro}
		F|_{\Omega_a}-f\in\Gamma(\Omega_a, \calO(K_\Omega\otimes E)\otimes\calI(\varphi+\psi))
	\end{equation}
	and
	\begin{equation} \label{Eq:OptEstOpen-intro}
		\int_\Omega |F|_{\omega,h}^2e^{-\varphi} dV_\omega
		\leqslant e^a \int_{\Omega_a} |f|_{\omega,h}^2e^{-\varphi} dV_\omega.
	\end{equation}
\end{theorem}

Here, $\calI(\varphi+\psi)\subset\calO_\Omega$ is the multiplier ideal sheaf associated to $\varphi+\psi$. If there exists a closed analytic subset $S\subset\Omega_a$ such that $\calI(\varphi+\psi)_x\subset\frakm_x$ for any $x\in S$, where $\frakm_x$ denotes the unique maximal ideal of $\calO_x$, then the condition \eqref{Eq:EqualwrtCalI-intro} implies $F|_S = f|_S$. The use of $\calI(\varphi+\psi)$ allows us to deal with more general situations.

If $\Omega=\BB^n$ is the unit ball in $\CC^n$, $(E,h)$ is a trivial line bundle, $\psi=2n\log|z|$ and $\varphi\equiv0$, then \eqref{Eq:EqualwrtCalI-intro} is equivalent to $F(0)=f(0)$, and the constant $e^a$ in \eqref{Eq:OptEstOpen-intro} cannot be replaced by any smaller one.
Therefore, the uniform estimate in Theorem \ref{MainThm:Opt} is optimal. A more general example will be given in Example \ref{Ex:GenOptConst}.

In Theorem \ref{MainThm:Opt}, let $\Omega\Subset\CC^n$ and $\psi=2(n+m) G_\Omega(\cdot,z_0)$, then we obtain the following generalization to Blocki's theorem.

\begin{corollary} \label{MainCor:Opt}
	Let $\Omega\subset\CC^n$ be a bounded pseudoconvex domain and $\varphi$ be a psh function on $\Omega$. Given $z_0\in\Omega$, $m\in\NN$ and $a\in\RR_+$, let $U=\{G_\Omega(\cdot,z_0)<-a\}$. For any $f\in A^2(U;e^{-\varphi})$, there exists a holomorphic function $F\in A^2(\Omega;e^{-\varphi})$ such that $F$ coincides with $f$ up to order $m$ at $z_0$ and
	\begin{equation}
		\int_\Omega |F|^2e^{-\varphi} d\lambda \leqslant e^{2(n+m)a} \int_U |f|^2e^{-\varphi} d\lambda.
	\end{equation}
\end{corollary}

Here, ``$F$ coincides with $f$ up to order $m$ at $z_0$'' means all partial derivatives up to order $m$ of $F$ and $f$ are equal at $z_0$ (see Definition \ref{Def:mEqual}).\\

Let us briefly explain the proof of Theorem \ref{MainThm:Opt}.
At first, we modify the techniques developed by Guan-Zhou \cite{GuanZhou2015} and Zhou-Zhu \cite{ZhouZhu2018} to prove a coarse version of Theorem \ref{MainThm:Opt}, in which the uniform estimate is asymptotically optimal as $a\to+\infty$.
Then, by using a concavity of certain minimal $L^2$ integrals, which is essentially due to Guan \cite{Guan2019}, we obtain the optimal estimate.

Using similar arguments, we also prove an optimal $L^2$ extension theorem of openness type that has extra multiplying terms (see Theorem \ref{Thm:NewOptOpenExt}). In particular, we have the following generalization of Theorem \ref{Thm:DemExt}, whose uniform estimate is optimal.

\begin{theorem} \label{MainThm:DemExt}
	Let $(\Omega,\omega)$ be a weakly pseudoconvex K\"ahler manifold, $(E,h)$ be a Hermitian holomorphic vector bundle over $\Omega$ and $\varphi$ be a quasi-psh function on $\Omega$. Suppose there is a continuous real $(1,1)$-form $\rho$ on $\Omega$ so that
	\begin{equation*}
		\ddbar\varphi\geqslant\rho \quad\text{and}\quad \sqrt{-1}\Theta(E,h)+\rho\otimes\Id_E \geqslant_\Nak 0.
	\end{equation*}
	Let $w=(w_1,\ldots,w_p)$ be a tube of holomorphic functions on $\Omega$ $(1\leqslant p\leqslant \dim\Omega)$, let
	\begin{equation*}
		S = \{x\in\Omega:w(x)=0\} \quad\text{and}\quad U = \{x\in\Omega:|w(x)|<1\}.
	\end{equation*}
	Assume that $dw_1\wedge\cdots\wedge dw_p\neq0$ generically on $S$. Given $m\in\NN$ and $\eps\in(0,m+p)$, for any holomorphic section $f\in\Gamma(U,K_\Omega\otimes E)$ satisfying $\int_U|f|_{\omega,h}^2e^{-\varphi}dV_\omega<+\infty$, there exists a holomorphic section $F\in\Gamma(\Omega,K_\Omega\otimes E)$ such that $F$ coincides with $f$ up to order $m$ along $S$ and
	\begin{equation*}
		\int_\Omega \frac{|F|_{\omega,h}^2e^{-\varphi}}{(1+|w|^2)^{m+p+\eps}} dV_\omega \leqslant
		\frac{\int_0^1\tau^{m+p-1}(1-\tau)^{\eps-1}d\tau}{\int_0^{1/2}\tau^{m+p-1}(1-\tau)^{\eps-1}d\tau}
		\int_U \frac{|f|_{\omega,h}^2e^{-\varphi}}{(1+|w|^2)^{m+p+\eps}} dV_\omega.
	\end{equation*}
\end{theorem}

In the proof of Theorem \ref{MainThm:Opt}, we need a concavity for certain minimal $L^2$ integrals. As a by-product of the proof, we obtain a necessary condition for the above concavity degenerating to linearity (see Remark \ref{Rmk:ConcaveLinear}).
We can use such a necessary condition to prove the equality part of Suita's conjecture, this approach is partially different from Guan-Zhou \cite{GuanZhou2015} and Dong \cite{DongArxiv} (see Section \ref{Sec:Suita}).\\

Blocki \cite{Blocki2014} obtained the optimal estimate in Theorem \ref{Thm:BlockiExt} by using a tensor power trick that relies on the product property of Bergman kernels. Notice that, for open set $\Omega\Subset\CC^n$ and $z_0\in\Omega$, one has
\begin{equation*}
	\tfrac{1}{B_\Omega(z_0)} = \inf\left\{ \|f\|_{A^2(\Omega)}^2: f\in A^2(\Omega),f(z_0)=1 \right\},
\end{equation*}
i.e. the Bergman kernel is the reciprocal of the norm of certain minimal $L^2$ extension. As another main result, we prove a product property for general minimal $L^2$ extensions, which generalizes the product property of Bergman kernels.

Let $\Omega_1$ and $\Omega_2$ be complex manifolds, $E_1\rightarrow\Omega_1$ and $E_2\rightarrow\Omega_2$ be holomorphic vector bundles, $S_1\subset\Omega_1$ and $S_2\subset\Omega_2$ be arbitrary closed subsets.
Let $\Omega:=\Omega_1\times\Omega_2$ be the product manifold, and let $p_i:\Omega\rightarrow\Omega_i$ be the natural projections, then $E:=p_1^*E_1\otimes p_2^*E_2$ is a holomorphic vector bundle over $\Omega$.
For $i=1$ and $2$, we choose a continuous volume form $dV_i$ on $\Omega_i$, a continuous Hermitian metric $h_i$ on $E_i$, and a measurable function $\psi_i$ on $\Omega_i$ which is locally bounded from above.
We define $dV:= p_1^*dV_1\times p_2^*dV_2$, $h:=p_1^*h_1\otimes p_2^*h_2$ and $\psi:=p_1^*\psi_1+p_2^*\psi_2$. We shall consider the following Bergman spaces:
\begin{equation*}
	A^2(\Omega_i,E_i) := A^2(\Omega_i,E_i;h_i,e^{-\psi_i}dV_i), \quad A^2(\Omega,E) := A^2(\Omega,E;h,e^{-\psi}dV).
\end{equation*}

\begin{theorem} \label{MainThm:Prod}
	Let $f^{(1)}\in A^2(\Omega_1,E_1)$, $f^{(2)}\in A^2(\Omega_2,E_2)$ and $m_1,m_2\in\NN$ be given. For $i=1$ and $2$, let $F^{(i)}$ be the unique element with minimal norm in $A^2(\Omega_i,E_i)$ that coincides with $f^{(i)}$ up order $m_i$ on $S_i$. Let $F$ be the unique element with minimal norm in $A^2(\Omega,E)$ that coincides with $f:=f^{(1)}\otimes f^{(2)}$ up to order $m_1+m_2$ on $S:=S_1\times S_2$. Then
	\begin{equation*}
		\|F\|_{A^2(\Omega,E)} \geqslant \|F^{(1)}\|_{A^2(\Omega_1,E_1)} \|F^{(2)}\|_{A^2(\Omega_2,E_2)}.
	\end{equation*}
	Moreover, if $f^{(i)}$ vanishes up to order $m_i-1$ on $S_i$, then
	\begin{equation*}
		F=F^{(1)}\otimes F^{(2)} \quad\text{and}\quad
		\|F\|_{A^2(\Omega,E)} = \|F^{(1)}\|_{A^2(\Omega_1,E_1)} \|F^{(2)}\|_{A^2(\Omega_2,E_2)}.
	\end{equation*}
\end{theorem}

If $S_i\subset\Omega_i$ are singleton sets and $m_1=m_2=0$, then the last statement corresponds to the product property for Bergman kernels. Imitating the proof of \cite{Blocki2014} and using Theorem \ref{MainThm:Prod}, we can prove a version of Theorem \ref{MainThm:Opt} (see Theorem \ref{Thm:OptExtWeakForm}).\\

The $L^2$ extension problem of openness type is closely related to the usual one. Roughly speaking, the limiting cases of $L^2$ extension theorems of openness type are $L^2$ extension theorems of Ohsawa-Takegoshi type.
In Theorem \ref{MainThm:Opt}, if $\psi$ has suitable singularities along the pole $S:=\{\psi=-\infty\}$, then the right hand side of \eqref{Eq:OptEstOpen-intro} will converge to an integral on $S$ as $a\to+\infty$.
In this way, we obtain an optimal $L^2$ extension theorem in the setting of Guan-Zhou \cite{GuanZhou2015} (see Theorem \ref{Thm:NewMesExt}). The same approach is applicable to certain jet $L^2$ extensions (see Theorem \ref{Thm:OptJetExt}).

Nevertheless, optimal $L^2$ extension theorems of openness type imply something more. As an application of Theorem \ref{MainThm:Opt}, we prove the following result.

\begin{theorem} \label{MainThm:Sharper}
	Let $\Omega$ be a weakly pseudoconvex K\"ahler manifold of dimension $n$, let $(E,h)$ be a Hermitian holomorphic vector bundle over $\Omega$ whose curvature is Nakano semi-positive. Suppose there exists a psh function $\psi:\Omega\rightarrow[-\infty,0)$ with a logarithmic pole at $w\in\Omega$. Moreover, we assume that:
	
	(1) $(E,h)$ is Griffiths positive at $w$; 
	(2) there exists a coordinate chart $(U,z)$ such that $z(w)=0$ and $\psi(z)=c+\log|z|+o(|z|^2)$ as $z\to0$.
	
	\noindent Then there exists a constant $\tau\in(0,1)$ depends on $h$ and $\psi$, for any $\xi\in E_w$, we can find a holomorphic section $F\in\Gamma(\Omega,K_\Omega\otimes E)$ so that $F(w)=dz_1\wedge\cdots\wedge dz_n\otimes\xi$ and
	\begin{equation}
		\int_\Omega(\sqrt{-1})^{n^2}F\wedge_h\overline{F} \leqslant (1-\tau)\frac{(2\pi)^n}{n!}e^{-2nc}|\xi|_h^2.
	\end{equation}
\end{theorem}

According to the optimal $L^2$ extension theorem (see \cite{GuanZhou2015,ZhouZhu2018}), without assuming the conditions (1) and (2) in Theorem \ref{MainThm:Sharper}, we can find a holomorphic section $F'\in\Gamma(\Omega,K_\Omega\otimes E)$ such that $F'(w)=dz_1\wedge\cdots\wedge dz_n\otimes\xi$ and
\begin{equation} \label{Eq:OptEst0}
	\int_\Omega (\sqrt{-1})^{n^2} F'\wedge_h\overline{F'} \leqslant \frac{(2\pi)^n}{n!}e^{-2nc}|\xi|_h^2,
\end{equation}
where $c:=\varliminf_{z\rightarrow0}(\psi(z)-\log|z|)$. Hence, Theorem \ref{MainThm:Sharper} gives a couple of additional conditions so that the estimate in \eqref{Eq:OptEst0} can be improved.
Notice that, Hosono's \cite{Hosono2019} result is a special case of the above theorem (see Remark \ref{Rmk:Hosono}).

If $\Omega$ is an open Riemann surface admitting Green function and $\psi=G_\Omega(\cdot,w)$, then the condition (2) is automatically satisfied.
We should point out that, if $(E,h)$ is Nakano positive at some point $x\in\Omega$, then without assuming (1) and (2), we also have sharper estimates (see Remark \ref{Rmk:NakSharper}).
However, ``Nakano positivity'' is stronger than ``Griffiths positivity'', and we have an explicit estimate of $\tau$ in Theorem \ref{MainThm:Sharper}.\\

The rest parts of the article are organized as follows.
In Section \ref{Chap:Pre}, we recall some notations and preparatory results.
In Section \ref{Chap:Coarse}, we prove an asymptotically optimal version of Theorem \ref{MainThm:Opt}.
In Section \ref{Chap:Prod}, we study the Bergman spaces on product manifolds and prove Theorem \ref{MainThm:Prod}.
In Section \ref{Chap:Optimal}, we prove Theorem \ref{MainThm:Opt} by using a concavity for minimal $L^2$ integrals; we also prove a variant of Theorem \ref{MainThm:Opt} by using a tensor power trick; as applications, we prove Theorem \ref{MainCor:Opt} and \ref{MainThm:Sharper}.
In Section \ref{Chap:OTtype}, we discuss the relations between $L^2$ extension problems of openness type and that of Ohsawa-Takegoshi type, we prove Theorem \ref{MainThm:DemExt} and some optimal (jet) $L^2$ extension theorems. 
Some material of this paper comes from the first author's PhD thesis \cite{XuPHD}.

\section{Preliminaries} \label{Chap:Pre}

In this section, we recall some notations and preparatory results.

Let $\Omega\subset\CC^n$ be an open set, let $\varphi$ be a measurable function on $\Omega$ which is locally bounded from above. We assume that $\{z\in\Omega:\varphi(z)=-\infty\}$ is a set of zero measure. Then the weighted \textbf{Bergman space} of $\Omega$ is the Hilbert space defined by
\begin{equation*}
	A^2(\Omega;e^{-\varphi}) := \left\{f\in\calO(\Omega): \|f\|^2=\int_\Omega |f|^2e^{-\varphi} d\lambda <+\infty \right\},
\end{equation*}
and the \textbf{Bergman kernel} $B_\Omega(z,w;e^{-\varphi})$ is the reproducing kernel of $A^2(\Omega;e^{-\varphi})$.
The diagonal Bergman kernel satisfies the following extremal property:
\begin{equation*}
	B_\Omega(w;e^{-\varphi}) := B_\Omega(w,w;e^{-\varphi}) = \sup\{|f(w)|^2:f\in A^2(\Omega;e^{-\varphi}),\|f\|\leqslant1\}.
\end{equation*}
If $\varphi\equiv0$, we shall simplify these notations as $A^2(\Omega)$, $B_\Omega(z,w)$ and $B_\Omega(w)$.

We also study the space of $L^2$ holomorphic sections. Let $M$ be a complex manifold and $dV$ be a continuous volume form on $M$. Let $E$ be a holomorphic vector bundle over $M$ and $h$ be a continuous Hermitian metric on $E$.
Let $\psi$ be a measurable function on $M$ which is locally bounded from above and $\{\psi=-\infty\}$ is a set of zero measure. Then the Bergman space of $E\rightarrow M$ is the Hilbert space defined by
\begin{equation*}
	A^2(M,E;h,e^{-\psi}dV) := \left\{f\in \Gamma(M,E): \|f\|^2 = \int_M |f|_h^2 e^{-\psi}dV <+\infty \right\}.
\end{equation*}
If there are no confusion, we simply denote it by $A^2(M,E)$.

The following generalized Montel's theorem will be used repeatedly.

\begin{theorem}[Montel's Theorem]
	Let $E$ be a holomorphic vector bundle over a complex manifold $M$. Let $h$ be a continuous Hermitian metric on $E$ and $dV$ be a continuous volume form on $M$.
	Let $\{U_j\}_{j=1}^\infty$ be a sequence of open subsets of $M$ and $f_j\in\Gamma(U_j,E)$ be holomorphic sections.
	Assume that for any relatively compact open set $D\Subset M$, there exists $j_D\in\NN_+$ and $C_D\in\RR_+$ so that $D\subset U_j$ and $\int_D |f_j|_h^2 dV \leqslant C_D$ for any $j\geqslant j_D$.
	Then there exists a subsequence of $\{f_j\}_{j=1}^\infty$ that converges uniformly on any compact subsets of $M$ to some $f\in\Gamma(M,E)$.
\end{theorem}

Let $\Omega$ be a domain in $\CC^n$, then the \textbf{pluricomplex Green function} of $\Omega$ with a pole at $w\in\Omega$ is defined as
\begin{equation*}
	G_\Omega(z,w) := \sup\{u(z): u\in\psh^-(\Omega), u(z)\leqslant\log|z-w|+O(1) \text{ as } z\rightarrow w \}.
\end{equation*}
If $G_\Omega(\cdot,w)\not\equiv-\infty$ (e.g. $\Omega$ is bounded), then $G_\Omega(\cdot,w)$ is a negative psh function on $\Omega$ having a logarithmic pole at $w$.

An upper semi-continuous function $\varphi:M\rightarrow[-\infty,+\infty)$ on a complex manifold $M$ is said to be \textbf{quasi-psh} if $\varphi$ is locally the sum of a psh function and a smooth function.
If $\varphi$ is quasi-psh, then $\varphi\in L^1_\loc(M)$ and there exists a continuous real $(1,1)$-form $\gamma$ on $M$ such that $\ddbar\varphi\geqslant\gamma$ in the sense of currents.

Let $\varphi$ be a quasi-psh function on $M$, we denote by $\calI(\varphi)\subset\calO_M$ the \textbf{multiplier ideal sheaf} associated to $\varphi$, i.e.
\begin{equation*}
	\calI(\varphi)_x := \left\{ f_x\in\calO_x: \int_U|f_x|^2e^{-\varphi}dV < +\infty \text{ for some nbhd }U\text{ of }x \right\}.
\end{equation*}
Notice that, if $\varphi(x)>-\infty$, then $\calI(\varphi)_x=\calO_x$. Nadel \cite{Nadel1990} showed that $\calI(\varphi)$ is a coherent analytic sheaf, and Guan-Zhou \cite{GuanZhouSOC} proved that $\calI(\varphi)$ satisfies the strong openness property, i.e. $\calI(\varphi) = \cup_{\eps>0}\calI((1+\eps)\varphi)$.

\begin{lemma}[see \cite{GrauertRemmertCoherent}] \label{Lemma:Coherent}
	Assume that $\{f_j\}_{j=1}^\infty$ is a sequence of holomorphic functions in $\Gamma(M,\calI(\varphi))$ that converges uniformly on any compact subsets of $M$ to some $f\in\calO(M)$, then $f\in\Gamma(M,\calI(\varphi))$.
\end{lemma}

Usually, we need to approximate a quasi-psh function by smooth functions.

\begin{theorem} \label{Thm:SteinAppro}
	Let $\varphi$ be a quasi-psh function on a Stein manifold $M$ and $\gamma$ be a continuous real $(1,1)$-form on $M$ such that $\ddbar\varphi\geqslant\gamma$. Then there exists a decreasing sequence $\{\varphi_k\}_{k=1}^\infty$ of smooth functions such that $\lim_{k\to+\infty}\varphi_k=\varphi$ and $\ddbar\varphi_k\geqslant\gamma$ for any $k$.
\end{theorem}

\begin{proof}
	Since $M$ is Stein, we may assume that $M$ is a closed submanifold of $\CC^n$. Following the proof of Sadullaev \cite{Sadullaev}, there exists a quasi-psh function $\tilde{\varphi}$ and a continuous real $(1,1)$-forms $\tilde{\gamma}$ on $\CC^n$ so that $\tilde{\varphi}|_M=\varphi$, $\tilde{\gamma}|_M=\gamma$ and $\ddbar\tilde{\varphi}\geqslant\tilde{\gamma}$. Clearly, it is sufficient to approximate $\tilde{\varphi}$ with similar requirements.
	
	Replacing $\tilde{\varphi}$ by $\tilde{\varphi}+\eta(|z|^2)$, where $\eta\in C^\infty(\RR)$ is a rapidly increasing convex function, we may assume that $\tilde{\gamma}\geqslant0$, and then $\tilde{\varphi}\in\psh(\CC^n)$.
	Let $\omega:=\ddbar|z|^2$ and $B_k:=\BB^n(0;k)$. We choose another rapidly increasing convex function $\chi\in C^\infty(\RR)$ such that $\chi>0$, $\chi'>1$ and
	\begin{equation*}
		\ddbar\chi(|z|^2)\geqslant k\tilde{\gamma} \quad\text{on}\quad \overline{B_{k+1}}\setminus B_k.
	\end{equation*}
	Let $\{\rho_\eps\}_\eps$ be a family of smoothing kernel, then $\tilde{\varphi}*\rho_\eps \in C^\infty(\CC^n)\cap\psh(\CC^n)$ and $\tilde{\varphi}*\rho_\eps\searrow\tilde{\varphi}$ as $\eps\searrow0$. Clearly, $\ddbar(\tilde{\varphi}*\rho_\eps) \geqslant \tilde{\gamma}*\rho_\eps$. By the continuity of $\tilde{\gamma}$, we can choose $0<\eps_k\ll1$ so that
	\begin{equation*}
		\ddbar\big( \tilde{\varphi}*\rho_{\eps_k} + k^{-1}\chi(|z|^2) \big) \geqslant \tilde{\gamma}*\rho_{\eps_k} + k^{-1}\omega \geqslant \tilde{\gamma} \quad\text{on}\quad \overline{B_k}.
	\end{equation*}
	Since $\tilde{\varphi}*\rho_{\eps_k}$ is psh and $\tilde{\gamma}\geqslant0$, we know
	\begin{equation*}
		\ddbar\big( \tilde{\varphi}*\rho_{\eps_k} + k^{-1}\chi(|z|^2) \big) \geqslant \tilde{\gamma} \quad\text{on}\quad \CC^n\setminus B_k.
	\end{equation*}
	We may require that $\eps_k\searrow0$ as $k\nearrow+\infty$.
	Then $\tilde{\varphi}_k:=\tilde{\varphi}*\rho_{\eps_k} + k^{-1}\chi(|z|^2)$ is a decreasing sequence of smooth functions on $\CC^n$ such that $\lim_{k\to+\infty}\tilde{\varphi}_k=\tilde{\varphi}$ and $\ddbar\tilde{\varphi}_k \geqslant \tilde{\gamma}$.
\end{proof}

\begin{theorem}[{see \cite[Theorem 2.3]{DPS2001}}] \label{Thm:EquiSingAppro}
	Let $\varphi$ be a quasi-psh function on a Hermitian manifold $(M,\omega)$. Let $\gamma$ be a continuous real $(1,1)$-form on $M$ so that $\ddbar\varphi\geqslant\gamma$ and $\eta$ be a continuous function on $M$ so that $\eta>\varphi$.
	Then for any relatively compact open set $D\Subset M$, there exists a decreasing sequence $\{\varphi_k\}_{k=1}^\infty$ of quasi-psh functions on $D$ satisfying the following properties: \par
	(1) $\varphi_k<\eta$ on $D$ for $k\gg1$ and $\lim_{k\rightarrow+\infty}\varphi_k=\varphi$; \par
	(2) $\varphi_k\in C^\infty(D\backslash S_k)$ and $S_k:=\varphi_k^{-1}(-\infty)$ is a closed analytic subset of $D$; \par
	(3) $\ddbar\varphi_k\geqslant \gamma-\delta_k\omega$, in which $\delta_k$ are positive constants so that $\delta_k\to0$; \par
	(4) (equisingular) for every $t>0$, $\mathcal{I}(t\varphi)|_D=\mathcal{I}(t\varphi_k)$ for $k\gg1$.
\end{theorem}

The statement of the theorem is slightly different from \cite{DPS2001}, but the proof is the same: all constructions are applicable to non-compact manifolds, and we still have uniform estimates on any relatively compact domain; for any $x\in M$, there exists an integer $k_x\gg1$ such that $\varphi_{k_x}(x)<\eta(x)$, and then $\varphi_{k_x} < \eta$ on some nbhd $U_x\ni x$; since $\overline{D}$ is compact, $\varphi_k<\eta$ on $D$ for $k\gg1$.\\

Let $M$ be a complex manifold and let $(E,h)$ be a Hermitian holomorphic vector bundle over $M$. Let $(U,z=(z_1,\ldots,z_n))$ be a coordinate chart of $M$ so that $E|_U$ is trivial, let $e=(e_1,\ldots,e_r)$ be a holomorphic frame of $E|_U$. We shall call such triple $(U,z,e)$ a \textbf{local trivialization} of $E\rightarrow M$.
Locally, the Chern curvature of $(E,h)$ can be written as
\begin{equation*}
	\Theta(E,h) = \sum\nolimits_{i,j,\alpha,\gamma} R_{i\bar{j}\alpha}^\gamma dz_i\wedge d\bar{z}_j \otimes e_\alpha^* \otimes e_\gamma.
\end{equation*}
We define $R_{i\bar{j}\alpha\bar{\beta}} := \sum_\gamma R_{i\bar{j}\alpha}^\gamma h_{\gamma\bar{\beta}}$, where $h_{\gamma\bar{\beta}} := \inner{e_\gamma,e_\beta}_h$.

\begin{definition}
	Positivity of Hermitian holomorphic vector bundles.
	
	(1) $\Theta(E,h)$ is said to be \textbf{Griffiths positive} (resp. semi-positive) at $x\in M$ if
	\begin{equation*}
		\sum\nolimits_{i,j,\alpha,\beta} R_{i\bar{j}\alpha\bar{\beta}}(x) a_i\overline{a_j}b_\alpha\overline{b_\beta} > 0 \text{ (resp. $\geqslant0$)}, \quad a\in\CC^n\setminus\{0\}, b\in\CC^r\setminus\{0\}.
	\end{equation*}
	
	(2) $\Theta(E,h)$ is said to be \textbf{Nakano positive} (resp. semi-positive) at $x\in M$ if
	\begin{equation*}
		\sum\nolimits_{i,j,\alpha,\beta} R_{i\bar{j}\alpha\bar{\beta}}(x) u_{i\alpha} \overline{u_{j\beta}} > 0 \text{ (resp. $\geqslant0$)}, \quad u\in\CC^{nr}\setminus\{0\}.
	\end{equation*}
	
	For simplicity, we write $>_\Nak$ (resp. $\geqslant_\Nak$) for Nakano (semi-)positivity. These notions of positivity naturally extend to other tensor having similar symmetry.
\end{definition}

Let $(M,\omega)$ be a Hermitian manifold and $(E,h)$ be a Hermitian holomorphic vector bundle over $M$.  Let $|\cdot|_{\omega,h}$ and $\inner{\cdot,\cdot}_{\omega,h}$ denote the norm and the inner product on $\wedge^{p,q}T^*M\otimes E$ that induced by $\omega$ and $h$. Let $L_{p,q}^2(M,E;\loc)$ be the space of all $E$-valued $(p,q)$-forms on $M$ whose coefficients are $L_\loc^2$ functions. Given a continuous volume form $dV$ on $M$, we define
\begin{equation*}
	L_{p,q}^2(M,E;\omega,h,dV) := \left\{f\in L_{p,q}^2(M,E;\loc): \int_M|f|_{\omega,h}^2dV <+\infty \right\}.
\end{equation*}
For simplicity, $L_{p,q}^2(M,E;\omega,h) := L_{p,q}^2(M,E;\omega,h,dV_\omega)$, where $dV_\omega:={\omega^n}/{n!}$ is the volume form induced by $\omega$.

\begin{lemma} \label{Lemma:WeakLim}
	Let $\{f^\nu\}_{\nu=1}^\infty$ be a sequence in $L_{p,q}^2(M,E;\omega,h,dV)$ that converges weakly to some $f$, and $\{\eta^\nu\}_{\nu=1}^\infty$ be a sequence of measurable functions on $M$ that converges pointwisely to some function $\eta$. Assume that there is a constant $C$ so that $|\eta^\nu|\leqslant C$ for all $\nu$. Then $\{\eta^\nu f^\nu\}_\nu$ converges weakly to $\eta f$ in $L_{p,q}^2(M,E;\omega,h,dV)$.
\end{lemma}

\begin{proof}
	Let $g\in L_{p,q}^2(M,E;\omega,h,dV)$ be given, then
	\begin{align*}
		|(\eta^\nu f^\nu,g) - (\eta f,g)| & \leqslant |(\eta^\nu f^\nu - \eta f^\nu,g)| + |(\eta f^\nu - \eta f,g)| \\
		& \leqslant \|f^\nu\| \|\overline{\eta^\nu}g-\overline{\eta}g\| + |(f^\nu-f,\overline{\eta}g)|.
	\end{align*}
	By Lebesgue's dominated convergence theorem, $\|\overline{\eta^\nu}g-\overline{\eta}g\| \rightarrow 0$. As $f^\nu\rightharpoonup f$, we know that $\|f^\nu\|$ is bounded and $|(f^\nu-f,\overline{\eta}g)| \rightarrow 0$. Let $\nu\to+\infty$, we have
	\begin{equation*}
		\lim_{\nu\rightarrow+\infty} |(\eta^\nu f^\nu,g) - (\eta f,g)| = 0.
	\end{equation*}
	Therefore, $\{\eta^\nu f^\nu\}_\nu$ converges weakly to $\eta f$ in $L_{p,q}^2(M,E;\omega,h,dV)$.
\end{proof}

Sometimes, we will encounter different metrics and different volume forms. Let $\tilde{\omega}$ be another Hermitian metric on $M$, $\tilde{h}$ be another Hermitian metric on $E$ and $\widetilde{dV}$ be another continuous volume form on $M$. For simplicity,
\begin{equation*}
	\calH := L_{p,q}^2(M,E;\omega,h,dV) \quad\text{and}\quad
	\tilde{\calH} := L_{p,q}^2(M,E;\tilde{\omega},\tilde{h},\widetilde{dV}).
\end{equation*}

\begin{lemma} \label{Lemma:DoubleWeakLim}
	Let $\{f^\nu\}_{\nu=1}^\infty$ be a sequence of differential forms in $\calH\cap\tilde{\calH}$. Assume that $\{f^\nu\}_\nu$ converges weakly to some $f$ in $\calH$ and $\|f^\nu\|_{\tilde{\calH}}$ is uniformly bounded. Then $f\in\calH\cap\tilde{\calH}$ and $\{f^\nu\}_\nu$ also converges weakly to $f$ in $\tilde{\calH}$.
\end{lemma}

\begin{proof}
	By Alaoglu's theorem, there exists a subsequence $\{f^{\nu_k}\}_k$ of $\{f^\nu\}_\nu$ that converges weakly to some $g$ in $\tilde{\calH}$. Since $f^{\nu_k}\rightharpoonup f$ in $\calH$ and $f^{\nu_k}\rightharpoonup g$ in $\tilde{\calH}$, it is easy to show that $f=g\in L_{p,q}^2(M,E;\loc)$.
	For any compactly supported continuous form $u\in\mathcal{C}_{p,q}(M,E)$, it is possible to find a $w\in\mathcal{C}_{p,q}(M,E)$ so that $(\xi,u)_{\tilde{\calH}} = (\xi,w)_{\calH}$ for all $\xi\in\calH\cap\tilde{\calH}$.
	Since $f^\nu\rightharpoonup f$ in $\calH$, we conclude that
	\begin{equation*}
		\lim_{\nu\rightarrow+\infty} (f^\nu, u)_{\tilde{\calH}} = (f, u)_{\tilde{\calH}}, \quad
		\forall u\in\mathcal{C}_{p,q}(M,E).
	\end{equation*}
	Via approximation, the above equation holds for all $u\in\tilde{\calH}$. Hence, $f^\nu\rightharpoonup f$ in $\tilde{\calH}$.
\end{proof}

$L^2$ estimates for the $\bar{\partial}$-operator provide powerful tools for solving many important problems in several complex variables. In this article, we need the following $L^2$ existence theorem for $\bar{\partial}$-equations.

\begin{theorem}[{see \cite[Remark 12.5]{DemaillyBook}}] \label{Thm:L2ExistenceErrTerm}
	Let $M$ be a complete K\"ahler manifold of dimension $n$ and let $(E,h)$ be a Hermitian holomorphic vector bundle over $M$. Let $\omega$ be an arbitrary K\"ahler metric on $M$, let $\eta>0$ and $\lambda>0$ be bounded smooth functions on $M$. Define
	\begin{equation*}
		\Theta := \eta\sqrt{-1}\Theta(E,h) -\ddbar\eta \otimes\Id_E -\lambda^{-1}\dwdbar{\eta} \otimes\Id_E.
	\end{equation*}
	Assume there is a constant $\delta\geqslant0$ so that $\Theta+\delta\omega\otimes\Id_E$ is Nakano semi-positive. Given a differential form $g\in L_{n,q}^2(M,E;\omega,h)$ $(1\leqslant q\leqslant n)$ such that $\bar{\pd}g=0$, if
	\begin{equation*}
		\int_M\inner{[\Theta+\delta\omega\otimes\Id_E,\Lambda_\omega]^{-1}g,g}_{\omega,h}dV_\omega<+\infty,
	\end{equation*}
	then there exists an approximate solution $f\in L_{n,q-1}^2(M,E;\omega,h)$ and a correcting term $w\in L_{n,q}^2(M,E;\omega,h)$ so that $\bar{\pd}f+\sqrt{\delta q}\,w=g$ in the sense of distribution and
	\begin{align*}
		&\, \int_M \frac{1}{\eta+\lambda}|f|_{\omega,h}^2dV_\omega + \int_M|w|_{\omega,h}^2dV_\omega \\
		\leqslant &\, \int_M\inner{[\Theta+\delta\omega\otimes\Id_E,\Lambda_\omega]^{-1}g,g}_{\omega,h}dV_\omega.
	\end{align*}
\end{theorem}

Here, $\Lambda_\omega$ is the adjoint operator of the Lefschetz operator $\xi\mapsto\xi\wedge\omega$ and $[\cdot,\cdot]$ is the graded commutator. In Theorem \ref{Thm:L2ExistenceErrTerm}, $[\Theta+\delta\omega\otimes\Id_E,\Lambda_\omega]$ is only a semi-positive Hermitian operator on $(\wedge^{n,q}T^*M\otimes E)_x$, and it may not be invertible. Hence,
\begin{equation*}
	\inner{[\Theta+\delta\omega\otimes\Id_E, \Lambda_\omega]^{-1}g,g}_{\omega,h}
\end{equation*}
is defined as the minimal constant $C\in[0,+\infty]$ such that
\begin{equation*}
	|\inner{g,u}_{\omega,h}|^2 \leqslant C\inner{[\Theta+\delta\omega\otimes\Id_E,\Lambda_\omega]u,u}_{\omega,h},
	\quad \forall u\in(\wedge^{n,q}T^*M\otimes E)_x.
\end{equation*}

The following results are useful for solving $L^2$ extension problems.

\begin{theorem}[{see \cite[Th\'eor\`eme 1.3 and 1.5]{Demailly1982}}] \label{Thm:CompleteKahler}
	(1) Every weakly pseudoconvex K\"ahler manifold admits a complete K\"ahler metric.
	(2) Let $M$ be a K\"ahler manifold and let $S$ be a closed analytic subset of $M$. If $D\Subset M$ admits a complete K\"ahler metric, then $D\backslash S$ also supports a complete K\"ahler metric.
\end{theorem}

\begin{lemma}[{see \cite[Lemme 6.9]{Demailly1982}}] \label{Lemma:dbarOverS}
	Let $D\subset\CC^n$ be an open set and $S$ be a closed analytic subset of $D$. Assume that $f$ is a $(p,q-1)$-form on $D$ with $L_\loc^2$ coefficients and $g$ is a $(p,q)$-form on $D$ with $L_\loc^1$ coefficients. If $\bar{\pd}f=g$ on $D\backslash S$ in the sense of distribution, then $\bar{\pd}f=g$ on $D$ in the sense of distribution.
\end{lemma}

\begin{lemma} \label{Lemma:AnIneq}
	Let $(E,h)$ be a Hermitian vector bundle over a Hermitian manifold $(M,\omega)$. Given $\xi\in (T^*M)_x$, $f\in(K_M\otimes E)_x$ and $g\in(\wedge^{n,1}T^*M\otimes E)_x$, we have
	\begin{equation*}
		|\inner{g,\overline{\xi}\wedge f}_{\omega,h}|^2 \leqslant |f|_{\omega,h}^2 \inner{[\sqrt{-1}\xi\wedge\overline{\xi},\Lambda_\omega]g,g}_{\omega,h}.
	\end{equation*}
\end{lemma}

\begin{proof}
	Direct computations in terms of orthonormal bases of $(T^*M)_x$ and $E_x$.
\end{proof}

\section{An Almost Optimal \texorpdfstring{$L^2$}{L2} Extension Theorem of Openness Type} \label{Chap:Coarse}

In this section, we prove an $L^2$ extension theorem of openness type on weakly pseudoconvex K\"ahler manifolds, and the $L^2$ estimate obtained here is asymptotically optimal. The main ideas of the proof come from Guan-Zhou \cite{GuanZhou2015} and Zhou-Zhu \cite{ZhouZhu2018}, but we choose different auxiliary functions for our purpose.

\begin{theorem} \label{Thm:CoarseExtThm}
	Let $(\Omega,\omega)$ be a weakly pseudoconvex K\"ahler manifold and $(E,h)$ be a Hermitian holomorphic vector bundle over $\Omega$. Let $\psi<0$ and $\varphi$ be quasi-psh functions on $\Omega$. Suppose there are continuous real $(1,1)$-forms $\gamma\geqslant0$ and $\rho$ on $\Omega$ such that
	\begin{equation*}
		\ddbar\psi\geqslant\gamma, \quad \ddbar\varphi\geqslant\rho \quad\text{and}\quad \sqrt{-1}\Theta(E,h)+(\gamma+\rho)\otimes\Id_E \geqslant_\Nak 0.
	\end{equation*}
	Given $a\in\RR_+$, let $\Omega_a:=\{z\in\Omega:\psi(z)<-a\}$. Then for any holomorphic section $f\in\Gamma(\Omega_a,K_\Omega\otimes E)$ satisfying $\int_{\Omega_a} |f|_{\omega,h}^2e^{-\varphi} dV_\omega < +\infty$, there exists a holomorphic section $F\in\Gamma(\Omega,K_\Omega\otimes E)$ such that
	\begin{equation*}
		F|_{\Omega_a}-f\in\Gamma(\Omega_a, \calO(K_\Omega\otimes E)\otimes\calI(\varphi+\psi))
	\end{equation*}
	and
	\begin{equation*}
		\int_\Omega |F|_{\omega,h}^2e^{-\varphi} dV_\omega \leqslant (e^a+1) \int_{\Omega_a} |f|_{\omega,h}^2e^{-\varphi} dV_\omega.
	\end{equation*}
\end{theorem}

\begin{proof}
	{\itshape\underline{Step 1.}} At first, we introduce some auxiliary functions. We define
	\begin{equation*}
		v(t) := \begin{cases} e^{t+a} & (t<-a) \\ t+a+1 & (-a\leqslant t<0) \end{cases}.
	\end{equation*}
	We need to approximate $v(t)$ by smooth functions. For each $\eps\in(0,\frac{1}{4})$, we choose a smooth function $\theta_\eps(t)\geqslant0$ on $(-\infty,0)$ so that:
	
	$\bullet$ $\supp\theta_\eps\subset[-a-\log\eps^{-1}-\frac{1}{2},-a]$, $\theta_\eps(t)\equiv e^{t+a}$ on $[-a-\log\eps^{-1},-a-\eps]$;\par
	$\bullet$ $\int_{-\infty}^{-a-\log\eps^{-1}} \theta_\eps(t)dt = \int_{-\infty}^{-a-\log\eps^{-1}} e^{t+a}dt$ and $\int_{-\infty}^{-a}\theta_\eps(t)dt = \int_{-\infty}^{-a} e^{t+a}dt$.
	
	\noindent Then we define a smooth convex increasing function $v_\eps(t)$ on $(-\infty,0)$ by
	\begin{equation*}
		v_\eps(t) := \int_{-a}^t \left(\int_{-\infty}^{\tau_1} \theta_\eps(\tau_2) d\tau_2\right) d\tau_1 + \left(1-\frac{\eps}{4}\right).
	\end{equation*}
	It is easy to verify that:
	
	$\bullet$ $v_\eps(t)$ is a positive constant for $t\leqslant -a-\log\eps^{-1}-1$;\par
	$\bullet$ $v_\eps(t)=t+a+1-\frac{\eps}{4}=v(t)-\frac{\eps}{4}$ for $t\geqslant-a$;\par
	$\bullet$ $v_\eps'(t)=v_\eps''(t)=e^{t+a}=v'(t)=v''(t)$ on $[-a-\log\eps^{-1},-a-\eps]$;\par
	$\bullet$ $v_\eps(t)$ converges uniformly to $v(t)$ as $\eps\rightarrow 0$.
	
	\noindent Let $\chi(t)$ be a smooth function on $\RR$ such that $0\leqslant\chi\leqslant 1$, $\chi(t)\equiv 1$ for $t\leqslant\delta$ and $\chi(t)\equiv 0$ for $t\geqslant1-\delta$, where $\delta\in(0,\tfrac{1}{2})$ is a constant. Notice that,
	\begin{equation*}
		\lim_{\eps\rightarrow0} v_\eps(-a-\log\eps^{-1}) = 0 \quad\text{and}\quad \lim_{\eps\rightarrow0} v_\eps(-a-\eps) = 1.
	\end{equation*}
	In the following, we always assume that $\eps>0$ is small enough so that
	\begin{equation*}
		v_\eps(-a-\log\eps^{-1})<\delta<1-\delta<v_\eps(-a-\eps).
	\end{equation*}
	Since $v_\eps(t)$ is increasing, then $\chi(v_\eps(t))\neq 0$ implies $t<-a-\eps$, and $\chi'(v_\eps(t))\neq 0$ implies $-a-\log\eps^{-1}<t<-a-\eps$.\\
	
	We choose a positive smooth increasing function $c(t)$ on $(0,a+1)$ such that $\int_1^{a+1}c(\tau)d\tau<+\infty$, and we define
	\begin{gather*}
		u(t) := a - \log\left( \int_t^{a+1}c(\tau)d\tau \right), \quad
		s(t) := \frac{\int_t^{a+1} \int_{\tau_1}^{a+1}c(\tau_2)d\tau_2 d\tau_1} {\int_t^{a+1}c(\tau)d\tau}, \\
		g(t) := \frac{\left(\int_t^{a+1}c(\tau)d\tau\right)^2 - c(t)\left(\int_t^{a+1}\int_{\tau_1}^{a+1}c(\tau_2)d\tau_2d\tau_1\right)} {c(t)\left(\int_t^{a+1}c(\tau)d\tau\right)}.
	\end{gather*}
	Then $u,s,g$ are smooth functions on $(0,a+1)$ satisfying the following properties:
	\begin{gather*}
		s>0, \quad g>0, \quad u'>0, \quad s'<0, \\ su'-s'\equiv 1, \quad su''-s''-g^{-1}s's'\equiv 0, \quad \frac{e^{a-u(t)}}{s(t)+g(t)}\equiv c(t).
	\end{gather*}
	Indeed, we obtain the expressions of $u,s,g$ by solving the ODEs determined by the last three equations (see Guan-Zhou \cite{GuanZhou2015}). Since $0<\inf v_\eps\leqslant\sup v_\eps<a+1$, it is clear that $u(v_\eps(t))$, $s(v_\eps(t))$ and $g(v_\eps(t))$ are bounded smooth functions on $(-\infty,0)$. In the following, we further assume that
	\begin{equation} \label{Eq:AssumeC}
		c_{\min}:=\lim_{t\rightarrow0^+}c(t)>0 \quad\text{and}\quad
		\begin{cases} c(t)>t & (0<t<1) \\ c(t)\geqslant e^{t-1} & (1\leqslant t<a+1) \end{cases}.
	\end{equation}
	
	{\itshape\underline{Step 2.}} Since $\Omega$ is weakly pseudoconvex, there exists a sequence $\{D^k\}_{k=1}^\infty$ of domains in $\Omega$ such that $D^k\Subset D^{k+1}\Subset \Omega$, $\cup_kD^k=\Omega$, and each $D^k$ is also weakly pseudoconvex. By Theorem \ref{Thm:CompleteKahler}, all $D^k$ are complete K\"ahler manifolds.
	
	{\itshape Suppose $\dim\Omega=n$. Let $k\in\NN_+$ and $0<\eps\ll1$ be fixed until Step 4.}
	
	According to Theorem \ref{Thm:EquiSingAppro}, there exist two decreasing sequences $\{\psi_{k,\nu}\}_{\nu=1}^\infty$ and $\{\varphi_{k,\nu}\}_{\nu=1}^\infty$ of quasi-psh functions on $D^{k+1}\Subset\Omega$ satisfying the following conditions:
	
	$\bullet$ $\psi\leqslant\psi_{k,\nu}<0$ and $\psi_{k,\nu}\searrow\psi$ as $\nu\nearrow+\infty$;
	
	$\bullet$ $\varphi\leqslant\varphi_{k,\nu} < \sup_{D^{k+2}}\varphi + 1$ and $\varphi_{k,\nu}\searrow\varphi$ as $\nu\nearrow+\infty$;
	
	$\bullet$ $\psi_{k,\nu}\in C^\infty(D^{k+1}\backslash\psi_{k,\nu}^{-1}(-\infty))$ and $\varphi_{k,\nu}\in C^\infty(D^{k+1}\backslash\varphi_{k,\nu}^{-1}(-\infty))$;
	
	$\bullet$ $\psi_{k,\nu}^{-1}(-\infty)$ and $\varphi_{k,\nu}^{-1}(-\infty)$ are closed analytic subsets of $D^{k+1}$;
	
	$\bullet$ $\ddbar\psi_{k,\nu} \geqslant \gamma-\sigma_{k,\nu}\omega$ and $\ddbar\varphi_{k,\nu} \geqslant \rho-\sigma_{k,\nu}\omega$, where $\{\sigma_{k,\nu}\}_{\nu=1}^\infty$ is a sequence of positive numbers so that $\lim_{\nu\rightarrow+\infty}\sigma_{k,\nu}=0$.
	
	Because $S_{k,\nu} := \psi_{k,\nu}^{-1}(-\infty) \cup \varphi_{k,\nu}^{-1}(-\infty)$ is a closed analytic subset of $D^{k+1}$, it follows from Theorem \ref{Thm:CompleteKahler} that $\DkSn$ admits a complete K\"ahler metric.
	
	Since $v_\eps(\psi_{k,\nu})$ is constant on $\{\psi_{k,\nu}<-a-\log\eps^{-1}-1\}$, we know $v_\eps(\psi_{k,\nu})\in C^\infty(D^k)$. As $0<\inf v_\eps \leqslant v_\eps(\psi_{k,\nu}) \leqslant \sup v_\eps<a+1$, it is clear that
	\begin{equation*}
		\phi_{\eps,k,\nu} := u(v_\eps(\psi_{k,\nu})), \quad \eta_{\eps,k,\nu} := s(v_\eps(\psi_{k,\nu})), \quad \lambda_{\eps,k,\nu} := g(v_\eps(\psi_{k,\nu}))
	\end{equation*}
	are bounded smooth functions on $D^k$. We shall consider the following smooth Hermitian metric on $E|_{\DkSn}$:
	\begin{equation*}
		h_{\eps,k,\nu} := he^{-\varphi_{k,\nu}-\psi_{k,\nu}-\phi_{\eps,k,\nu}}.
	\end{equation*}
	For convenience, we use $|\cdot|_{\eps,k,\nu}$ and $\inner{\cdot,\cdot}_{\eps,k,\nu}$ to denote the fibre metric and the inner product on $\wedge^{p,q}T^*\Omega\otimes E$ that induced by $\omega$ and $h_{\eps,k,\nu}$.\\
	
	In view of Theorem \ref{Thm:L2ExistenceErrTerm}, we define
	\begin{align*}
		\Theta_{\eps,k,\nu} := \eta_{\eps,k,\nu}\sqrt{-1}\Theta(E,h_{\eps,k,\nu}) - \ddbar\eta_{\eps,k,\nu} \\
		- \lambda_{\eps,k,\nu}^{-1}\dwdbar{\eta_{\eps,k,\nu}}.
	\end{align*}
	By direct computations,
	\begin{align*}
		\Theta_{\eps,k,\nu} = \eta_{\eps,k,\nu}\sqrt{-1}\Theta(E,he^{-\varphi_{k,\nu}-\psi_{k,\nu}}) + (su'-s')|_{v_\eps(\psi_{k,\nu})} \ddbar (v_\eps(\psi_{k,\nu})) \\
		+ (su''-s''-g^{-1}s's')|_{v_\eps(\psi_{k,\nu})} \dwdbar{(v_\eps(\psi_{k,\nu}))}.
	\end{align*}
	Since $su'-s' \equiv 1$ and $su''-s''-g^{-1}s's' \equiv 0$, it follows that
	\begin{align*}
		\Theta_{\eps,k,\nu} = &\, \eta_{\eps,k,\nu}\sqrt{-1}\Theta(E,he^{-\varphi_{k,\nu}-\psi_{k,\nu}}) + \ddbar (v_\eps(\psi_{k,\nu})) \\
		= &\, s(v_\eps(\psi_{k,\nu}))(\sqrt{-1}\Theta(E,h) + \ddbar\varphi_{k,\nu} + \ddbar\psi_{k,\nu}) \\
		& + v_\eps'(\psi_{k,\nu})\ddbar{\psi_{k,\nu}} + v_\eps''(\psi_{k,\nu})\dwdbar{\psi_{k,\nu}}.
	\end{align*}
	Since $\ddbar\varphi_{k,\nu} \geqslant \rho-\sigma_{k,\nu}\omega$ and $\ddbar\psi_{k,\nu} \geqslant \gamma-\sigma_{k,\nu}\omega \geqslant -\sigma_{k,\nu}\omega$, then
	\begin{align*}
		\Theta_{\eps,k,\nu} \geqslant_\Nak -2s(v_\eps(\psi_{k,\nu}))\sigma_{k,\nu}\omega - v_\eps'(\psi_{k,\nu})\sigma_{k,\nu}\omega \\ + v_\eps''(\psi_{k,\nu})\dwdbar{\psi_{k,\nu}}.
	\end{align*}
	Since $0 \leqslant 2s(v_\eps(t))+v_\eps'(t) \leqslant 2s(\inf v_\eps) + 1 =: C_\eps <+\infty$, we conclude that
	\begin{equation} \label{Eq:LowBdCurv}
		\Theta_{\eps,k,\nu} + \sigma_{k,\nu} C_\eps \omega \geqslant_\Nak v_\eps''(\psi_{k,\nu})\dwdbar{\psi_{k,\nu}} \geqslant_\Nak 0.
	\end{equation}
	Clearly, $\sigma_{k,\nu} C_\eps\to0$ as $\nu\to+\infty$. \\
	
	If $\chi(v_\eps(\psi_{k,\nu}))\neq0$, then $\psi\leqslant\psi_{k,\nu}<-a-\eps$. Since $f\in\Gamma(\Omega_a,K_\Omega\otimes E)$, it follows that $\chi(v_\eps(\psi_{k,\nu}))f$ is a smooth $(n,0)$-form on $D^k$. We define
	\begin{align*}
		\xi_{\eps,k,\nu} := &\, \bar{\pd}(\chi(v_\eps(\psi_{k,\nu}))f) \\
		= &\, \chi'(v_\eps(\psi_{k,\nu}))v'_\eps(\psi_{k,\nu})\cdot\bar{\pd}\psi_{k,\nu}\wedge f.
	\end{align*}
	Recall that, if $\chi'(v_\eps(t))\neq 0$, then $v'_\eps(t)=v''_\eps(t)=e^{t+a}$. For any $x\in\DkSn$ and $\alpha\in(\wedge^{n,1}T^*\Omega\otimes E)_x$, it follows from Lemma \ref{Lemma:AnIneq} and \eqref{Eq:LowBdCurv} that
	\begin{align*}
		&\,|\inner{\alpha,\xi_{\eps,k,\nu}}_{\eps,k,\nu}|^2 \\
		= &\, \chi'(v_\eps(\psi_{k,\nu}))^2 v'_\eps(\psi_{k,\nu})^2 \cdot |\inner{\alpha, \bar{\pd}\psi_{k,\nu}\wedge f}_{\eps,k,\nu}|^2 \\
		\leqslant &\, \chi'(v_\eps(\psi_{k,\nu}))^2 v'_\eps(\psi_{k,\nu})^2 \cdot |f|_{\eps,k,\nu}^2 \inner{[\dwdbar{\psi_{k,\nu}},\Lambda_\omega]\alpha,\alpha}_{\eps,k,\nu} \\
		= &\, \chi'(v_\eps(\psi_{k,\nu}))^2 |f|_{\eps,k,\nu}^2 e^{\psi_{k,\nu}+a} \cdot \inner{[v_\eps''(\psi_{k,\nu})\dwdbar{\psi_{k,\nu}}, \Lambda_\omega] \alpha, \alpha}_{\eps,k,\nu} \\
		\leqslant &\, \chi'(v_\eps(\psi_{k,\nu}))^2 |f|_{\eps,k,\nu}^2 e^{\psi_{k,\nu}+a} \cdot \inner{[\Theta_{\eps,k,\nu}+\sigma_{k,\nu} C_\eps\omega,\Lambda_\omega]\alpha,\alpha}_{\eps,k,\nu}.
	\end{align*}
	Let $B_{\eps,k,\nu} := [\Theta_{\eps,k,\nu}+\sigma_{k,\nu} C_\eps\omega,\Lambda_\omega]$, then the above inequality means that
	\begin{align*}
		\inner{B_{\eps,k,\nu}^{-1}\xi_{\eps,k,\nu},\xi_{\eps,k,\nu}}_{\eps,k,\nu}
		& \leqslant \chi'(v_\eps(\psi_{k,\nu}))^2 |f|_{\eps,k,\nu}^2 e^{\psi_{k,\nu}+a} \\
		& = \chi'(v_\eps(\psi_{k,\nu}))^2 e^{a-u(v_\eps(\psi_{k,\nu}))} |f|_{\omega,h}^2e^{-\varphi_{k,\nu}}.
	\end{align*}
	
	We shall examine the integrability of $\inner{B_{\eps,k,\nu}^{-1}\xi_{\eps,k,\nu},\xi_{\eps,k,\nu}}_{\eps,k,\nu}$:
	\begin{align*}
		&\, \int_{\DkSn} \inner{B_{\eps,k,\nu}^{-1}\xi_{\eps,k,\nu},\xi_{\eps,k,\nu}}_{\eps,k,\nu} dV_\omega \\
		\leqslant &\, \int_{\DkSn} \chi'(v_\eps(\psi_{k,\nu}))^2 e^{a-u(v_\eps(\psi_{k,\nu}))} |f|_{\omega,h}^2e^{-\varphi_{k,\nu}} dV_\omega \\
		\leqslant &\, \max_{\delta\leqslant t\leqslant 1-\delta}\chi'(t)^2e^{a-u(t)} \int_{\DkSn} \one_{\{-a-\log\eps^{-1}<\psi_{k,\nu}<-a-\eps\}} |f|_{\omega,h}^2e^{-\varphi_{k,\nu}} dV_\omega \\
		\leqslant &\, \max_{\delta\leqslant t\leqslant 1-\delta}\chi'(t)^2e^{a-u(t)} \int_{\Omega_a} |f|_{\omega,h}^2e^{-\varphi} dV_\omega =: Q < +\infty.
	\end{align*}
	Moreover, by Lebesgue's dominated convergence theorem,
	\begin{equation}\begin{aligned} \label{Eq:LimRHS}
			&\, \lim_{\nu\rightarrow+\infty} \int_{\DkSn} \chi'(v_\eps(\psi_{k,\nu}))^2 e^{a-u(v_\eps(\psi_{k,\nu}))} |f|_{\omega,h}^2 e^{-\varphi_{k,\nu}} dV_\omega \\
			= &\, \int_{D^k} \chi'(v_\eps(\psi))^2 e^{a-u(v_\eps(\psi))} |f|_{\omega,h}^2 e^{-\varphi} dV_\omega \leqslant Q.
	\end{aligned}\end{equation}
	
	Recall that, $\DkSn$ is a complete K\"ahler manifold, $\eta_{\eps,k,\nu}>0$ and $\lambda_{\eps,k,\nu}>0$ are bounded smooth functions on $\DkSn$. According to Theorem \ref{Thm:L2ExistenceErrTerm} and \eqref{Eq:LowBdCurv}, there exist differential forms
	\begin{equation*}
		\gamma_{\eps,k,\nu}\in L_{n,0}^2(\DkSn,E;\omega,h_{\eps,k,\nu})
	\end{equation*}
	and
	\begin{equation*}
		w_{\eps,k,\nu}\in L_{n,1}^2(\DkSn,E;\omega,h_{\eps,k,\nu})
	\end{equation*}
	such that
	\begin{equation} \label{Eq:SolErrTerm}
		\bar{\pd}\gamma_{\eps,k,\nu} + \sqrt{\sigma_{k,\nu} C_\eps}\,w_{\eps,k,\nu} = \xi_{\eps,k,\nu} = \bar{\pd}\left(\chi(v_\eps(\psi_{k,\nu}))f\right)
	\end{equation}
	on $\DkSn$ in the sense of distribution and
	\begin{align*}
		&\, \int_{\DkSn}\frac{|\gamma_{\eps,k,\nu}|_{\eps,k,\nu}^2}{\eta_{\eps,k,\nu}+\lambda_{\eps,k,\nu}}dV_\omega
		+ \int_{\DkSn}|w_{\eps,k,\nu}|_{\eps,k,\nu}^2dV_\omega \\
		\leqslant &\, \int_{\DkSn} \inner{B_{\eps,k,\nu}^{-1}\xi_{\eps,k,\nu},\xi_{\eps,k,\nu}}_{\eps,k,\nu} dV_\omega \leqslant Q.
	\end{align*}
	By the definitions of $\phi_{\eps,k,\nu}, \eta_{\eps,k,\nu}, \lambda_{\eps,k,\nu}$ and $h_{\eps,k,\nu}$, the above inequality implies
	\begin{equation}\begin{aligned} \label{Eq:EstSol}
			&\, \int_{\DkSn} |\gamma_{\eps,k,\nu}|_{\omega,h}^2 e^{-\varphi_{k,\nu}-\psi_{k,\nu}-a} c(v_\eps(\psi_{k,\nu})) dV_\omega \\
			&\, + \int_{\DkSn} |w_{\eps,k,\nu}|_{\omega,h}^2 e^{-\varphi_{k,\nu}-\psi_{k,\nu}-u(v_\eps(\psi_{k,\nu}))} dV_\omega \\
			\leqslant &\, \int_{\DkSn} \chi'(v_\eps(\psi_{k,\nu}))^2 e^{a-u(v_\eps(\psi_{k,\nu}))} |f|_{\omega,h}^2e^{-\varphi_{k,\nu}} dV_\omega \leqslant Q.
	\end{aligned}\end{equation}
	
	{\itshape\underline{Step 3.}} Recall that, $c(t)\geqslant c_{\min}>0$ and $u(v_\eps(\psi_{k,\nu})) \leqslant u(\sup v_\eps) < +\infty$. Since $\psi_{k,\nu}<0$ and $\varphi_{k,\nu}<\sup_{D^{k+2}}\varphi+1$ on $D^k$, it follows from \eqref{Eq:EstSol} that
	\begin{gather*}
		\int_{\DkSn}|\gamma_{\eps,k,\nu}|_{\omega,h}^2dV_\omega \leqslant \frac{1}{c_{\min}} \exp(\sup_{D^{k+2}}\varphi+a+1) Q < +\infty, \\
		\int_{\DkSn}|w_{\eps,k,\nu}|_{\omega,h}^2dV_\omega \leqslant \exp(\sup_{D^{k+2}}\varphi+u(\sup v_\eps)+1) Q < +\infty.
	\end{gather*}
	Because $S_{k,\nu}$ is a set of zero measure, $\gamma_{\eps,k,\nu}$ (resp. $w_{\eps,k,\nu}$) can be regarded as an element of $L_{n,0}^2(D^k,E;\omega,h)$ (resp. $L_{n,1}^2(D^k,E;\omega,h)$). On the other hand,
	\begin{equation*}
		\int_{D^k} |\chi(v_\eps(\psi_{k,\nu}))f|_{\omega,h}^2 dV_\omega \leqslant \exp(\sup_{D^k}\varphi) \int_{D^k\cap\Omega_a} |f|_{\omega,h}^2e^{-\varphi} dV_\omega < +\infty,
	\end{equation*}
	i.e. $\chi(v_\eps(\psi_{k,\nu}))f\in L_{n,0}^2(D^k,E;\omega,h)$. By Lemma \ref{Lemma:dbarOverS} and \eqref{Eq:SolErrTerm},
	\begin{equation} \label{Eq:SolAcrossS}
		\bar{\pd}\left(\chi(v_\eps(\psi_{k,\nu}))f - \gamma_{\eps,k,\nu}\right) = \sqrt{\sigma_{k,\nu} C_\eps}\,w_{\eps,k,\nu}
	\end{equation}
	on $D^k$ in the sense of distribution.
	
	Since $\{\gamma_{\eps,k,\nu}\}_\nu$ is a bounded sequence in $L_{n,0}^2(D^k,E;\omega,h)$, there exists a subsequence of $\{\gamma_{\eps,k,\nu}\}_\nu$ that converges weakly to some element $\gamma_{\eps,k}\in L_{n,0}^2(D^k,E;\omega,h)$. Without loss of generality, we may assume that the subsequence is $\{\gamma_{\eps,k,\nu}\}_\nu$ itself. Since $\{w_{\eps,k,\nu}\}_\nu$ is a bounded sequence in $L_{n,1}^2(D^k,E;\omega,h)$ and $\lim_{\nu\to+\infty}\sigma_{k,\nu}=0$, it is clear that $\sqrt{\sigma_{k,\nu} C_\eps}\,w_{\eps,k,\nu} \rightarrow0$ in $L_{n,1}^2(D^k,E;\omega,h)$. Notice that, $\chi(v_\eps(\psi_{k,\nu}))f$ converges pointwisely to $\chi(v_\eps(\psi))f$ and
	\begin{equation*}
		\int_{D^k} |\chi(v_\eps(\psi_{k,\nu}))f-\chi(v_\eps(\psi))f|_{\omega,h}^2 dV_\omega \leqslant \int_{D^k\cap\Omega_a} |f|_{\omega,h}^2 dV_\omega < +\infty,
	\end{equation*}
	then it follows from Lebesgue's dominated convergence theorem that $\chi(v_\eps(\psi_{k,\nu}))f$ $\to\chi(v_\eps(\psi))f$ in $L_{n,0}^2(D^k,E;\omega,h)$. In summary, as $\nu\to+\infty$,
	\begin{equation*}
		\chi(v_\eps(\psi_{k,\nu}))f \to \chi(v_\eps(\psi))f, \quad \gamma_{\eps,k,\nu} \rightharpoonup \gamma_{\eps,k}
		\quad\text{and}\quad \sqrt{\sigma_{k,\nu} C_\eps}\,w_{\eps,k,\nu} \to 0.
	\end{equation*}
	Since $\bar{\pd}$ is a closed operator, letting $\nu\rightarrow+\infty$ in \eqref{Eq:SolAcrossS}, we conclude that
	\begin{equation*}
		\bar{\pd}(\chi(v_\eps(\psi))f - \gamma_{\eps,k}) = 0
	\end{equation*}
	on $D^k$ in the sense of distribution. Therefore, after a modification on sets of zero measure, we obtain a holomorphic section
	\begin{equation*}
		F_{\eps,k}:=\chi(v_\eps(\psi))f-\gamma_{\eps,k} \in \Gamma(D^k,K_\Omega\otimes E).
	\end{equation*}
	
	We need to estimate the $L^2$ norm of $\gamma_{\eps,k}$.
	
	Let $S:=\psi^{-1}(-\infty)\cup\varphi^{-1}(-\infty)$, then $S$ is a set of zero measure and $S_{k,\nu}\subset S$ for any $\nu\in\NN_+$. Clearly, $c(v_\eps(\psi_{k,\nu}))$ converges pointwisely to $c(v_\eps(\psi))$ as $\nu\rightarrow+\infty$ and $0 < c(v_\eps(\psi_{k,\nu})) \leqslant c(\sup v_\eps) < +\infty$.
	Since $\gamma_{\eps,k,\nu}\rightharpoonup\gamma_{\eps,k}$ in $L_{n,0}^2(\DkS,E;\omega,h)$, it follows from Lemma \ref{Lemma:WeakLim} that
	\begin{equation*}
		\sqrt{c(v_\eps(\psi_{k,\nu}))} \gamma_{\eps,k,\nu} \rightharpoonup \sqrt{c(v_\eps(\psi))} \gamma_{\eps,k} \quad \text{in}\quad L_{n,0}^2(\DkS,E;\omega,h).
	\end{equation*}
	Let $\nu_0\in\NN_+$ be fixed for the moment. By \eqref{Eq:EstSol}, for any $\nu\geqslant\nu_0$,
	\begin{align*}
		&\, \int_{\DkS} |\sqrt{c(v_\eps(\psi_{k,\nu}))} \gamma_{\eps,k,\nu}|_{\omega,h}^2 e^{-\varphi_{k,\nu_0}-\psi_{k,\nu_0}-a} dV_\omega \\
		\leqslant &\, \int_{\DkSn} \chi'(v_\eps(\psi_{k,\nu}))^2 e^{a-u(v_\eps(\psi_{k,\nu}))} |f|_{\omega,h}^2 e^{-\varphi_{k,\nu}} dV_\omega \leqslant Q.
	\end{align*}
	According to Lemma \ref{Lemma:DoubleWeakLim},
	\begin{align*}
		\sqrt{c(v_\eps(\psi_{k,\nu}))} \gamma_{\eps,k,\nu} \rightharpoonup \sqrt{c(v_\eps(\psi))} \gamma_{\eps,k}
	\end{align*}
	in $L_{n,0}^2(\DkS,E;\omega,h,e^{-\varphi_{k,\nu_0}-\psi_{k,\nu_0}-a}dV_\omega)$. Then it follows from \eqref{Eq:LimRHS} that
	\begin{align*}
		&\, \int_{\DkS} |\sqrt{c(v_\eps(\psi))}\gamma_{\eps,k}|_{\omega,h}^2 e^{-\varphi_{k,\nu_0}-\psi_{k,\nu_0}-a} dV_\omega \\
		\leqslant &\, \varliminf_{\nu\rightarrow+\infty} \int_{\DkS} |\sqrt{c(v_\eps(\psi_{k,\nu}))}\gamma_{\eps,k,\nu}|_{\omega,h}^2 e^{-\varphi_{k,\nu_0}-\psi_{k,\nu_0}-a} dV_\omega \\
		\leqslant &\, \int_{D^k} \chi'(v_\eps(\psi))^2e^{a-u(v_\eps(\psi))} |f|_{\omega,h}^2e^{-\varphi} dV_\omega \leqslant Q.
	\end{align*}
	Since $S$ is a set of zero measure, letting $\nu_0\nearrow+\infty$, we have
	\begin{equation}\begin{aligned} \label{Eq:EstGamma}
			&\, \int_{D^k}|\gamma_{\eps,k}|_{\omega,h}^2 e^{-\varphi-\psi-a}c(v_\eps(\psi))dV_\omega \\
			\leqslant &\, \int_{D^k} \chi'(v_\eps(\psi))^2e^{a-u(v_\eps(\psi))} |f|_{\omega,h}^2e^{-\varphi} dV_\omega \leqslant Q.
	\end{aligned}\end{equation}
	
	Since $\gamma_{\eps,k} = F_{\eps,k}-\chi(v_\eps(\psi))f$, it follows that
	\begin{equation*}
		\int_{D^k} |F_{\eps,k}-\chi(v_\eps(\psi))f|_{\omega,h}^2 e^{-\varphi-\psi}dV_\omega \leqslant \frac{1}{c_{\min}}e^aQ < +\infty.
	\end{equation*}
	On the other hand, since $\chi(v_\eps(t))=1$ for $t\leqslant-a-\log\eps^{-1}$, we know
	\begin{align*}
		&\, \int_{D^k\cap\Omega_a} |f-\chi(v_\eps(\psi))f|_{\omega,h}^2 e^{-\varphi-\psi}dV_\omega \\
		\leqslant &\, \int_{D^k\cap\Omega_a} \one_{\{\psi>-a-\log\eps^{-1}\}} |f|_{\omega,h}^2 e^{-\varphi-\psi}dV_\omega \\
		\leqslant &\, \eps^{-1}e^a \int_{D^k\cap\Omega_a} |f|_{\omega,h}^2 e^{-\varphi}dV_\omega < +\infty.
	\end{align*}
	Since $F_{\eps,k}-f = (F_{\eps,k}-\chi(v_\eps(\psi))f) - (f-\chi(v_\eps(\psi))f)$, we conclude that
	\begin{equation} \label{Eq:EqualwrtCalI}
		F_{\eps,k}|_{D^k\cap\Omega_a} - f|_{D^k\cap\Omega_a} \in \Gamma(D^k\cap\Omega_a,\calO(K_\Omega\otimes E)\otimes\calI(\varphi+\psi)).
	\end{equation}
	
	{\itshape\underline{Step 4.}} In this step, we let $\eps\to0$ and then $k\to+\infty$. By \eqref{Eq:EstGamma},
	\begin{equation*}
		\int_{D^k} |F_{\eps,k}-\chi(v_\eps(\psi))f|_{\omega,h}^2 dV_\omega \leqslant \frac{1}{c_{\min}}\exp(\sup_{D^k}\varphi+a)Q < +\infty.
	\end{equation*}
	On the other hand,
	\begin{equation*}
		\int_{D^k}|\chi(v_\eps(\psi))f|_{\omega,h}^2dV_\omega \leqslant \exp(\sup_{D^k}\varphi) \int_{D^k\cap\Omega_a} |f|_{\omega,h}^2 e^{-\varphi}dV_\omega < +\infty.
	\end{equation*}
	Therefore, for fixed $k\in\NN_+$, $\{F_{\eps,k}\}_\eps$ is a bounded family in $A^2(D^k,K_\Omega\otimes E)$. By Montel's theorem, there exists a sequence $\{\eps_j\}_{j=1}^\infty$ such that $\eps_j\searrow0$ and $\{F_{\eps_j,k}\}_j$ converges uniformly on any compact subsets of $D^k$ to some holomorphic section $F_k\in\Gamma(D^k,K_\Omega\otimes E)$. According to Lemma \ref{Lemma:Coherent} and \eqref{Eq:EqualwrtCalI},
	\begin{equation*}
		F_k|_{D^k\cap\Omega_a} - f|_{D^k\cap\Omega_a} \in \Gamma(D^k\cap\Omega_a,\calO(K_\Omega\otimes E)\otimes\calI(\varphi+\psi)).
	\end{equation*}
	
	Notice that, $\chi'(v_{\eps}(\psi))^2e^{a-u(v_{\eps}(\psi))} |f|_{\omega,h}^2e^{-\varphi}$ is dominated by
	\begin{equation*}
		\max_{\delta\leqslant t\leqslant 1-\delta}\chi'(t)^2e^{a-u(t)} \cdot \one_{\{\psi<-a\}} |f|_{\omega,h}^2e^{-\varphi} \in L^1(D^k;dV_\omega).
	\end{equation*}
	Using Fatou's lemma, the inequality \eqref{Eq:EstGamma} and Lebesgue's dominated convergence theorem, we have
	\begin{align*}
		&\, \int_{D^k} |F_k-\chi(v(\psi))f|_{\omega,h}^2 e^{-\varphi-\psi-a}c(v(\psi)) dV_\omega \\
		= &\, \int_{D^k} \lim_{j\rightarrow+\infty} |F_{\eps_j,k}-\chi(v_{\eps_j}(\psi))f|_{\omega,h}^2 e^{-\varphi-\psi-a}c(v_{\eps_j}(\psi)) dV_\omega \\
		\leqslant &\, \lim_{j\rightarrow+\infty} \int_{D^k}\chi'(v_{\eps_j}(\psi))^2e^{a-u(v_{\eps_j}(\psi))} |f|_{\omega,h}^2e^{-\varphi} dV_\omega \\
		= &\, \int_{D^k\cap\Omega_a} \chi'(v(\psi))^2e^{a-u(v(\psi))} |f|_{\omega,h}^2e^{-\varphi} dV_\omega.
	\end{align*}
	
	Clearly, $(x+y)^2\leqslant \alpha x^2+\frac{\alpha}{\alpha-1}y^2$ for any $x>0,y>0,\alpha>1$. According to our assumptions \eqref{Eq:AssumeC}:
	
	$\bullet$ if $\psi<-a$, then $v(\psi)=e^{\psi+a}<1$, and then $e^{-\psi-a}c(v(\psi))>1$; \par
	$\bullet$ if $-a\leqslant\psi<0$, then $v(\psi)=\psi+a+1\geqslant1$, and then $e^{-\psi-a}c(v(\psi))\geqslant 1$.
	
	\noindent Since $\supp\chi(v(\psi))\subset\Omega_a$, we have
	\begin{gather*}
		\int_{D^k} |F_k|_{\omega,h}^2e^{-\varphi} dV_\omega
		\leqslant \int_{D^k} e^{-\psi-a}c(v(\psi)) |F_k-\chi(v(\psi))f|_{\omega,h}^2e^{-\varphi} dV_\omega \\
		+ \int_{D^k\cap\Omega_a} \frac{e^{-\psi-a}c(v(\psi))}{e^{-\psi-a}c(v(\psi))-1} |\chi(v(\psi))f|_{\omega,h}^2e^{-\varphi} dV_\omega \\
		\leqslant \int_{D^k\cap\Omega_a} \left(\chi'(v(\psi))^2e^{a-u(v(\psi))} + \frac{c(v(\psi))}{c(v(\psi))-e^{\psi+a}}\chi(v(\psi))^2\right) |f|_{\omega,h}^2e^{-\varphi} dV_\omega.
	\end{gather*}
	Recall that, $e^{a-u(t)}=\int_t^{a+1}c(\tau)d\tau$ and $v(\psi)=e^{\psi+a}<1$ for $\psi<-a$. If we define
	\begin{equation*}
		C(\chi,c):= \sup_{0<t<1} \left(\chi'(t)^2\int_t^{a+1}c(\tau)d\tau+\frac{c(t)}{c(t)-t}\chi(t)^2\right) < +\infty,
	\end{equation*}
	then the above inequality implies
	\begin{equation*}
		\int_{D^k} |F_k|_{\omega,h}^2e^{-\varphi} dV_\omega \leqslant C(\chi,c) \int_{D^k\cap\Omega_a} |f|_{\omega,h}^2e^{-\varphi} dV_\omega.
	\end{equation*}
	
	By Montel's theorem, there is a subsequence of $\{F_k\}_k$ that converges uniformly on any compact subsets of $\Omega$ to some holomorphic section $F\in\Gamma(\Omega,K_\Omega\otimes E)$. It is easy to show that
	\begin{equation} \label{Eq:Cond1}
		F|_{\Omega_a} - f \in \Gamma(\Omega_a,\calO(K_\Omega\otimes E)\otimes\calI(\varphi+\psi))
	\end{equation}
	and
	\begin{equation} \label{Eq:Cond2}
		\int_\Omega |F|_{\omega,h}^2e^{-\varphi} dV_\omega \leqslant C(\chi,c) \int_{\Omega_a} |f|_{\omega,h}^2e^{-\varphi} dV_\omega.
	\end{equation}
	
	In summary, if $\chi$ and $c$ are smooth functions satisfying the requirements listed in Step 1, then there exists a holomorphic section $F\in\Gamma(\Omega,K_\Omega\otimes E)$ satisfying \eqref{Eq:Cond1} and \eqref{Eq:Cond2}. Via approximation, we can draw the same conclusion for some other $\chi$ and $c$.
	
	{\itshape\underline{Step 5.}} In the following, let $\chi\in C^1((0,1))$ be a decreasing function such that
	\begin{equation*}
		\lim_{t\to0^+}\chi(t) = 1, \quad \lim_{t\to1^-}\chi(t) = 0 \quad\text{and}\quad A:=\sup_{t\in(0,1)}|\chi'(t)|<+\infty,
	\end{equation*}
	let
	\begin{equation} \label{Eq:defc}
		c(t)=\begin{cases}1 & (0<t<1) \\ e^{t-1} & (1\leqslant t< a+1) \end{cases}.
	\end{equation}
	
	We shall construct smooth approximations of $\chi$ and $c$. Let $0<\eps\ll 1$ be given.
	
	We take a smooth function $\rho_\eps\in C^\infty((0,1))$ so that $0\leqslant\rho_\eps(t) \leqslant \frac{1}{1-\eps}|\chi'(t)| + \eps$, $\int_0^1\rho_\eps(t)dt=1$ and $\supp\rho_\eps\subset[\delta_\eps,1-\delta_\eps]$ for some $\delta_\eps>0$. Let $\chi_\eps(t)=\int_t^1\rho_\eps(\tau)d\tau$, then $\chi_\eps(t)\equiv1$ for $t\leqslant\delta_\eps$ and $\chi_\eps(t)\equiv0$ for $t\geqslant1-\delta_\eps$. Moreover, $\chi_\eps(t) \leqslant \frac{1}{1-\eps}\chi(t) + \eps(1-t)$. Let $c_\eps(t)$ be a smooth increasing function on $(0,a+1)$ such that $c(t)\leqslant c_\eps(t)\leqslant c(t)+\eps$ and $c_\eps(t)\equiv c(t)$ on $(0,1-\eps)\cup(1+\eps,a+1)$.
	
	Clearly, $\chi_\eps(t)$ and $c_\eps(t)$ are smooth functions satisfying the requirements listed in Step 1. Moreover,
	\begin{equation*}
		\varlimsup_{\eps\rightarrow0} C(\chi_\eps,c_\eps) \leqslant C(\chi,c) <+\infty.
	\end{equation*}
	For each $0<\eps\ll1$, there exists a holomorphic section $F_\eps\in\Gamma(\Omega,K_\Omega\otimes E)$ such that $F_\eps|_{\Omega_a} - f \in \Gamma(\Omega_a,\calO(K_\Omega\otimes E)\otimes\calI(\varphi+\psi))$ and
	\begin{equation*}
		\int_\Omega |F_\eps|_{\omega,h}^2e^{-\varphi} dV_\omega \leqslant C(\chi_\eps,c_\eps) \int_{\Omega_a} |f|_{\omega,h}^2e^{-\varphi} dV_\omega.
	\end{equation*}
	Applying Montel's theorem, we find a holomorphic section $F\in\Gamma(\Omega,K_\Omega\otimes E)$ with
	\begin{equation*}
		F|_{\Omega_a} - f \in \Gamma(\Omega_a,\calO(K_\Omega\otimes E)\otimes\calI(\varphi+\psi))
	\end{equation*}
	and
	\begin{equation*}
		\int_\Omega |F|_{\omega,h}^2e^{-\varphi} dV_\omega \leqslant C(\chi,c) \int_{\Omega_a} |f|_{\omega,h}^2e^{-\varphi} dV_\omega.
	\end{equation*}
	In particular, we take $\chi(t)\equiv1-t$, then
	\begin{align*}
		C(\chi,c) & = \sup_{0<t<1}(1\cdot(e^a-t)+\tfrac{1}{1-t}\cdot(1-t)^2) \\
		& = \sup_{0<t<1}(e^a+1-2t) = e^a+1.
	\end{align*}
	This completes the proof!
\end{proof}

The example after Theorem \ref{MainThm:Opt} shows that  the uniform constant in Theorem \ref{Thm:CoarseExtThm} should be $\geqslant e^a$. Therefore, the constant $e^a+1$ obtained here is \textit{asymptotically optimal} as $a\to+\infty$. In Section \ref{Chap:Optimal}, we will obtain the optimal estimate by indirect methods. Nevertheless, the estimate of Theorem \ref{Thm:CoarseExtThm} is sufficient to yields certain optimal (jet) $L^2$ extension theorem of Ohsawa-Takegoshi type (see Section \ref{Chap:OTtype}).

\begin{remark}
	If we choose $\chi(t)$ and $c(t)$ more carefully in Step 5, then we can obtain a uniform constant smaller than $e^a+1$. Let $\kappa=\kappa(a)>0$ be a constant so that $\dfrac{\kappa}{e^{\kappa}-1}=\sqrt{1-e^{-a}}$, let $\chi(t)=\dfrac{e^{\kappa}-e^{\kappa t}}{e^{\kappa}-1}$, and let $c(t)$ be the same as \eqref{Eq:defc}.  By careful computations, for $a\geqslant1$, we can show that
	\begin{equation*}
		C(\chi,c) \leqslant e^a + \left(\frac{\kappa e^\kappa}{e^\kappa-1}\right)^2 - 1 < e^a+\frac{25}{16}e^{-a}.
	\end{equation*}
	In particular, the gap between $C(\chi,c)$ and $e^a$ decays to $0$ exponentially.
\end{remark}

\begin{remark}
	In the proof of Theorem \ref{Thm:CoarseExtThm}, the assumption that
	\begin{center}
		\itshape ``$(\Omega,\omega)$ is a weakly pseudoconvex K\"ahler manifold''
	\end{center}
	is only used at the beginning of Step 2, where we construct a sequence $\{D^k\}_{k=1}^\infty$ of domains in $\Omega$ such that $\Omega=\cup_kD^k$, $D^k\Subset D^{k+1}\Subset \Omega$, and each $D^k$ is also weakly pseudoconvex. Since $D^{k+1}$ is K\"ahler, $S_{k,\nu}\subset D^{k+1}$ is a closed analytic subset and $D^k$ has a complete K\"ahler metric, we know that $\DkSn$ is a complete K\"ahler manifold, and then we can solve $\bar{\pd}$-equations with $L^2$-estimates on $\DkSn$. Moreover, for any $\xi\in K_\Omega\otimes E$, it is clear that $|\xi|_{\omega,h}^2dV_\omega = (\sqrt{-1})^{n^2}\xi\wedge_h\bar{\xi}$ is independent of the choice of $\omega$. Therefore, we may replace the assumption by
	\begin{quotation}
		\itshape $\Omega$ is a complex manifold and there exists a sequence $\{D^k\}_{k=1}^\infty$ of domains in $\Omega$ so that $\Omega=\cup_kD^k$, $D^k\Subset D^{k+1}\Subset \Omega$, and each $D^k$ has a complete K\"ahler metric. \hfill \upshape ($\star$)
	\end{quotation}
	So does Theorem \ref{MainThm:Opt}, \ref{Thm:OptExtWeakForm}, \ref{Thm:BasicL2Existence} and \ref{Thm:NewOptOpenExt} in subsequent sections.
	
	If $\Omega$ is a complex manifold satisfying ($\star$) and $S\subset\Omega$ is a closed submanifold, then $S$ also satisfies this condition. If $\Omega_1,\Omega_2$ are complex manifolds satisfying ($\star$), then the product manifold $\Omega_1\times\Omega_2$ also satisfies this condition.
\end{remark}

\begin{example}
	It will be interesting to find examples other than weakly pseudoconvex K\"ahler manifolds that satisfy ($\star$).
	Forn{\ae}ss \cite{Fornaess1976} constructed a counterexample showing that the increasing union of Stein manifolds may not be Stein. More precisely, there exists a complex manifold $M$ and open sets $M_1\Subset M_2\Subset\cdots\Subset M$ so that $M=\cup_jM_j$, each $M_j$ is biholomorphic to some open ball in $\CC^3$, and $M$ is not holomorphically convex. With some modifications, we can further require that $M$ is not weakly pseudoconvex. Obviously, $M$ satisfies the condition ($\star$).
\end{example}

\section{A Product Property for Minimal \texorpdfstring{$L^2$}{L2} Extensions} \label{Chap:Prod}

Let $D_1\subset\CC^n$ and $D_2\subset\CC^m$ be open sets, let $D:=D_1\times D_2$ be their Cartesian product. For $i=1$ and $2$, let $\psi_i$ be a measurable function on $D_i$ which is locally bounded from above, and we assume that $\{\psi_i=-\infty\}$ is a set of zero measure. We define $\psi(z_1,z_2):=\psi_1(z_1)+\psi_2(z_2)$ on $D$. The following product properties for Bergman spaces and Bergman kernels are well-known:
\begin{gather}
	A^2(D;e^{-\psi}) = A^2(D_1;e^{-\psi_1}) \hat{\otimes} A^2(D_2;e^{-\psi_2}), \\
	\label{Eq:ProdPropKer} B_D((z_1,z_2),(w_1,w_2);e^{-\psi}) = B_{D_1}(z_1,w_1;e^{-\psi_1}) B_{D_2}(z_2,w_2;e^{-\psi_2}).
\end{gather}
Here, $\hat{\otimes}$ means the Hilbert tensor product: let $H$ and $H'$ be two Hilbert spaces, let $\{\phi_\mu\}_\mu$ (resp. $\{\varphi_\nu\}_\nu$) be a complete orthonormal basis of $H$ (resp. $H'$), then $H\hat{\otimes}H'$ is a Hilbert space having $\{\phi_\mu\otimes\varphi_\nu\}_{\mu,\nu}$ as a complete orthonormal basis.

Let $w=(w_1,w_2)\in D_1\times D_2$. By the extremal property of Bergman kernels,
\begin{equation*}
	f_i = \frac{B_{D_i}(\cdot,w_i;e^{-\psi_i})} {B_{D_i}(w_i;e^{-\psi_i})} \quad
	(\text{resp. } f = \frac{B_D(\cdot,w;e^{-\psi})} {B_D(w;e^{-\psi})})
\end{equation*}
is the unique element with minimal $L^2$ norm in $A^2(D_i;e^{-\psi_i})$ (resp. $A^2(D;e^{-\psi})$) such that $f_i(w_i)=1$ (resp. $f(w)=1$). In other words, $f_i\in A^2(D_i;e^{-\psi_i})$ (resp. $f\in A^2(D;e^{-\psi})$) is the \textbf{minimal $L^2$ extension} of $1\in\CC$ from $w_i$ to $D_i$ (resp. from $w$ to $D$). According to \eqref{Eq:ProdPropKer},
\begin{gather*}
	f(z_1,z_2) = f_1(z_1)f_2(z_2), \\
	\|f\|_{A^2(D;e^{-\psi})} = \|f_1\|_{A^2(D_1;e^{-\psi_1})} \|f_2\|_{A^2(D_2;e^{-\psi_2})}.
\end{gather*}
Hence, the product property for Bergman kernels is equivalent to a product property for certain minimal $L^2$ extensions.

In this section, we will generalize these product properties to the setting of complex manifolds, holomorphic vector bundles and general minimal $L^2$ extensions.\\

Throughout this section, $\Omega_1$ (resp. $\Omega_2$) is a complex manifold of dimension $n$ (resp. $m$), and $E_1\to\Omega_1$ (resp. $E_2\to\Omega_2$) is a holomorphic vector bundle of rank $r$ (resp. $s$). Let $\Omega:=\Omega_1\times\Omega_2$ and let $p_i:\Omega\rightarrow\Omega_i$ be the natural projection, then $E:=p_1^*E_1\otimes p_2^*E_2$ is a holomorphic vector bundle on $\Omega$.

For $i=1$ and $2$, we choose a continuous volume form $dV_i$ on $\Omega_i$, a continuous Hermitian metric $h_i$ on $E_i$ and a measurable function $\psi_i$ on $\Omega_i$. We assume that $\psi_i$ is locally bounded from above and $\{\psi_i=-\infty\}$ is a set of zero measure. Then we define $dV:= p_1^*dV_1\times p_2^*dV_2$, $h:=p_1^*h_1\otimes p_2^*h_2$ and $\psi:=p_1^*\psi_1+p_2^*\psi_2$. We shall investigate the relation between the following Bergman spaces:
\begin{equation*}
	A^2(\Omega_i,E_i) := A^2(\Omega_i,E_i;h_i,e^{-\psi_i}dV_i), \quad
	A^2(\Omega,E) := A^2(\Omega,E;h,e^{-\psi}dV).
\end{equation*}

Let $(U, z=(z_1,\ldots,z_n), e=(e_1,\ldots,e_r))$ be a local trivialization of $E_1\rightarrow\Omega_1$. For any multi-order $\alpha=(\alpha_1,\ldots,\alpha_n)\in\NN^n$, we can define a differential operator $\pd_z^\alpha:C^\infty(U,E_1)\to C^\infty(U,E_1)$ by
\begin{equation*}
	\sum\nolimits_{i=1}^r f_i\cdot e_i \quad\mapsto\quad \sum\nolimits_{i=1}^r
	\frac{\pd^{|\alpha|}f_i}{\pd z_1^{\alpha_1}\cdots\pd z_n^{\alpha_n}} \cdot e_i.
\end{equation*}
Notice that, this definition \textit{depends} on the chosen trivialization! Let $(V,w,\tilde{e})$ be a local trivialization of $E_2\rightarrow\Omega_2$, then we can define $\pd_w^\beta:C^\infty(V,E_2)\to C^\infty(V,E_2)$ and $\pd_z^\alpha,\pd_w^\beta:C^\infty(U\times V,E)\to C^\infty(U\times V,E)$ in similar ways. Notice that, for any $a\in(E_1)_x$, $b\in(E_2)_y$ and $\xi\in E_{(x,y)}$, where $x\in\Omega_1$ and $y\in\Omega_2$, the pairings $\inner{\xi,a}_{h_1} \in (E_2)_y$ and $\inner{\xi,b}_{h_2} \in (E_1)_x$ have obvious interpretations.

By elementary calculations in terms of some local trivializations, it is not hard to prove to following lemmas\footnote{Following the referees' suggestions, we have removed the lengthy proofs of Lemma 4.1 and 4.2 in the final version. Interested readers can refer to the previous version (v2) for the detailed proofs.}.

\begin{lemma} \label{Lemma:BergmanSlice}
	Let $(U,z,e)$ be a local trivialization of $E_1\rightarrow\Omega_1$. Given a holomorphic section $f\in A^2(U\times\Omega_2,E)$ and a multi-order $\alpha\in\NN^n$, we write $\pd_z^\alpha f=\sum_i e_i\otimes f_i^\alpha$. Then $f_i^\alpha(x,\cdot)\in A^2(\Omega_2,E_2)$ for any $x\in U$ and $1\leqslant i\leqslant r$.
\end{lemma}

\begin{lemma} \label{Lemma:BergmanInt}
	Let $f\in A^2(\Omega,E)$ and $\phi\in A^2(\Omega_1,E_1)$ be given. For any $y\in\Omega_2$, the $(E_2)_y$-valued integral
	\begin{equation*}
		F(y) := \int_{\Omega_1} \inner{f(z,y),\phi(z)}_{h_1} e^{-\psi_1(z)}dV_1(z)
	\end{equation*}
	is convergent. Moreover, $F(\cdot)\in A^2(\Omega_2,E_2)$. Given a local trivialization $(U,w,\tilde{e})$ of $E_2\rightarrow\Omega_2$, for any multi-order $\beta\in\NN^m$, we have
	\begin{equation*}
		\pd_w^\beta F(y) = \int_{\Omega_1} \inner{\pd_w^\beta f(z,y),\phi(z)}_{h_1} e^{-\psi_1(z)}dV_1(z).
	\end{equation*}
\end{lemma}

Having Lemma \ref{Lemma:BergmanSlice} and \ref{Lemma:BergmanInt}, it is routine to prove the following product property.

\begin{proposition}
	$A^2(\Omega,E) = A^2(\Omega_1,E_1) \hat{\otimes} A^2(\Omega_2,E_2)$.
\end{proposition}

\begin{definition} \label{Def:mEqual}
	Let $U$ be an open set in $\CC^n$ and let $f,g\in\calO(U)$. Given $x\in U$ and $k\in\NN$, if $\pd^\alpha f(x)=\pd^\alpha g(x)$ for all multi-order $\alpha\in\NN^n$ with $|\alpha|\leqslant k$, then we say ``$f$ coincides with $g$ up to order $k$ at $x$''. If $f$ coincides with the zero function up to order $k$ at $x$, then we say ``$f$ vanishes up to order $k$ at $x$''.
	
	Obviously, $f$ vanishes up to order $k$ at $x$ if and only if $[f]_x\in\frakm_x^{k+1}$, in which $\frakm_x$ is the unique maximal ideal of $\calO_x$. These concepts can be extended to holomorphic sections of vector bundles in an obvious way.
\end{definition}

In the following, let $S_1\subset\Omega_1$ and $S_2\subset\Omega_2$ be arbitrary closed subsets. For $i=1$ (or $2$) and $k\in\NN$, we define
\begin{equation*}
	J_k(S_i) := \{ f\in A^2(\Omega_i,E_i): f\text{ vanishes up to order }k\text{ on }S_i \},
\end{equation*}
then $J_k(S_i)$ is a closed subspace of $A^2(\Omega_i,E_i)$. For convenience, we let $J_{-1}(S_i) := A^2(\Omega_i,E_i)$. Then there is a sequence of nesting closed subspaces:
\begin{equation*}
	A^2(\Omega_i,E_i)=J_{-1}(S_i)\supset J_0(S_i)\supset J_1(S_i)\supset J_2(S_i)\supset\cdots
\end{equation*}
For each $k\in\NN$, let $H_k(S_i):=J_{k-1}(S_i)\ominus J_k(S_i)$ be the orthogonal complement of $J_k(S_i)$ in $J_{k-1}(S_i)$. Consequently, for each $m\in\NN$, there is an orthogonal decomposition for $A^2(\Omega_i,E_i)$:
\begin{equation} \label{Eq:OrtDecom0}
	A^2(\Omega_i,E_i) = H_0(S_i)\oplus H_1(S_i)\oplus\cdots\oplus H_m(S_i)\oplus J_m(S_i).
\end{equation}

Since $S:=S_1\times S_2$ is a closed subset of $\Omega$, the subspaces $J_k(S)$ and $H_k(S)$ of $A^2(\Omega,E)$ can be defined in similar ways. In particular, for each $m\in\NN$, there is an orthogonal decomposition:
\begin{equation} \label{Eq:OrtDecom1}
	A^2(\Omega,E) = H_0(S)\oplus H_1(S)\oplus\cdots\oplus H_m(S)\oplus J_m(S).
\end{equation}
On the other hand, since $A^2(\Omega,E) = A^2(\Omega_1,E_1)\hat{\otimes}A^2(\Omega_2,E_2)$, we have
\begin{equation} \label{Eq:OrtDecom2}
	A^2(\Omega,E) = (\bigoplus_{0\leqslant p,q\leqslant m} H_p(S_1)\hat{\otimes}H_q(S_2)) \oplus R_m,
\end{equation}
in which $R_m$ denotes the direct sum of all remaining terms:
\begin{equation*}
	(\bigoplus_{0\leqslant p\leqslant m}H_p(S_1)\hat{\otimes}J_m(S_2)) \oplus (\bigoplus_{0\leqslant q\leqslant m}J_m(S_1)\hat{\otimes}H_q(S_2)) \oplus(J_m(S_1)\hat{\otimes}J_m(S_2)).
\end{equation*}

\begin{proposition} \label{Prop:BergmanDecom}
	The orthogonal decomposition \eqref{Eq:OrtDecom1} and \eqref{Eq:OrtDecom2} of $A^2(\Omega,E)$ are compatible, i.e.
	\begin{align*}
		H_k(S) & = \bigoplus_{p+q=k} H_p(S_1)\hat{\otimes}H_q(S_2) \quad (0\leqslant k\leqslant m), \\
		J_m(S) & = (\bigoplus_{\substack{0\leqslant p,q\leqslant m\\p+q\geqslant m+1}} H_p(S_1)\hat{\otimes}H_q(S_2)) \oplus R_m.
	\end{align*}
\end{proposition}

\begin{proof}
	Clearly, if $f\in J_p(S_1)$ and $g\in J_q(S_2)$, then $f\otimes g\in J_{p+q+1}(S)$. Therefore, $H_p(S_1)\hat{\otimes}H_q(S_2)\subset J_{p+q-1}(S)$ and $R_m\subset J_m(S)$.
	Using Lemma \ref{Lemma:BergmanSlice} and \ref{Lemma:BergmanInt}, it is easy to show that, if $f\perp J_p(S_1)$ and $g\perp J_q(S_2)$, then $f\otimes g\perp J_{p+q}(S)$. Therefore, $H_p(S_1)\hat{\otimes}H_q(S_2)\perp J_{p+q}(S)$. Then the conclusion is obvious.
\end{proof}

We would like to prove a product property for minimal $L^2$ extensions. In practice, $S_i$ is a closed submanifold of $\Omega_i$ and we are interested in the minimal $L^2$ extension of $f^{(i)}\in\Gamma(S_i,E_i)$, i.e. the unique element $F^{(i)}$ with minimal norm in $A^2(\Omega_i,E_i)$ such that $F^{(i)}|_{S_i} = f^{(i)}$. Since we need to assume the existence of $L^2$ extensions in advance, we simply assume that $f^{(i)}\in A^2(\Omega_i,E_i)$. In general, $S_i$ is an arbitrary closed subset of $\Omega_i$, and we consider the high-order minimal $L^2$ extensions of $f^{(i)}\in A^2(\Omega_i,E_i)$, i.e. the unique element $F^{(i)}$ with minimal norm in $A^2(\Omega_i,E_i)$ that coincides with $f^{(i)}$ up to order $m_i$ on $S_i$, where $m_i$ is a non-negative integer.

\begin{theorem} \label{Thm:ProdProp}
	Let $f^{(1)}\in A^2(\Omega_1,E_1)$, $f^{(2)}\in A^2(\Omega_2,E_2)$ and $m_1,m_2\in\NN$ be given. For $i=1$ and $2$, let $F^{(i)}$ be the unique element with minimal norm in $A^2(\Omega_i,E_i)$ that coincides with $f^{(i)}$ up order $m_i$ on $S_i$. Let $F$ be the unique element with minimal norm in $A^2(\Omega,E)$ that coincides with $f:=f^{(1)}\otimes f^{(2)}$ up to order $m_1+m_2$ on $S$. Then
	\begin{equation*}
		\|F\|_{A^2(\Omega,E)} \geqslant \|F^{(1)}\|_{A^2(\Omega_1,E_1)} \|F^{(2)}\|_{A^2(\Omega_2,E_2)}.
	\end{equation*}
	Moreover, if $f^{(i)}$ vanishes up to order $m_i-1$ on $S_i$, then
	\begin{equation*}
		F=F^{(1)}\otimes F^{(2)} \quad\text{and}\quad
		\|F\|_{A^2(\Omega,E)} = \|F^{(1)}\|_{A^2(\Omega_1,E_1)} \|F^{(2)}\|_{A^2(\Omega_2,E_2)}.
	\end{equation*}
\end{theorem}

\begin{proof}
	Let $m=m_1+m_2$. With respect to the orthogonal decomposition \eqref{Eq:OrtDecom0}, the holomorphic section $f^{(i)}\in A^2(\Omega_i,E_i)$ can be written as
	\begin{equation*}
		f^{(i)} = f_0^{(i)}+f_1^{(i)}+\cdots+f_{m_1+m_2}^{(i)}+g^{(i)},
	\end{equation*}
	in which $f_p^{(i)}\in H_p(S_i)$ for all $p$ and $g^{(i)}\in J_{m_1+m_2}(S_i)$. Clearly,
	\begin{equation*}
		f = f^{(1)}\otimes f^{(2)} = \sum_{k=0}^{m_1+m_2}\sum_{p+q=k}f_p^{(1)}\otimes f_q^{(2)} + g,
	\end{equation*}
	where $g$ collects all the remaining terms.
	
	By Proposition \ref{Prop:BergmanDecom}, $\sum_{p+q=k}f_p^{(1)}\otimes f_q^{(2)}\in H_k(S)$ for any $0\leqslant k\leqslant m_1+m_2$ and $g\in J_{m_1+m_2}(S)$. By the definitions of $J_\bullet(S_i)$ and $J_\bullet(S)$, it is clear that
	\begin{equation*}
		F^{(i)} = \sum_{p=0}^{m_i} f_p^{(i)} \quad\text{and}\quad F = \sum_{k=0}^{m_1+m_2} \sum_{p+q=k} f_p^{(1)}\otimes f_q^{(2)}.
	\end{equation*}
	Since $\{f_p^{(i)}\}_{p=0}^{m_1+m_2}$ are orthogonal families,
	\begin{multline*}
		\|F\|_{A^2(\Omega,E)}^2 = \sum_{k=0}^{m_1+m_2}\sum_{p+q=k}\|f_p^{(1)}\otimes f_q^{(2)}\|^2 = \sum_{k=0}^{m_1+m_2}\sum_{p+q=k}\|f_p^{(1)}\|^2\|f_q^{(2)}\|^2 \\
		\geqslant \sum_{0\leqslant p\leqslant m_1}\|f_p^{(1)}\|^2\sum_{0\leqslant q\leqslant m_2}\|f_q^{(2)}\|^2 = \|F^{(1)}\|_{A^2(\Omega_1,E_1)}^2 \|F^{(2)}\|_{A^2(\Omega_2,E_2)}^2.
	\end{multline*}
	
	Moreover, if $f^{(i)}$ vanishes up to order $m_i-1$ on $S_i$, then $f_p^{(i)}=0$ for all $p<m_i$. (If $m_i=0$, there are no restriction on $f_\bullet^{(i)}$.) In this case,
	\begin{equation*}
		F^{(i)}=f_{m_i}^{(i)} \quad\text{and}\quad F=f_{m_1}^{(1)}\otimes f_{m_2}^{(2)}=F^{(1)}\otimes F^{(2)}. \qedhere
	\end{equation*}
\end{proof}

\begin{remark}
	If $S_i=\{w_i\}$ are singleton sets and $m_1=m_2=0$, then the last statement of Theorem \ref{Thm:ProdProp} corresponds to the product property for Bergman kernels.
\end{remark}

\section{Optimal \texorpdfstring{$L^2$}{L2} Extension Theorems of Openness Type} \label{Chap:Optimal}

In this section, we prove an optimal $L^2$ extension theorem of openness type on weakly pseudoconvex K\"ahler manifolds, which is the optimal version of Theorem \ref{Thm:CoarseExtThm}. Apart from the existence theorem, the main point of the proof is a concavity for certain minimal $L^2$ integrals  (see \cite{Guan2019}). Moreover, using the product property of minimal $L^2$ extensions, we prove a version of the theorem in alternative way. Finally, we give an application of the optimal estimate.

\subsection{Approach 1: A Concavity for Minimal \texorpdfstring{$L^2$}{L2} Integrals} \label{Sec:LogConcave} \hfill

Let $(\Omega,\omega)$ be a weakly pseudoconvex K\"ahler manifold and $(E,h)$ be a Hermitian holomorphic vector bundle over $\Omega$. Let $\psi<0$ and $\varphi$ be two quasi-psh functions on $\Omega$. Suppose there are continuous real $(1,1)$-forms $\gamma\geqslant0$ and $\rho$ on $\Omega$ such that $\ddbar\psi\geqslant\gamma$, $\ddbar\varphi\geqslant\rho$ and $\sqrt{-1}\Theta(E,h)+(\gamma+\rho)\otimes\Id_E$ is Nakano semi-positive. For each $t\geqslant0$, we define $\Omega_t := \{z\in\Omega:\psi(z)<-t\}$ and
\begin{equation*}
	\calA_t := A^2(\Omega_t,K_\Omega\otimes E; (\det\omega)^{-1}\otimes h,e^{-\varphi}dV_\omega).
\end{equation*}

Let $F\in\calA_0$ be given. For each $t\geqslant0$, let $F_t$ be the unique element with minimal norm in $\calA_t$ such that
\begin{equation*}
	F_t-F|_{\Omega_t} \in \Gamma(\Omega_t,\calO(K_\Omega\otimes E)\otimes\calI(\varphi+\psi)).
\end{equation*}
For convenience, we define $I(t):=\|F_t\|_{\calA_t}^2$. Then $I(t)\geqslant0$ is a decreasing function on $[0,+\infty)$ and $I(t)\leqslant\|F|_{\Omega_t}\|_{\calA_t}^2$.

\begin{proposition}
	(1) $I(t)$ is right continuous and $\lim_{t\rightarrow+\infty} I(t) = 0$.
	
	(2) If $I(s)=0$ for some $s\geqslant0$, then $I(t)\equiv0$ on $[0,+\infty)$.
\end{proposition}

\begin{proof}
	Given $\tau\in[0,+\infty)$, let $\{t_j\}_j$ be a sequence decreasing to $\tau$, then $\Omega_\tau=\cup_j \Omega_{t_j}$ and $\|F_{t_j}\|_{\calA_{t_j}} \leqslant \|F_\tau\|_{\calA_\tau}$. Using Montel's theorem, we may assume that $F_{t_j}$ converges uniformly on any compact subsets of $\Omega_\tau$ to some $F'\in\Gamma(\Omega_\tau,K_\Omega\otimes E)$. Clearly, $F'-F|_{\Omega_\tau} \in \Gamma(\Omega_\tau,\calO(K_\Omega\otimes E)\otimes\calI(\varphi+\psi))$. By Fatou's lemma, the monotonicity of $I(t)$ and the minimality of $F_\tau$, we have
	\begin{equation*}
		\|F_\tau\|_{\calA_\tau}^2 \leqslant \|F'\|_{\calA_\tau}^2 \leqslant \lim_{j\to+\infty} \|F_{t_j}\|_{\calA_{t_j}}^2 = \lim_{t\to\tau^+} \|F_t\|_{\calA_t}^2 \leqslant \|F_\tau\|_{\calA_\tau}^2.
	\end{equation*}
	Therefore, $\lim_{t\to\tau^+}I(t)=I(\tau)$ for any $\tau$, i.e. $I(t)$ is right continuous on $[0,+\infty)$. Moreover, by Lebesgue's dominated convergence theorem,
	\begin{equation*}
		0 \leqslant \lim_{t\rightarrow+\infty} I(t) \leqslant \lim_{t\to+\infty} \int_\Omega \one_{\Omega_t} |F|_{\omega,h}^2e^{-\varphi}dV_\omega = 0.
	\end{equation*}
	
	Next, we assume that $I(s)=0$ for some $s\geqslant0$. In this case, $F_s\equiv0$ and
	\begin{equation*}
		F|_{\Omega_s} \in \Gamma(\Omega_s,\calO(K_\Omega\otimes E)\otimes\calI(\varphi+\psi)).
	\end{equation*}
	Using the strong openness property and H\"older's inequality, it is easy to see that $\calI(\varphi+\psi)_x=\calI(\varphi)_x$ for all $x\notin\Omega_s$.
	Since $\int_\Omega |F|_{\omega,h}^2e^{-\varphi}dV_\omega<+\infty$, it follows that $F\in \Gamma(\Omega,\calO(K_\Omega\otimes E)\otimes\calI(\varphi+ \psi))$, and then $I(t)=0$ for any $t$.
\end{proof}

In the following, we assume that $I(t)>0$ for all $t\in[0,+\infty)$. Using the method of Guan \cite{Guan2019}, we can prove a concavity for $I(t)$. Guan \cite{Guan2019} considered the case of Stein manifolds, but there are no essential difference if we assume that $\Omega$ is a weakly pseudoconvex K\"ahler manifold.

\begin{theorem} \label{Thm:Concave}
	$r \mapsto I(-\log r)$ is a concave increasing function on $(0,1]$.	In particular, $I(0) \leqslant I(t)e^t \leqslant I(s)e^s$ for any $0\leqslant t\leqslant s$.
\end{theorem}

\begin{proof}
	At first, we shall prove that
	\begin{equation} \label{Eq:ConcaveIneq}
		\varliminf_{t\rightarrow0^+} \frac{I(a) - I(a+t)}{t} \geqslant \frac{I(0) - I(a)}{e^a - 1} \quad \text{for all } a>0.
	\end{equation}
	Clearly, it is sufficient to consider the case that $\varliminf_{t\rightarrow0^+} \frac{I(a)-I(a+t)}{t} < +\infty$. We choose a decreasing sequence $\{t_j\}_{j=1}^\infty$ of positive reals so that $t_j\searrow0$ and
	\begin{equation*}
		\varliminf_{t\rightarrow0^+} \frac{I(a) - I(a+t)}{t} = \lim_{j\rightarrow+\infty} \frac{I(a) - I(a+t_j)}{t_j}.
	\end{equation*}
	In particular, there is a constant $C$ so that $(e^{a+t_j}-1) \frac{I(a)-I(a+t_j)}{t_j} \leqslant C$ for all $j$.
	
	For each $j\in\NN_+$, we define
	\begin{equation*}
		v_j(t) := v_{a,t_j}(t) = \begin{cases} -a-t_j/2 & (t\leqslant-a-t_j) \\ (t^2+2(a+t_j)t+a^2)/2t_j & (-a-t_j<t<-a) \\ t & (t\geqslant-a) \end{cases},
	\end{equation*}
	and $\chi_j(t):=1-v_j'(t)$. By the $L^2$ techniques developed by Guan-Zhou \cite{GuanZhou2015} (see Theorem \ref{Thm:BasicL2Existence} with $A=0$ and $c(t)\equiv e^t$), for each $j\in\NN_+$, there exists a holomorphic section $f_j\in\Gamma(\Omega,K_\Omega\otimes E)$ such that
	\begin{equation*}
		f_j|_{\Omega_a}-F_a \in \Gamma(\Omega_a,\calO(K_\Omega\otimes E)\otimes \calI(\varphi+\psi))
	\end{equation*}
	and
	\begin{align*}
		&\, \int_\Omega |f_j-\chi_j(\psi)F_a|_{\omega,h}^2 e^{-\varphi-\psi+v_j(\psi)} dV_\omega \\
		\leqslant &\, \frac{1-e^{-a-t_j/2}}{t_j} \int_{\{-a-t_j<\psi<-a\}} |F_a|_{\omega,h}^2 e^{-\varphi-\psi} dV_\omega.
	\end{align*}
	Since $v_j(t)\geqslant t$ and $\|F_a|_{\Omega_{a+t_j}}\|_{\calA_{a+t_j}}^2 \geqslant \|F_{a+t_j}\|_{\calA_{a+t_j}}^2$, the above inequality implies
	\begin{multline*}
		\int_\Omega |f_j-\chi_j(\psi)F_a|_{\omega,h}^2 e^{-\varphi} dV_\omega
		\leqslant \int_\Omega |f_j-\chi_j(\psi)F_a|_{\omega,h}^2 e^{-\varphi-\psi+v_j(\psi)} dV_\omega \\
		\leqslant \frac{e^{a+t_j}-e^{t_j/2}}{t_j} ( \int_{\{\psi<-a\}} |F_a|_{\omega,h}^2 e^{-\varphi} dV_\omega - \int_{\{\psi\leqslant-a-t_j\}} |F_a|_{\omega,h}^2 e^{-\varphi} dV_\omega ) \\
		\leqslant \frac{e^{a+t_j}-1}{t_j} (I(a)-I(a+t_j)) \leqslant C < +\infty.
	\end{multline*}
	On the other hand,
	\begin{equation*}
		\int_\Omega |\chi_j(\psi)F_a|_{\omega,h}^2 e^{-\varphi} dV_\omega \leqslant \int_{\Omega_a} |F_a|_{\omega,h}^2 e^{-\varphi} dV_\omega < +\infty.
	\end{equation*}
	Then it is clear that $\{f_j\}_{j=1}^\infty$ is a bounded sequence in $\calA_0$. Using Montel's theorem, we may assume that $\{f_j\}_{j=1}^\infty$ converges uniformly on any compact subsets of $\Omega$ to some holomorphic section $f\in\Gamma(\Omega,K_\Omega\otimes E)$. By Lemma \ref{Lemma:Coherent},
	\begin{equation*}
		f|_{\Omega_a}-F_a \in \Gamma(\Omega_a,\calO(K_\Omega\otimes E)\otimes \calI(\varphi+\psi)).
	\end{equation*}
	By the definitions of $\chi_j$ and $v_j$,
	\begin{equation*}
		\lim_{j\to+\infty}\chi_j(t) = \one_{(-\infty,-a)} \quad\text{and}\quad \lim_{j\to+\infty}v_j(t) = v(t): = \begin{cases} -a & (t<-a) \\ t & (t\geqslant-a) \end{cases}.
	\end{equation*}
	Then it follows from Fatou's lemma that
	\begin{align*}
		&\, \int_\Omega |f-\one_{\{\psi<-a\}}F_a|_{\omega,h}^2 e^{-\varphi-\psi+v(\psi)} dV_\omega \\
		\leqslant &\, \varliminf_{j\rightarrow+\infty} \int_\Omega |f_j-\chi_j(\psi)F_a|_{\omega,h}^2 e^{-\varphi-\psi+v_j(\psi)} dV_\omega \\
		\leqslant &\, \lim_{j\rightarrow+\infty} (e^{a+t_j}-1) \tfrac{I(a)-I(a+t_j)}{t_j} \leqslant C.
	\end{align*}
	Since $\int_{\Omega_a}|F_a|_{\omega,h}^2e^{-\varphi}dV_\omega<+\infty$, we knows $f\in\calA_0$.
	
	Notice that, $f|_{\Omega_a}\in \calA_a$ and $f|_{\Omega_a}-F_a \in \Gamma(\Omega_a,\calO(K_\Omega\otimes E)\otimes \calI(\varphi+\psi))$. By the minimality of $F_a\in \calA_a$, we know $\|F_a+\tau(f|_{\Omega_a}-F_a)\|_{\calA_a}^2 \geqslant \|F_a\|_{\calA_a}^2$ for any $\tau\in\CC$.
	Consequently, $(f|_{\Omega_a} - F_a) \perp F_a$ in $\calA_a$ and
	\begin{equation*}
		\|f|_{\Omega_a}\|_{\calA_a}^2 = \|f|_{\Omega_a}-F_a\|_{\calA_a}^2 + \|F_a\|_{\calA_a}^2.
	\end{equation*}
	Therefore,
	\begin{align}
		&\, \int_\Omega |f-\one_{\{\psi<-a\}}F_a|_{\omega,h}^2 e^{-\varphi-\psi+v(\psi)} dV_\omega \notag \\
		= &\, \int_{\{\psi\geqslant-a\}} |f|_{\omega,h}^2 e^{-\varphi} dV_\omega + \int_{\Omega_a} |f-F_a|_{\omega,h}^2 e^{-\varphi-\psi-a} dV_\omega \notag \\
		\geqslant &\, \int_{\{\psi\geqslant-a\}} |f|_{\omega,h}^2 e^{-\varphi} dV_\omega + \int_{\Omega_a} |f-F_a|_{\omega,h}^2 e^{-\varphi} dV_\omega \label{Eq:Ineq1} \\
		= &\, \int_{\{\psi\geqslant-a\}} |f|_{\omega,h}^2 e^{-\varphi} dV_\omega + \int_{\Omega_a} |f|_{\omega,h}^2 e^{-\varphi} dV_\omega - \int_{\Omega_a} |F_a|_{\omega,h}^2 e^{-\varphi} dV_\omega \notag \\
		= &\, \|f\|_{\calA_0}^2 - \|F_a\|_{\calA_a}^2 \geqslant I(0) - I(a) \label{Eq:Ineq2}.
	\end{align}
	In summary, we prove that
	\begin{align*}
		I(0) - I(a) & \leqslant \lim_{j\rightarrow+\infty} (e^{a+t_j}-1) \frac{I(a)-I(a+t_j)}{t_j} \\
		& = (e^a-1) \varliminf_{t\rightarrow0^+} \frac{I(a) - I(a+t)}{t}.
	\end{align*}
	
	Given constants $0<r<r_1<r_2\leqslant 1$, we choose $R,a,t\in\RR_+$ such that
	\begin{equation*}
		-\log r=R+a+t, \quad -\log r_1=R+a, \quad -\log r_2=R.
	\end{equation*}
	In the above arguments, we replace $\Omega$ by $\Omega_R:=\{\psi<-R\}$ and $\psi$ by $\psi_R:=\psi+R$. Since $\{\psi_R<-s\}=\Omega_{R+s}$ $(\forall s\geqslant0)$, the inequality \eqref{Eq:ConcaveIneq} becomes
	\begin{equation} \label{Eq:ConcaveIneq2}
		\varliminf_{t\rightarrow0^+} \frac{I(R+a) - I(R+a+t)}{t} \geqslant \frac{I(R) - I(R+a)}{e^a-1}.
	\end{equation}
	(Notice that, corresponding to $f_j\in\Gamma(\Omega,K_\Omega\otimes E)$, we need to construct holomorphic section satisfying suitable $L^2$ estimate on $\Omega_R$. Although we can not apply Theorem \ref{Thm:BasicL2Existence} directly, the existence of such holomorphic section is guaranteed by Theorem \ref{Thm:L2ExistenceSub}.)
	
	Let $J(s):=I(-\log s)$, then \eqref{Eq:ConcaveIneq2} can be reformulated as
	\begin{equation*}
		\varliminf_{r\rightarrow r_1^-} \frac{J(r_1) - J(r)}{\log r_1 - \log r} \geqslant \frac{J(r_2) - J(r_1)}{r_2/r_1-1}.
	\end{equation*}
	Therefore, for any $0<r_1<r_2\leqslant 1$, we have
	\begin{equation*}
		\varliminf_{r\rightarrow r_1^-} \frac{J(r_1) - J(r)}{r_1 - r}
		= \varliminf_{r\rightarrow r_1^-} \frac{J(r_1) - J(r)}{\log r_1 - \log r} \frac{\log r_1 - \log r}{r_1-r}
		\geqslant \frac{J(r_2) - J(r_1)}{r_2-r_1}.
	\end{equation*}
	Since $J(r)=I(-\log r)$ is lower semi-continuous on $(0,1]$, we conclude that $J(r)$ is a concave function (see \cite[Lemma 4.7]{Guan2019}). By the concavity,
	\begin{equation*}
		J(r_1)/r_1 \geqslant J(r_2)/r_2 \geqslant J(1), \quad 0<r_1\leqslant r_2\leqslant 1,
	\end{equation*}
	i.e. $I(0) \leqslant I(t)e^t \leqslant I(s)e^s$ for any $0\leqslant t\leqslant s$.
\end{proof}

\begin{remark} \label{Rmk:ConcaveLinear}
	We now assume that the concave function $r\mapsto I(-\log r)$ is linear on $(0,1]$. In this case, $I(-\log r)\equiv r I(0)$ and $I(a)\equiv e^{-a}I(0)$. Then \eqref{Eq:ConcaveIneq} becomes an equality. By tracing the proof of \eqref{Eq:ConcaveIneq}, we find that the inequalities at \eqref{Eq:Ineq1} and \eqref{Eq:Ineq2} must be equalities.
	Clearly, the equality at \eqref{Eq:Ineq2} means that $\|f\|_{\calA_0}^2=I(0)$. By the uniqueness of the minimal element, $f\equiv F_0$. Since $e^{-\psi-a}>1$ on $\Omega_a$, the equality at \eqref{Eq:Ineq1} implies that $f|_{\Omega_a}\equiv F_a$. In summary,
	\begin{center}
		\itshape if $I(-\log r)$ is a linear function on $(0,1]$, then $F_a \equiv F_0|_{\Omega_a}$ for any $a>0$.
	\end{center}
	
	We notice that, this necessary condition can be used to prove the equality part of Suita's conjecture, this approach is different from Guan-Zhou \cite{GuanZhou2015} at one key step (see Section \ref{Sec:Suita} for details).
	Independently, Guan-Mi \cite{GuanMi2021} also obtained a necessary condition for the general concavity degenerating to linearity.
\end{remark}

Using Theorem \ref{Thm:Concave}, we obtain the optimal version of Theorem \ref{Thm:CoarseExtThm}.

\begin{proof}[The proof of Theorem \ref{MainThm:Opt}]
	We may assume that $f \notin \Gamma(\Omega_a,\calO(K_\Omega\otimes E)\otimes\calI(\varphi+\psi))$. Otherwise, we can simply take $F\equiv 0$. According to Theorem \ref{Thm:CoarseExtThm}, there exists a holomorphic section $\tilde{f}\in A^2(\Omega,K_\Omega\otimes E; (\det\omega)^{-1}\otimes h,e^{-\varphi}dV_\omega)$ such that
	\begin{equation*}
		\tilde{f}|_{\Omega_a}-f \in \Gamma(\Omega_a,\calO(K_\Omega\otimes E)\otimes\calI(\varphi+\psi)).
	\end{equation*}
	Let $\Omega_t := \{\psi<-t\}$ and $\calA_t := A^2(\Omega_t,K_\Omega\otimes E; (\det\omega)^{-1}\otimes h,e^{-\varphi}dV_\omega)$. For each $t\geqslant 0$, let $F_t$ be the unique element with minimal norm in $\calA_t$ such that
	\begin{equation*}
		F_t-\tilde{f}|_{\Omega_t} \in \Gamma(\Omega_t,\calO(K_\Omega\otimes E)\otimes\calI(\varphi+\psi)).
	\end{equation*}
	We define $I(t)=\|F_t\|_{\calA_t}^2$, then $I(t)>0$ for all $t\geqslant0$. By Theorem \ref{Thm:Concave}, $I(-\log r)$ is a concave increasing function on $(0,1]$ and $I(0)\leqslant e^tI(t)$ for any $t\geqslant 0$. In summary, the holomorphic section $F_0\in\Gamma(\Omega,K_\Omega\otimes E)$ satisfying
	\begin{equation*}
		F_0|_{\Omega_a}-f\in\Gamma(\Omega_a,\calO(K_\Omega\otimes E)\otimes\calI(\varphi+\psi))
	\end{equation*}
	and
	\begin{equation*}
		\int_\Omega |F_0|_{\omega,h}^2 e^{-\varphi}dV_\omega = I(0) \leqslant e^aI(a) \leqslant e^a \int_{\Omega_a} |f|_{\omega,h}^2 e^{-\varphi}dV_\omega. \qedhere
	\end{equation*}
\end{proof}

The example after Theorem \ref{MainThm:Opt} shows that the uniform constant $e^a$ in Theorem \ref{MainThm:Opt} is optimal. In the following, we give a more general example.

\begin{example} \label{Ex:GenOptConst}
	Let $D=\{z\in\CC^n:h(z)<1\}$ be a bounded balanced domain, in which $h:\CC^n\to[0,\infty)$ is upper semi-continuous and homogeneous (i.e. $h(\tau z)=|\tau|h(z)$ for all $\tau\in\CC$ and $z\in\CC^n$). We assume that $D$ is pseudoconvex ($\Leftrightarrow$ $\log h$ is psh), then $C^{-1}|z|\leqslant h(z)\leqslant C|z|$ for some constant $C>0$ and $G_D(\cdot,0)\equiv\log h$. For $r>0$, we write $rD=\{z\in\CC^n:h(z)<r\}$. Given $k\geqslant0$, let $[k]$ be the largest integer so that $[k]\leqslant k$, then $\calI(2(n+k)\log h)_0 = \frakm_0^{[k]+1}$.
	
	Let $M$ be a weakly pseudoconvex K\"ahler manifold of dimension $m$ such that $A^2(M,K_M)\neq\{0\}$. The product manifold $\Omega:=D\times M$ is also weakly pseudo-convex. Let $p_1:\Omega\rightarrow D$ and $p_2:\Omega\rightarrow M$ be the natural projections, then $\psi:=p_1^*(2(n+k)\log h)$ is a psh function on $\Omega$ and $S:=\{\psi=-\infty\} =\{0\}\times M$. Let $a>0$ be a constant, then $\Omega_a := \{\psi<-a\} = rD\times M$ with $r=\exp(\frac{-a}{2(n+k)})$.
	
	Given $f\in A^2(rD)$ and $\eta\in A^2(M,K_M)$, we define
	\begin{equation*}
		F = p_1^*(fdz)\wedge p_2^*(\eta) \in A^2(\Omega_a,K_\Omega),
	\end{equation*}
	where $dz=dz_1\wedge\cdots\wedge dz_n$. By Fubini's theorem,
	\begin{equation*}
		\|F\|_{A^2(\Omega_a,K_\Omega)}^2 = 2^n \|f\|_{A^2(rD)}^2 \|\eta\|_{A^2(M,K_M)}^2.
	\end{equation*}
	Since $rD$ is balanced, there exist homogeneous polynomials $s_i(z)$ of degree $i$ such that $f(z)=\sum_{i=0}^\infty s_i(z)$ with uniform convergence on any compact set of $rD$. Let $f_0(z) = \sum\nolimits_{i\leqslant[k]} s_i(z)$ be the truncated series, then
	\begin{align*}
		\|f_0\|_{A^2(D)}^2 & = \sum\nolimits_{i\leqslant[k]} \|s_i\|_{A^2(D)}^2 = \sum\nolimits_{i\leqslant[k]} r^{-2(n+i)}\|s_i\|_{A^2(rD)}^2 \\
		& \leqslant e^a \sum\nolimits_{i\leqslant[k]} \|s_i\|_{A^2(rD)}^2 \leqslant e^a \|f\|_{A^2(rD)}^2.
	\end{align*}
	In particular, if $k\in\NN$ and $f(z)$ is a homogeneous polynomial of degree $k$, then
	\begin{equation*}
		\|f_0\|_{A^2(D)}^2 = e^a \|f\|_{A^2(rD)}^2.
	\end{equation*}
	We define $F_0 = p_1^*(f_0dz)\wedge p_2^*(\eta) \in A^2(\Omega,K_\Omega)$, then
	\begin{equation*}
		\|F_0\|_{A^2(\Omega,K_\Omega)}^2 = 2^n \|f_0\|_{A^2(D)}^2 \|\eta\|_{A^2(M,K_M)}^2.
	\end{equation*}
	
	Using Fubini's theorem and the fact that $\psi(z,w)=\log|z|^{2(n+k)}+O(1)$, we know
	\begin{itemize}
		\item[(i)] $F_0|_{\Omega_a}-F \in \Gamma(\Omega_a,\calO(K_\Omega)\otimes\calI(\psi))$;
		\item[(ii)] if $G\in A^2(\Omega,K_\Omega)\cap\Gamma(\Omega,\calO(K_\Omega)\otimes\calI(\psi))$, then $F_0\perp G$ in $A^2(\Omega,K_\Omega)$.
	\end{itemize}
	Assume that $F'\in A^2(\Omega,K_\Omega)$ is another holomorphic section such that
	\begin{equation*}
		F'|_{\Omega_a}-F \in \Gamma(\Omega_a,\calO(K_\Omega)\otimes\calI(\psi)).
	\end{equation*}
	By (i), we have $F'-F_0\in A^2(\Omega,K_\Omega)\cap\Gamma(\Omega,\calO(K_\Omega)\otimes\calI(\psi))$. By (ii), we know $(F'-F_0)\perp F_0$ in $A^2(\Omega,K_\Omega)$, and then
	\begin{equation*}
		\|F'\|_{A^2(\Omega,K_\Omega)}^2 = \|F_0\|_{A^2(\Omega,K_\Omega)}^2 + \|F'-F_0\|_{A^2(\Omega,K_\Omega)}^2 \geqslant \|F_0\|_{A^2(\Omega,K_\Omega)}^2.
	\end{equation*}
	Therefore, $F_0$ is the unique element with minimal norm in $A^2(\Omega,K_\Omega)$ such that
	\begin{equation*}
		F_0|_{\Omega_a}-F \in \Gamma(\Omega_a,\calO(K_\Omega)\otimes\calI(\psi)).
	\end{equation*}
	We have prove that $\|F_0\|_{A^2(\Omega,K_\Omega)}^2 \leqslant e^a\|F\|_{A^2(\Omega_a,K_\Omega)}^2$. Moreover, if $k\in\NN$ and $f(z)$ is a homogeneous polynomial of degree $k$, then $\|F_0\|_{A^2(\Omega,K_\Omega)}^2 = e^a\|F\|_{A^2(\Omega_a,K_\Omega)}^2$. Therefore, the uniform constant $e^a$ in Theorem \ref{MainThm:Opt} is optimal.
\end{example}

\subsection{Application: The Equality Part of Suita's Conjecture} \label{Sec:Suita} \hfill

Let $\Omega$ be an open Riemann surface admitting Green functions $G_\Omega$. Let $(V,w)$ be a connected coordinate chart of $\Omega$. Let $\kappa_\Omega(w) = B_\Omega(w)|dw|^2$ and $c_\beta(w)|dw|$ be the Bergman kernel form and the logarithmic capacity of $\Omega$, i.e.
\begin{gather*}
	\kappa_\Omega(z_0) := \sup\left\{ \sqrt{-1}F(z_0)\wedge\overline{F(z_0)}: F\in\Gamma(\Omega,K_\Omega), \int_\Omega \tfrac{\sqrt{-1}}{2} F\wedge\overline{F} \leqslant 1 \right\}, \\
	c_\beta(z_0) := \lim_{z\to z_0}\exp(G_\Omega(z,z_0)-\log|w(z)-w(z_0)|).
\end{gather*}
Suita \cite{Suita1972} conjectured that $\pi B_\Omega(z_0) \geqslant c_\beta(z_0)^2$, and the equality holds if and only if $\Omega$ is conformally equivalent to $\DD$ less a possible closed polar set. The inequality part of the conjecture was solved by Blocki \cite{Blocki2013} and Guan-Zhou \cite{GuanZhou2012}, and the equality part was solved by Guan-Zhou \cite{GuanZhou2015}. Using Remark \ref{Rmk:ConcaveLinear}, we have a different approach to a key step of proving the equality part of Suita's conjecture.

Let $\psi=2G_\Omega(\cdot,z_0)$. For each $t\geqslant0$, we define $\Omega_t:=\{\psi<-t\}$. Let $\kappa_{\Omega_t}= B_t|dw|^2$ be the Bergman kernel of $\Omega_t$, and $F_t$ be the unique element with minimal norm in $A^2(\Omega_t,K_\Omega)$ such that $F_t(z_0)=dw$, then
\begin{equation*}
	\int_{\Omega_t} \tfrac{\sqrt{-1}}{2} F_t\wedge\overline{F_t} = \frac{1}{B_t(z_0)}.
\end{equation*}
Since $\calI(\psi)_{z_0} = \frakm_{z_0}$, it follows from Theorem \ref{Thm:Concave} that $r\mapsto \frac{1}{B_{-\log r}(z_0)}$ is a concave increasing function and $B_s(z_0)e^{-s}\leqslant B_t(z_0)e^{-t}\leqslant B_\Omega(z_0)$ for any $0\leqslant t\leqslant s$.

After a change of coordinate, we may assume $G_\Omega(\cdot,z_0)|_V = \log|c_\beta(z_0)w|$. If $s\gg1$, then $\Omega_s=\DD(0;c_\beta(z_0)^{-1}e^{-s/2})$ is an open disc in $(V,w)$. For such $s\gg1$, it is clear that $F_s \equiv dw$ and $B_s(z_0) = \pi^{-1}c_\beta(z_0)^2e^s$. Consequently,
\begin{equation*}
	\pi^{-1}c_\beta(z_0)^2 = B_s(z_0)e^{-s} \leqslant B_t(z_0)e^{-t} \leqslant B_\Omega(z_0), \quad 0\leqslant t\ll s.
\end{equation*}
Then we obtain the inequality part of the conjecture (also see \cite{Guan2019Saitoh}). We further assume that $\pi B_\Omega(z_0) = c_\beta(z_0)^2$, then $B_t(z_0)e^{-t} \equiv \pi^{-1}c_\beta(z_0)^2$ and $\frac{1}{B_{-\log r}(z_0)}$ is a linear function of $r\in(0,1]$. By Remark \ref{Rmk:ConcaveLinear}, $F_t \equiv F_0|_{\Omega_t}$ for any $t\geqslant0$. Since $F_s\equiv dw$ for $s\gg 1$, we conclude that $F_0|_{\Omega_s} \equiv dw$ and then $F_0|_V \equiv dw$.

In summary,
\begin{quotation}
	\itshape if $\pi B_\Omega(z_0) = c_\beta(z_0)^2$ and $(V,w)$ is a connected coordinate chart around $z_0$ such that $G_\Omega(\cdot,z_0)|_V = \log|c_\beta(z_0)w|$, then there exists a holomorphic 1-form $F_0\in\Gamma(\Omega,K_\Omega)$ with $F_0|_V \equiv dw$.
\end{quotation}

This fact is a key step in Guan-Zhou's proof to the equality part of the conjecture (see \cite[Lemma 4.21]{GuanZhou2015}). The proof of this fact here is different from Guan-Zhou \cite{GuanZhou2015} and Dong \cite{DongArxiv}. Now, we complete the proof by repeating the arguments of \cite{GuanZhou2015}:

Let $p:\DD\to\Omega$ be a universal covering of $\Omega$ and $g\in\calO(\DD)$ be a holomorphic function so that $\log|g|=p^* G_\Omega(\cdot,z_0)$. Shrinking $V$ if necessary, we may assume that $p$ is biholomorphic on any connected component of $p^{-1}(V)$. Let $U$ be a fixed connected component of $p^{-1}(V)$, let $h=p_*(g|_U)$, then $\log|h|= G_\Omega(\cdot,z_0)$ on $V$. Consequently, $h\equiv cw$ on $V$ for some constant $c$, and there exists an $F\in\Gamma(\Omega,K_\Omega)$ with $F|_V=dh$. Since $p^*F|_U=dg|_U$, by the uniqueness of analytic continuation, we have $p^*F\equiv dg$ on $\DD$.

The fundamental group $\pi_1(\Omega)$ acts holomorphically on $\DD$ and $p\circ\sigma=p$ for any $\sigma\in\pi_1(\Omega)$. Since $\log|g|=p^* G_\Omega(\cdot,z_0)$ and $dg\equiv p^*F$, we have $|\sigma^*g|=|g|$ and $\sigma^*(dg)=dg$. Then it is clear that $\sigma^*g\equiv g$. Consequently, $\hat{g}=p_*g$ is a well-defined holomorphic function on $\Omega$ with $\log|\hat{g}| = G_\Omega(\cdot,z_0)$. By Lemma 4.25 and 4.26 of \cite{GuanZhou2015} (or Theorem 1 of \cite{Minda1987}), we conclude that $\Omega$ is conformally equivalent to $\DD$ less a possible closed polar set.

We remark that our approach is applicable to generalized Suita conjectures with weights (see \cite{GuanZhou2015CN,GuanZhou2015}) and jets (see \cite{BlockiZwonek2018}).

\subsection{Approach 2: A Tensor Power Trick} \hfill

Recall that, Blocki \cite{Blocki2014} obtained the optimal estimate in Theorem \ref{Thm:BlockiExt} by using a tensor power trick that relies on the product property for Bergman kernels. Using Theorem \ref{Thm:ProdProp} as a replacement, we can prove a more general result.

\begin{definition}[see \cite{RashkovskiiSigurdsson2005}]
	Let $M$ be a connected complex manifold and $S$ be a closed analytic subset of $M$. The class $\mathcal{G}_{M,S}$ consists of all negative psh functions $u$ on $M$ satisfying the following conditions:
	
	for every $x\in S$, there exist local generators $w_1,\cdots,w_p$ for the ideal sheaf $\calI_S$ near $x$ and a constant $C$ depending on $u$ such that $u\leqslant\log|w|+C$ near $x$.
	
	Then the \textbf{pluricomplex Green function} $G_{M,S}$ with singularities along $S$ is the upper envelope of all functions in $\mathcal{G}_{M,S}$, i.e. $G_{M,S} = (\sup\{u:u\in\mathcal{G}_{M,S}\})^*$.
	
	Rashkovskii-Sigurdsson \cite{RashkovskiiSigurdsson2005} proved that $G_{M,S} \in \mathcal{G}_{M,S}$ or $G_{M,S}\equiv-\infty$, and such pluricomplex Green functions satisfy a product property.
\end{definition}

\begin{definition} \label{Def:LogTypeSing}
	Let $M$ be a complex manifold and $S$ be a closed analytic subset of $M$. An upper semi-continuous function $\psi:M\to[-\infty,+\infty)$ is said to has \textbf{log-type singularities} along $S$ if $\psi$ satisfies the following conditions:\par
	(1) $S=\{z\in M:\psi(z)=-\infty\}$; \par
	(2) there exists a dense subset $S_0\subset S$, for each $x\in S_0$, there is a coordinate chart $(U,z=(z',z''))$ around $x$ so that $S\cap U=\{z'=0\}$ and $\psi(z)-\log|z'|$ is bounded from above on $U\backslash S$.
\end{definition}

Clearly, if $u\in\mathcal{G}_{M,S}$, then $u$ has log-type singularities along $S$. However, for a function $\psi$ having log-type singularities along $S$, we do not explicitly require any local behaviour of $\psi$ near a singular point $x\in S_{\sing}$.

\begin{theorem} \label{Thm:OptExtWeakForm}
	Let $(\Omega,\omega)$ be a weakly pseudoconvex K\"ahler manifold, $S$ be a closed analytic subset of $\Omega$ and $(E,h)$ be a Hermitian holomorphic vector bundle over $\Omega$. Let $\psi<0$ be a psh function on $\Omega$ having log-type singularities along $S$. Assume that there exists a quasi-psh function $\varphi$ and a continuous real (1,1)-form $\rho$ on $\Omega$ such that
	\begin{equation*}
		\ddbar\varphi\geqslant\rho \quad\text{and}\quad \sqrt{-1}\Theta(E,h)+\rho\otimes\Id_E \geqslant_\Nak 0.
	\end{equation*}
	Let $p$ be the maximal codimension of the irreducible components of $S$. Let $m\in\NN$, $a\in\RR_+$ and $\Omega_a:=\{z\in\Omega:\psi(z)<-a\}$. Then for any holomorphic section $f\in\Gamma(\Omega_a,K_\Omega\otimes E)$ satisfying $\int_{\Omega_a} |f|_{\omega,h}^2 e^{-\varphi}dV_\omega < +\infty$, there exists a holomorphic section $F\in\Gamma(\Omega,K_\Omega\otimes E)$ such that $F$ coincides with $f$ up to order $m$ on $S$ and
	\begin{equation*}
		\int_\Omega |F|_{\omega,h}^2 e^{-\varphi}dV_\omega \leqslant e^{2(p+m)a} \int_{\Omega_a} |f|_{\omega,h}^2 e^{-\varphi}dV_\omega.
	\end{equation*}
\end{theorem}

Given a holomorphic section $u\in\Gamma(\Omega, \calO(K_\Omega\otimes E)\otimes \calI(2(p+m)\psi))$, it is easy to show that $u$ vanishes up to order $m$ on $S_0$, where $S_0$ is the same as Definition \ref{Def:LogTypeSing}. By the density of $S_0\subset S$, $u$ vanishes up to order $m$ along $S$. Therefore, Theorem \ref{Thm:OptExtWeakForm} can be regarded as a variant of Theorem \ref{MainThm:Opt}.

\begin{proof}
	Given a constant $C\geqslant 1$, we denote by $\text{Ext}(C)$ the statement that
	\begin{center}
		\itshape ``Theorem \ref{Thm:OptExtWeakForm} is true if we replace the constant $e^{2(p+m)a}$ by $C\cdot e^{2(p+m)a}$''.
	\end{center}
	Since $\calI(2(p+m)\psi)_x\subset\frakm_x^{m+1}$ for any $x\in S$, by Theorem \ref{Thm:CoarseExtThm}, there exists an $F\in\Gamma(\Omega,K_\Omega\otimes E)$ such that $F$ coincides with $f$ up to order $m$ on $S$ and
	\begin{equation*}
		\int_\Omega |F|_{\omega,h}^2 e^{-\varphi}dV_\omega \leqslant (e^{2(p+m)a} + 1) \int_{\Omega_a} |f|_{\omega,h}^2 e^{-\varphi}dV_\omega.
	\end{equation*}
	Therefore, $\text{Ext}(2)$ is true.
	
	In the following, we assume that $\text{Ext}(C)$ is true for some constant $C>1$.
	
	Let $\tilde{\Omega}=\Omega\times\Omega$ and let $p_1,p_2:\tilde{\Omega}\rightarrow\Omega$ be the natural projections. Then $\tilde{\Omega}$ is a weakly pseudoconvex K\"ahler manifold equipped with a K\"ahler metric $\tilde{\omega}:=p_1^*\omega+p_2^*\omega$ and $\tilde{S}:=S\times S$ is a closed analytic subset in $\tilde{\Omega}$ whose maximal codimension is $2p$.
	Let $\tilde{E}:=p_1^*E\otimes p_2^*E$, $\tilde{h}:=p_1^*h\otimes p_2^*h$, $\tilde{\varphi}:=p_1^*\varphi+p_2^*\varphi$ and $\tilde{\rho}=p_1^*\rho+p_2^*\rho$, then
	\begin{equation*}
		\ddbar\tilde{\varphi}\geqslant\tilde{\rho} \quad\text{and}\quad \sqrt{-1}\Theta(\tilde{E},\tilde{h}) + \tilde{\rho}\otimes\Id_{\tilde{E}} \geqslant_\Nak 0.
	\end{equation*}
	We define $\tilde{\psi}:=\max\{p_1^*\psi,p_2^*\psi\},$ then $\tilde{\psi}<0$ is a psh function on $\tilde{\Omega}$ having log-type singularities on $\tilde{S}$. Indeed, let $S_0\subset S$ be the same as Definition \ref{Def:LogTypeSing}, then $S_0\times S_0$ is a dense subset of $\tilde{S}$, and $\tilde{\psi}$ has the desired behavior near each point of $S_0\times S_0$. Clearly,
	\begin{equation*}
		\tilde{\Omega}_a:=\{(z,w)\in\tilde{\Omega}:\tilde{\psi}(z,w)<-a\}=\Omega_a\times\Omega_a.
	\end{equation*}
	
	By Theorem \ref{Thm:CoarseExtThm}, there exists a holomorphic section $g\in\Bergman$ such that $g$ coincides with $f$ up to order $2m$ on $S$. Let $F_0$ be the unique element with minimal norm in $\Bergman$ that coincides with $g$ up to order $m$ on $S$. Let $\tilde{F}_0$ be the unique element with minimal norm in $\BergmanN$ that coincides with $g\otimes g$ up to order $2m$ on $\tilde{S}$.
	Notice that, the Bergman spaces $\Bergman$ and $\BergmanN$ fit in the setting of Section \ref{Chap:Prod}. According to Theorem \ref{Thm:ProdProp},
	\begin{equation*}
		\|\tilde{F}_0\|_{\BergmanN} \geqslant \|F_0\|_{\Bergman}^2.
	\end{equation*}
	
	Since $\text{Ext}(C)$ is true, there exists a holomorphic section $\tilde{F}\in\Gamma(\tilde{\Omega},K_{\tilde{\Omega}}\otimes\tilde{E})$ such that $\tilde{F}$ coincides with $\tilde{f} := f\otimes f \in\Gamma(\tilde{\Omega}_a, K_{\tilde\Omega}\otimes\tilde{E})$ up to order $2m$ on $\tilde{S}$ and
	\begin{align*}
		\|\tilde{F}\|_{\BergmanN}^2 & \leqslant Ce^{2(2p+2m)a} \int_{\tilde{\Omega}_a} |\tilde{f}|_{\tilde{\omega},\tilde{h}}^2 e^{-\tilde{\varphi}}dV_{\tilde{\omega}} \\
		& = Ce^{4(p+m)a} \big( \int_{\Omega_a} |f|_{\omega,h}^2 e^{-\varphi}dV_\omega \big)^2.
	\end{align*}
	Clearly, $\tilde{F}$ coincides with $g\otimes g$ up to order $2m$ on $\tilde{S}$.
	Since $\tilde{F}_0\in\BergmanN$ is the minimal $L^2$ extension,
	\begin{equation*}
		\|F_0\|^2 \leqslant \|\tilde{F}_0\| \leqslant \|\tilde{F}\| \leqslant \sqrt{C}e^{2(p+m)a} \int_{\Omega_a} |f|_{\omega,h}^2 e^{-\varphi}dV_\omega.
	\end{equation*}
	
	In summary, if $\text{Ext}(C)$ is true, then $\text{Ext}(\sqrt{C})$ is also true. Since $\sqrt[2^N]{2}\to1$ as $N\rightarrow+\infty$, using Montel's theorem, it is easy to conclude that $\text{Ext}(1)$ is true.
\end{proof}

The following corollary generalizes Blocki's Theorem \ref{Thm:BlockiExt}.

\begin{corollary} \label{Cor:ExtPluriGreen2}
	Let $\Omega\subset\CC^n$ be a bounded pseudoconvex domain, $S\subset\Omega$ be a closed submanifold of codimension $p$ and $\varphi$ be a psh function on $\Omega$. Let $G_{\Omega,S}$ be the pluricomplex Green function of $\Omega$ with singularities along $S$. Assume that $G_{\Omega,S}\not\equiv-\infty$. Let $m\in\NN$ and $a\in\RR_+$ be given, let $U=\{G_{\Omega,S}<-a\}$. For any $f\in A^2(U;e^{-\varphi})$, there exists a holomorphic function $F\in A^2(\Omega;e^{-\varphi})$ such that $F$ coincides with $f$ up to order $m$ along $S$ and
	\begin{equation*}
		\int_\Omega |F|^2e^{-\varphi} d\lambda \leqslant e^{2(p+m)a} \int_U |f|^2e^{-\varphi} d\lambda.
	\end{equation*}
\end{corollary}

\subsection{Application: The Case of Strictly Positive Curvature} \hfill

Let $\Omega$ be a weakly pseudoconvex K\"ahler manifold of dimension $n$. Let $(E,h)$ be a Hermitian holomorphic vector bundle on $\Omega$, whose curvature is Nakano semi-positive. Suppose there is a psh function $\psi:\Omega\rightarrow[-\infty,0)$ having a logarithmic pole at $w\in\Omega$: let $(U,z)$ be a coordinate chart so that $z(w)=0$, then $\psi-\log|z|$ is bounded near $w$. Let $dz=dz_1\wedge\cdots\wedge dz_n$ and $ c:=\varliminf_{z\rightarrow0}(\psi(z)-\log|z|)$.

By the optimal $L^2$ extension theorem (see Guan-Zhou \cite{GuanZhou2015} and Zhou-Zhu \cite{ZhouZhu2018}): {\itshape for any $\xi\in E_w$, there exists a holomorphic section $F\in\Gamma(\Omega,K_\Omega\otimes E)$ such that}
\begin{equation}\label{Eq:OptEst}
	F(w)=dz\otimes\xi \quad\text{and}\quad \int_\Omega(\sqrt{-1})^{n^2}F\wedge_h\overline{F} \leqslant \frac{(2\pi)^n}{n!}e^{-2nc}|\xi|_h^2.
\end{equation}
The uniform constant $\frac{(2\pi)^n}{n!}e^{-2nc}$ is optimal in the sense that there exist examples satisfying the above conditions, in which the constant can not be replaced by any smaller one.
Notice that, $(E,h)$ is only assumed to be Nakano \textit{semi-positive}. Clearly, with some stronger curvature conditions, we may obtain a sharper $L^2$ estimate.

Using Theorem \ref{Thm:OptExtWeakForm} (or Theorem \ref{MainThm:Opt}), we can prove the following result.

\begin{theorem} \label{Thm:Sharper}
	Let $\Omega\ni w$, $(E,h)$ and $\psi<0$ be the same as above. Moreover, we assume that
	(1) $(E,h)$ is Griffiths positive at $w$;
	(2) there exists a coordinate chart $(U,z)$ so that $z(w)=0$ and $\psi(z)=c+\log|z|+o(|z|^2)$.
	
	Then there exists a constant $\tau\in(0,1)$ depends on $h$ and $\psi$, for any $\xi\in E_w$, we can find a holomorphic section $F\in\Gamma(\Omega,K_\Omega\otimes E)$ so that $F(w)=dz\otimes\xi$ and
	\begin{equation*}
		\int_\Omega(\sqrt{-1})^{n^2}F\wedge_h\overline{F} \leqslant (1-\tau)\frac{(2\pi)^n}{n!}e^{-2nc}|\xi|_h^2.
	\end{equation*}
\end{theorem}

\begin{proof}
	Shrinking $U$ if necessary, we may assume that $E|_U$ is trivial. Let $(e_1,\ldots,e_r)$ be a holomorphic frame of $E|_U$ so that $h_{\alpha\bar{\beta}}(w)=\delta_{\alpha\beta}$ and $dh_{\alpha\bar{\beta}}(w)=0$, in which $h_{\alpha\bar{\beta}} := \inner{e_\alpha,e_\beta}_h$. With respect to the local trivialization $(U,z,e)$, the curvature components of $(E,h)$ are
	\begin{equation*}
		R_{i\bar{j}\alpha\bar{\beta}} = -\frac{\pd^2 h_{\alpha\bar{\beta}}}{\pd z_i\pd\bar{z}_j} + h^{\alpha'\overline{\beta'}} \frac{\pd h_{\alpha\overline{\beta'}}}{\pd z_i} \frac{\pd h_{\alpha'\bar{\beta}}}{\pd\bar{z}_j} \overset{(\text{at }w)}{=} -\frac{\pd^2 h_{\alpha\bar{\beta}}}{\pd z_i\pd\bar{z}_j}(0).
	\end{equation*}
	Since $(E,h)$ is Griffiths positive at $w=0$, there exists an $\eps>0$ such that
	\begin{equation*}
		\sum\nolimits_{i,j,\alpha,\beta} -\frac{\pd^2h_{\alpha\bar{\beta}}}{\pd z_i\pd\bar{z}_j}(0) a_i\overline{a_j}b_\alpha\overline{b_\beta} \geqslant 2\eps|a|^2|b|^2, \quad a\in\CC^n, b\in\CC^r.
	\end{equation*}
	By the smoothness of $h$, we find a neighborhood $\mathbb{B}^n(0;r_0)\subset U$ of $w=0$ so that
	\begin{equation*}
		\sum\nolimits_{i,j,\alpha,\beta} -\frac{\pd^2h_{\alpha\bar{\beta}}}{\pd z_i\pd\bar{z}_j}(z) a_i\overline{a_j}b_\alpha\overline{b_\beta} \geqslant \eps|a|^2|b|^2, \quad z\in\mathbb{B}^n(0;r_0), a\in\CC^n, b\in\CC^r.
	\end{equation*}
	
	Given $\xi=\sum_\alpha\xi_\alpha\cdot e_\alpha(w)\in E_w$, we define a holomorphic section of $E$ on $U$ by $f(z)=\sum_\alpha \xi_\alpha\cdot e_\alpha(z)$. It is clear that $f(w)=\xi$ and
	\begin{equation*}
		\frac{\pd^2}{\pd z_i\pd\bar{z}_j} \left(-\eps|\xi|_h^2|z|^2-|f|_h^2\right) = \sum\nolimits_{\alpha,\beta} \left(-\eps\delta_{ij}\delta_{\alpha\beta} -\tfrac{\pd^2 h_{\alpha\bar{\beta}}}{\pd z_i\pd\bar{z}_j}\right) \xi_\alpha\overline{\xi_\beta}.
	\end{equation*}
	Since the complex Hessian of $(-\eps|\xi|_h^2|z|^2-|f|_h^2)$ is semi-positive on $\mathbb{B}^n(0;r_0)$, we know that $(-\eps|\xi|_h^2|z|^2-|f|_h^2)$ is a psh function on $\mathbb{B}^n(0;r_0)$. Therefore, by the mean value inequality, for any $r\in(0,r_0)$,
	\begin{equation*}
		-|\xi|_h^2 \leqslant \frac{1}{\textup{Vol}(\mathbb{B}^n(0;r))} \int_{\mathbb{B}^n(0;r)}\left(-\eps|\xi|_h^2|z|^2-|f|_h^2\right)d\lambda.
	\end{equation*}
	By direct computations, the above inequality can be reformulated as
	\begin{equation} \label{Eq:LocalExt}
		\int_{\mathbb{B}^n(0;r)}|f|_h^2d\lambda \leqslant \frac{\pi^n}{n!}r^{2n}(1-\eps\tfrac{n}{n+1}r^2)|\xi|_h^2.
	\end{equation}
	
	Since $\psi(z)=c+\log|z|+o(|z|^2)$, there exists a small ball $\mathbb{B}^n(0;r_1)\Subset\mathbb{B}^n(0;r_0)$ such that
	\begin{equation*}
		-\kappa|z|^2 < \psi(z)-\log|z|-c < \kappa|z|^2, \quad z\in\overline{\mathbb{B}^n(0;r_1)},
	\end{equation*}
	in which $\kappa:=\frac{\eps}{2(n+1)}>0$. In particular,
	\begin{equation*}
		\psi(z) > \log r_1+c-\kappa r_1^2, \quad z\in\pd\mathbb{B}^n(0;r_1).
	\end{equation*}
	We define $a:=-\log r_1-c+\kappa r_1^2$ and $D:=\{z\in\Omega:\psi(z)<-a\}$. Let $D^*$ be the connected component of $D$ that contains $w$, then $D^* \subset \mathbb{B}^n(0;r_1)$.
	
	Let $\tilde{f}=dz\otimes f\in\Gamma(D^*,K_\Omega\otimes E)$. By Theorem \ref{Thm:OptExtWeakForm} (or Theorem \ref{MainThm:Opt}), there is a holomorphic section $F\in\Gamma(\Omega,K_\Omega\otimes E)$ such that $F(w)=\tilde{f}(w)=dz\otimes\xi$ and
	\begin{equation*}
		\int_\Omega(\sqrt{-1})^{n^2}F\wedge_h\overline{F} \leqslant e^{2na} \int_{D^*}(\sqrt{-1})^{n^2}\tilde{f}\wedge_h\overline{\tilde{f}}.
	\end{equation*}
	Since
	\begin{equation*}
		(\sqrt{-1})^{n^2}\tilde{f}\wedge_h\overline{\tilde{f}} = |f|_h^2(\sqrt{-1})^{n^2}dz\wedge\overline{dz} = 2^n|f|_h^2d\lambda,
	\end{equation*}
	it follows from \eqref{Eq:LocalExt} that
	\begin{align*}
		\int_\Omega(\sqrt{-1})^{n^2}F\wedge_h\overline{F} & \leqslant 2^ne^{2na}\int_{\mathbb{B}^n(0;r_1)}|f|_h^2d\lambda \\
		& \leqslant \frac{(2\pi)^n}{n!} e^{-2nc}e^{2n\kappa r_1^2}(1-2n\kappa r_1^2)|\xi|_h^2.
	\end{align*}
	Since $e^{2n\kappa r_1^2}(1-2n\kappa r_1^2)<1$, this completes the proof.
\end{proof}

\begin{remark}
	The technical assumption (2) in Theorem \ref{Thm:Sharper} can be satisfied in many cases: e.g. $\Omega$ is an open Riemann surface admitting Green functions and $\psi=G_\Omega(\cdot,w)$. Let $z$ be any local coordinate of $\Omega$ centered at $w$, since $G_\Omega(z,w)-\log|z|$ is harmonic, there exists a holomorphic function $f$ in a neighborhood of $w$ such that $G_\Omega(z,w)-\log|z| = \textup{Re}\,f(z)$, and then $G_\Omega(z,w) = \log|e^{f(z)}z|$ in a neighborhood of $w$. Clearly, $e^{f(z)}z$ is a local coordinate around $w$ satisfying the assumption (2).
\end{remark}

\begin{remark} \label{Rmk:NakSharper}
	Let $\Omega\ni w$, $(E,h)$ and $\psi<0$ be the same as above, but we do not require the conditions (1) and (2) of Theorem \ref{Thm:Sharper}.
	
	We assume that $(E,h)$ is Nakano positive at some point $x\in\Omega$. Then there exists an open set $W\Subset\Omega\backslash\{w\}$ such that $(E,h)$ is Nakano positive on $W$. We choose a function $\sigma\in C_c^\infty(W)$ such that $\sigma\leqslant 0$ and $\sigma\not\equiv0$. Let $\eps\in\RR_+$ be small enough so that $\sqrt{-1}\Theta(E,h) + \eps\ddbar\sigma\otimes\Id_E$ is Nakano semi-positive on $\Omega$. By the optimal $L^2$ extension theorem, for each $0\neq\xi\in E_w$, there exists a holomorphic section $F\in\Gamma(\Omega,K_\Omega\otimes E)$ such that
	\begin{equation*}
		F(w)=dz\otimes\xi \quad\text{and}\quad \int_\Omega (\sqrt{-1})^{n^2} F\wedge_h\overline{F} e^{-\eps\sigma} \leqslant \frac{(2\pi)^n}{n!}e^{-2nc}|\xi|_h^2e^{-\eps\sigma(w)}.
	\end{equation*}
	Since $\sigma(w)=0$ and $\sigma<0$ on some open subset of $W$, it is clear that
	\begin{equation*}
		\int_\Omega (\sqrt{-1})^{n^2} F\wedge_h\overline{F} < \int_\Omega (\sqrt{-1})^{n^2} F\wedge_h\overline{F} e^{-\eps\sigma} \leqslant \frac{(2\pi)^n}{n!}e^{-2nc}|\xi|_h^2.
	\end{equation*}
	Therefore, if $(E,h)$ is Nakano positive somewhere, then the estimate in \eqref{Eq:OptEst} can be improved. However, ``Nakano positivity'' is stronger than ``Griffiths positivity'', and we have an explicit estimate of $\tau$ in Theorem \ref{Thm:Sharper}, i.e. $\tau=1-e^{2n\kappa r_1^2}(1-2n\kappa r_1^2)$.
\end{remark}

\begin{remark} \label{Rmk:Hosono}
	Using the theory of complex Monge-Amp\`{e}re equations, Hosono \cite{Hosono2019} proved the following results:
	\begin{quotation}
		\itshape Let $\Omega\ni 0$ be a smoothly bounded domain in $\CC$. Assume that there exists a strictly subharmonic function $\rho\in C^\infty(\overline{\Omega})$ such that $\Omega=\{z:\rho(z)<0\}$ and $\rho(0)=-1$. Let $\varphi=-\log(-\rho)$ and $c=\lim_{z\rightarrow 0}(G_\Omega(z,0)-\log|z|)$. Then there exists a holomorphic function $f\in\calO(\Omega)$ such that $f(0)=1$ and
		\begin{equation*}
			\int_\Omega|f(z)|^2e^{-\varphi(z)}d\lambda_z < \pi e^{-2c}.
		\end{equation*}
	\end{quotation}
	The strictly subharmonic function $\varphi$ can be regarded as a positively curved Hermitian metric on a trivial line bundle over $\Omega$, then Hosono's result is a special case of Theorem \ref{Thm:Sharper} or Remark \ref{Rmk:NakSharper}.
\end{remark}

\section{Relations to the Usual \texorpdfstring{$L^2$}{L2} Extension Problem} \label{Chap:OTtype}

In this section, we discuss the relations between the $L^2$ extension problem of openness type and the usual one. Roughly speaking, the limiting cases of $L^2$ extension theorems of openness type are $L^2$ extension theorems of Ohsawa-Takegoshi type. We also discuss Guan-Zhou's approach \cite{GuanZhou2015} to the optimal $L^2$ extension problem, whose main point can be summarized as an $L^2$ existence theorem of openness type (i.e. Theorem \ref{Thm:BasicL2Existence}).

\subsection{A Basic \texorpdfstring{$L^2$}{L2} Existence Theorem} \hfill

Let us recall some auxiliary functions introduced by Guan-Zhou \cite{GuanZhou2015}.

Let $A\in(-\infty,+\infty]$, $a\in(-A,+\infty)$ and $b\in(0,+\infty)$ be given. We define
\begin{equation*}
	v_{a,b}(t) = \int_{-\infty}^t \left( \int_{-\infty}^{\tau_1} \frac{1}{b}\one_{(-a-b,-a)} d\tau_2 \right) d\tau_1
	- \left(a+\frac{b}{2}\right)
\end{equation*}
and
\begin{equation*}
	\chi_{a,b}(t) = 1 - v_{a,b}'(t) = 1 - \int_{-\infty}^t \frac{1}{b}\one_{(-a-b,-a)} d\tau.
\end{equation*}
Then $v_{a,b}$ and $\chi_{a,b}$ are continuous functions on $\RR$ satisfying the following properties:

$\bullet$ $0\leqslant\chi_{a,b}(t)\leqslant1$ and $v_{a,b}(t)\geqslant t$ for all $t\in\RR$; \par
$\bullet$ $\chi_{a,b}(t)\equiv1$ and $v_{a,b}(t)\equiv-a-\frac{b}{2}$ for $t\leqslant-a-b$; \par
$\bullet$ $\chi_{a,b}(t)\equiv0$ and $v_{a,b}(t)\equiv t$ for $t\geqslant-a$.

\noindent Let $c(t)>0$ be a smooth function on $(-\infty,A)$ so that
\begin{equation} \label{Eq:CondForC}
	\varliminf_{t\rightarrow-\infty} c(t)e^{-t} > 0 \quad\text{and}\quad \int_{-\infty}^Ac(t)dt<+\infty.
\end{equation}
Moreover, we assume that
\begin{equation} \label{Eq:IneqForC}
	\left(\int_t^Ac(\tau)d\tau\right)^2 > c(t)\int_t^A\int_{\tau_1}^Ac(\tau_2)d\tau_2d\tau_1 \quad \text{for all } t<A.
\end{equation}
If $A=+\infty$, we need to assume that the double integral in \eqref{Eq:IneqForC} is convergent.

It is easy to show that, if $A<+\infty$ and $c(t)$ is increasing, then \eqref{Eq:IneqForC} holds for any $t<A$. In general, if there exists a $t_0\in[-\infty,A]$ such that $c'\geqslant0$ on $(-\infty,t_0)$, but $c'<0$ and $(\log c)''<0$ on $(t_0,A)$, then \eqref{Eq:IneqForC} holds for any $t<A$ (see \cite[Remark 4.12]{GuanZhou2015}). For example, we may take $c(t)=\frac{e^t}{(1+e^{t/m})^{m+\eps}}$, where $\eps\in\RR_+$ and $m\in\NN_+$.

\begin{theorem} \label{Thm:BasicL2Existence}
	Let $(\Omega,\omega)$ be a weakly pseudoconvex K\"ahler manifold and $(E,h)$ be a Hermitian holomorphic vector bundle over $\Omega$. Let $\psi<A$ and $\varphi$ be quasi-psh functions on $\Omega$. Suppose there are continuous real $(1,1)$-forms $\gamma\geqslant0$ and $\rho$ on $\Omega$ such that
	\begin{equation*}
		\ddbar\psi\geqslant\gamma, \quad \ddbar\varphi\geqslant\rho \quad\text{and}\quad
		\sqrt{-1}\Theta(E,h)+(\gamma+\rho)\otimes\Id_E \geqslant_\Nak 0.
	\end{equation*}
	Let $\Omega_a:=\{\psi<-a\}$ and let $f\in\Gamma(\Omega_a,K_\Omega\otimes E)$ be a holomorphic section so that
	\begin{equation} \label{Eq:L2ExistenceCond1}
		\int_{\{-a-b<\psi<-a\}} |f|_{\omega,h}^2e^{-\varphi} dV_\omega < +\infty,
	\end{equation}
	and
	\begin{equation} \label{Eq:L2ExistenceCond2}
		\int_{D\cap \Omega_a} |f|_{\omega,h}^2e^{-\varphi} dV_\omega < +\infty \quad\text{for any}\quad D\Subset\Omega.
	\end{equation}
	Then there exists a holomorphic section $F\in\Gamma(\Omega,K_\Omega\otimes E)$ such that
	\begin{equation*}
		F|_{\Omega_a}-f \in \Gamma(\Omega_a,\calO(K_\Omega\otimes E) \otimes \calI(\varphi+\psi))
	\end{equation*}
	and
	\begin{align*}
		&\, \int_\Omega |F-\chi_{a,b}(\psi)f|_{\omega,h}^2 e^{-\varphi-\psi}c(v_{a,b}(\psi)) dV_\omega \\
		\leqslant &\, \frac{1}{b} \int_{-a-b/2}^Ac(t)dt \int_{\{-a-b<\psi<-a\}} |f|_{\omega,h}^2e^{-\varphi-\psi} dV_\omega.
	\end{align*}
\end{theorem}

Various forms of Theorem \ref{Thm:BasicL2Existence} already appeared in \cite[etc]{GuanZhou2015, GuanZhou2015Eff, Guan2019, Guan2018arxiv, GuanMi2021}, and here we consider the case of weakly pseudoconvex K\"ahler manifolds. The proof of the theorem is standard: the main idea is the same as the case of Stein manifolds, but we use Theorem \ref{Thm:EquiSingAppro} to approximate the quasi-psh functions, and use Theorem \ref{Thm:L2ExistenceErrTerm} to deal with the loss of positivity. We will not repeat such a proof.

\begin{remark} \label{Rmk:Coarse}
	We assume the setting of Theorem \ref{Thm:BasicL2Existence}, and let $A=0$, $b=1$, $c(t)\equiv e^t$. Assume that $f\in\Gamma(\Omega_a,K_\Omega\otimes E)$ is a holomorphic section such that $\int_{\Omega_a} |f|_{\omega,h}^2e^{-\varphi} dV_\omega<+\infty$. Since $c(v_{a,b}(t))e^{-t} = e^{v_{a,b}(t)-t} \geqslant 1$, it follows from Theorem \ref{Thm:BasicL2Existence} that there exists a holomorphic section $F\in\Gamma(\Omega,K_\Omega\otimes E)$ such that
	\begin{gather*}
		F|_{\Omega_a}-f \in \Gamma(\Omega_a, \calO(K_\Omega\otimes E)\otimes \calI(\varphi+\psi)), \\
		\int_\Omega |F-\chi_{a,b}(\psi)f|_{\omega,h}^2 e^{-\varphi} dV_\omega
		\leqslant (1-e^{-a-\frac{1}{2}}) \int_{\{-a-1<\psi<-a\}} |f|_{\omega,h}^2e^{-\varphi-\psi} dV_\omega.
	\end{gather*}
	By direct computations, one has the following estimate:
	\begin{align*}
		\int_\Omega |F|_{\omega,h}^2 e^{-\varphi} dV_\omega & \leqslant \int_\Omega 2\left(|F-\chi_{a,b}(\psi)f|_{\omega,h}^2 + |\chi_{a,b}(\psi)f|_{\omega,h}^2 \right) e^{-\varphi} dV_\omega \\
		& \leqslant 2e^{a+1} \int_{\Omega_a} |f|_{\omega,h}^2 e^{-\varphi} dV_\omega.
	\end{align*}
	Therefore, using Theorem \ref{Thm:BasicL2Existence} directly, we can prove a coarse version of Theorem \ref{MainThm:Opt}, where the uniform constant is $2e^{a+1}$. In Theorem \ref{Thm:CoarseExtThm}, we modify the proof of Theorem \ref{Thm:BasicL2Existence} by choosing different auxiliary functions, then the uniform constant becomes $e^a+1$, which is asymptotically optimal as $a\rightarrow+\infty$.
\end{remark}

\begin{lemma} \label{lemma:ctAppro}
	Let $\tilde{c}(t)>0$ be a smooth increasing function on $(-\infty,A)$ so that
	\begin{equation*}
		\varliminf_{t\rightarrow-\infty} \tilde{c}(t)e^{-t} > 0 \quad\text{and}\quad \int_{-\infty}^{-R}\tilde{c}(t)dt<+\infty,
	\end{equation*}
	in which $R>-A$ is a constant. Then there exists a sequence $\{c_m\}_{m=1}^\infty$ of positive smooth functions on $(-\infty,A)$ such that $c_m(t)\equiv \tilde{c}(t)$ on $(-\infty,-R]$,
	\begin{gather*}
		\lim_{m\to+\infty} \int_{-R}^A c_m(t)dt = 0, \\
		\left(\int_t^A c_m(\tau)d\tau\right)^2 > c_m(t) \int_t^A\int_{\tau_1}^A c_m(\tau_2)d\tau_2d\tau_1
		\quad\text{for all}\quad t<A.
	\end{gather*}
\end{lemma}

\begin{proof}
	By translation, we may assume that $R=0$. For each $m\geqslant1$, let
	\begin{equation*}
		\tilde{c}_m(t) := \begin{cases}
			\tilde{c}(t) & (-\infty<t\leqslant \eps_m) \\ \tilde{c}(\eps_m) e^{-m(t-\eps_m)^2} & (\eps_m<t<A)
		\end{cases},
	\end{equation*}
	where $0<\eps_m\ll\frac{1}{m}$ is a constant. Clearly, $\int_0^A \tilde{c}_m(t)dt \to 0$ as $m\to+\infty$ and
	\begin{equation*}
		\int_t^A \int_{\tau_1}^A \tilde{c}_m(\tau_2) d\tau_2 d\tau_1 = \int_t^A (\tau_2-t)\tilde{c}_m(\tau_2) d\tau_2 < +\infty.
	\end{equation*}
	
	For convenience, we define
	\begin{equation*}
		J_m(t) := \frac{(\int_t^A \tilde{c}_m(\tau) d\tau)^2}{\tilde{c}_m(t)} - \int_t^A \int_{\tau_1}^A \tilde{c}_m(\tau_2) d\tau_2 d\tau_1.
	\end{equation*}
	By direct computations, for $t\neq\eps_m$,
	\begin{equation*}
		J_m'(t) = - \int_t^A \tilde{c}_m(\tau) d\tau - \frac{\tilde{c}_m'(t)(\int_t^A \tilde{c}_m(\tau) d\tau)^2}{\tilde{c}_m(t)^2}.
	\end{equation*}
	Since $\tilde{c}_m'(t)=\tilde{c}'(t)\geqslant0$ on $(-\infty,\eps_m)$, we know $J_m'(t)<0$ on $(-\infty,\eps_m)$. On the other hand, $\tilde{c}_m'(t)<0$ on $(\eps_m,A)$. For $t>\eps_m$, we write
	\begin{equation*}
		J_m'(t) = - \frac{\tilde{c}_m'(t) \int_t^A \tilde{c}_m(\tau) d\tau}{\tilde{c}_m(t)^2} K_m(t),
		\text{ where }
		K_m(t) := \frac{\tilde{c}_m(t)^2}{\tilde{c}_m'(t)} + \int_t^A \tilde{c}_m(\tau) d\tau.
	\end{equation*}
	Direct computations show that
	\begin{equation*}
		K_m'(t) = \tilde{c}_m(t) - \frac{\tilde{c}_m''(t)\tilde{c}_m(t)^2}{\tilde{c}_m'(t)^2} = \tilde{c}(\eps_m) \frac{e^{-m(t-\eps_m)^2}}{2m(t-\eps_m)^2} > 0, \quad t>\eps_m.
	\end{equation*}
	Since $\tilde{c}_m'(t)<0$ on $(\eps_m,A)$ and $\varlimsup_{t\to A} K_m(t) = \varlimsup_{t\to A} {\tilde{c}_m(t)^2}/{\tilde{c}_m'(t)} \leqslant 0$, it is clear that $K_m(t)<0$ and $J_m'(t)<0$ for $t\in(\eps_m,A)$.
	
	Since $\varliminf_{t\to A}J_m(t)\geqslant0$ and $J_m'(t)<0$ on $(-\infty,\eps_m)\cup(\eps_m,A)$, we conclude that $J_m(t)>0$ for all $t\in(-\infty,A)$. Equivalently,
	\begin{equation*}
		\big(\int_t^A \tilde{c}_m(\tau)d\tau\big)^2 - \tilde{c}_m(t) \int_t^A\int_{\tau_1}^A \tilde{c}_m(\tau_2)d\tau_2d\tau_1 > 0 \quad\text{for all}\quad t<A.
	\end{equation*}
	By the continuity of $\tilde{c}_m$, the left hand side has a positive lower bound for $t\in[0,2\eps_m]$.
	Let $c_m(t)>0$ be a smooth function on $(-\infty,A)$ so that $\sup|c_m(t)-\tilde{c}_m(t)|<\delta_m$ and $c_m(t)\equiv\tilde{c}_m(t)$ on $(-\infty,0]\cup[2\eps_m,A)$. We can choose $0<\delta_m\ll\frac{1}{m}$ so that
	\begin{equation*}
		\big(\int_t^A c_m(\tau)d\tau\big)^2 - c_m(t) \int_t^A\int_{\tau_1}^A c_m(\tau_2)d\tau_2d\tau_1 > 0 \quad \text{for all } t\in[0,2\eps_m].
	\end{equation*}
	Since $c_m(t) \equiv \tilde{c}(t)$ is increasing on $(-\infty,0]$, the same inequality holds for all $t<0$. It is easy to check that $c_m(t)$ satisfies all the requirements.
\end{proof}

\begin{theorem} \label{Thm:L2ExistenceSub}
	Let $(\Omega,\omega)$, $(E,h)$, $\psi<A$ and $\varphi$ be the same as Theorem \ref{Thm:BasicL2Existence}.
	For each $t\geqslant-A$, let $\Omega_t:=\{\psi<-t\}$. Given $a>R>-A$, assume that $\tilde{c}(t)>0$ is a smooth increasing function on $(-\infty,-R)$ so that
	\begin{equation*}
		\varliminf_{t\rightarrow-\infty} \tilde{c}(t)e^{-t} > 0 \quad\text{and}\quad \int_{-\infty}^{-R}\tilde{c}(t)dt<+\infty.
	\end{equation*}
	Let $f\in\Gamma(\Omega_a,K_\Omega\otimes E)$ be a holomorphic section satisfying \eqref{Eq:L2ExistenceCond1} and \eqref{Eq:L2ExistenceCond2}. Then there exists a holomorphic section $F\in\Gamma(\Omega_R,K_\Omega\otimes E)$ such that
	\begin{equation*}
		F|_{\Omega_a}-f \in \Gamma(\Omega_a,\calO(K_\Omega\otimes E) \otimes \calI(\varphi+\psi))
	\end{equation*}
	and
	\begin{align*}
		&\, \int_{\Omega_R} |F-\chi_{a,b}(\psi)f|_{\omega,h}^2 e^{-\varphi-\psi}\tilde{c}(v_{a,b}(\psi)) dV_\omega \\
		\leqslant &\, \frac{1}{b} \int_{-a-b/2}^{-R}\tilde{c}(t)dt \int_{\{-a-b<\psi<-a\}} |f|_{\omega,h}^2e^{-\varphi-\psi} dV_\omega.
	\end{align*}
\end{theorem}

\begin{proof}
	At first, we assume that $\tilde{c}(t)>0$ is a smooth increasing function on $(-\infty,A)$. Let $\{c_m\}_{m=1}^\infty$ be a sequence of positive smooth functions as Lemma \ref{lemma:ctAppro}.
	By Theorem \ref{Thm:BasicL2Existence}, for each $m\geqslant1$, there exists an $F_m\in\Gamma(\Omega,K_\Omega\otimes E)$ such that
	\begin{equation*}
		F_m|_{\Omega_a}-f \in \Gamma(\Omega_a,\calO(K_\Omega\otimes E) \otimes \calI(\varphi+\psi))
	\end{equation*}
	and
	\begin{align*}
		&\, \int_\Omega |F_m-\chi_{a,b}(\psi)f|_{\omega,h}^2 e^{-\varphi-\psi}c_m(v_{a,b}(\psi)) dV_\omega \\
		\leqslant &\, \frac{1}{b} \int_{-a-b/2}^Ac_m(t)dt \int_{\{-a-b<\psi<-a\}} |f|_{\omega,h}^2e^{-\varphi-\psi} dV_\omega =: Q_m < +\infty.
	\end{align*}
	Since $c_m(t)\equiv\tilde{c}(t)$ on $(-\infty,-R]$ and $v_{a,b}(\psi)<-R$ on $\Omega_R$, we have
	\begin{equation*}
		\int_{\Omega_R} |F_m-\chi_{a,b}(\psi)f|_{\omega,h}^2 e^{-\varphi-\psi} \tilde{c}(v_{a,b}(\psi)) dV_\omega \leqslant Q_m.
	\end{equation*}
	Clearly, $\lim_{m\to+\infty}Q_m = Q := \frac{1}{b} \int_{-a-b/2}^{-R}\tilde{c}(t)dt \int_{\{-a-b<\psi<-a\}} |f|_{\omega,h}^2e^{-\varphi-\psi} dV_\omega$. We may assume that $Q_m\leqslant2Q$ for all $m\geqslant1$.
	
	Notice that, for any relatively compact open set $D\Subset\Omega_R$, we have
	\begin{gather*}
		\int_D |F_m-\chi_{a,b}(\psi)f|_{\omega,h}^2 dV_\omega \leqslant \exp(\sup_D\varphi+\sup_D\psi) \frac{2Q}{\tilde{c}(-a-b/2)} < +\infty, \\
		\int_D |\chi_{a,b}(\psi)f|_{\omega,h}^2 dV_\omega \leqslant \exp(\sup_D\varphi) \int_{D\cap\Omega_a} |f|_{\omega,h}^2e^{-\varphi} dV_\omega < +\infty.
	\end{gather*}
	Therefore, for any $D\Subset\Omega_R$, the integral $\int_D |F_m|_{\omega,h}^2 dV_\omega$ is uniformly bounded. By Montel's theorem, there exists a subsequence of $\{F_m\}_{m=1}^\infty$ that converges uniformly on any compact subsets of $\Omega_R$ to some holomorphic section $F\in\Gamma(\Omega_R,K_\Omega\otimes E)$. It is clear that
	\begin{equation*}
		F|_{\Omega_a}-f \in \Gamma(\Omega_a,\calO(K_\Omega\otimes E) \otimes \calI(\varphi+\psi))
	\end{equation*}
	and
	\begin{equation*}
		\int_{\Omega_R} |F-\chi_{a,b}(\psi)f|_{\omega,h}^2 e^{-\varphi-\psi}\tilde{c}(v_{a,b}(\psi)) dV_\omega \leqslant Q.
	\end{equation*}
	
	If $\tilde{c}(t)$ is only defined on $(-\infty,-R)$, let $R_j=R+j^{-1}$ ($j\gg1$), then the same arguments show that there exist holomorphic sections $F_j\in\Gamma(\Omega_{R_j},K_\Omega\otimes E)$ so that
	\begin{gather*}
		F_j|_{\Omega_a}-f \in \Gamma(\Omega_a,\calO(K_\Omega\otimes E) \otimes \calI(\varphi+\psi)), \\
		\int_{\Omega_{R_j}} |F_j-\chi_{a,b}(\psi)f|_{\omega,h}^2 e^{-\varphi-\psi}\tilde{c}(v_{a,b}(\psi)) dV_\omega \leqslant Q.
	\end{gather*}
	Another application of Montel's theorem yields the desired conclusion.
\end{proof}

In Section \ref{Sec:LogConcave}, we use a special case of Theorem \ref{Thm:BasicL2Existence} to prove a concavity for minimal $L^2$ integrals, from which we obtain the desired optimal estimate. Using Theorem \ref{Thm:BasicL2Existence} and the essential idea of \cite{Guan2018arxiv}, we can prove an optimal $L^2$ extension theorem of openness type that has an extra multiplying term $c(\psi)e^{-\psi}$.

\begin{theorem} \label{Thm:NewOptOpenExt}
	Let $(\Omega,\omega)$ be a weakly pseudoconvex K\"ahler manifold and $(E,h)$ be a Hermitian holomorphic vector bundle over $\Omega$. Let $\psi<A$ and $\varphi$ be quasi-psh functions on $\Omega$, where $A\in(-\infty,+\infty]$ is a constant. Suppose there are continuous real $(1,1)$-forms $\gamma\geqslant0$ and $\rho$ on $\Omega$ so that
	\begin{equation*}
		\ddbar\psi\geqslant\gamma, \quad \ddbar\varphi\geqslant\rho \quad\text{and}\quad \sqrt{-1}\Theta(E,h)+(\gamma+\rho)\otimes\Id_E \geqslant_\Nak 0.
	\end{equation*}
	\indent Let $c(t)$ be a positive smooth function on $(-\infty,A)$ satisfying \eqref{Eq:CondForC} and \eqref{Eq:IneqForC}. Assume that there is a constant $t_0\in(-\infty,A]$ so that $c(t)$ is increasing on $(-\infty,t_0)$. Let $\alpha\geqslant-t_0$ and $\Omega_\alpha:=\{z\in\Omega:\psi(z)<-\alpha\}$. Then for any holomorphic section $F'\in\Gamma(\Omega_\alpha,K_\Omega\otimes E)$ satisfying
	\begin{equation*}
		\int_{\Omega_\alpha} |F'|_{\omega,h}^2e^{-\varphi-\psi}c(\psi) dV_\omega < +\infty,
	\end{equation*}
	there exists a holomorphic section $F\in\Gamma(\Omega,K_\Omega\otimes E)$ such that
	\begin{equation} \label{Eq:NewEqwrtI}
		F|_{\Omega_\alpha}-F'\in\Gamma(\Omega_\alpha, \calO(K_\Omega\otimes E)\otimes\calI(\varphi+\psi))
	\end{equation}
	and
	\begin{equation} \label{Eq:NewOptEst}
		\int_\Omega |F|_{\omega,h}^2 e^{-\varphi-\psi}c(\psi)dV_\omega \leqslant \frac{\int_{-\infty}^Ac(\tau)d\tau}{\int_{-\infty}^{-\alpha}c(\tau)d\tau} \int_{\Omega_\alpha} |F'|_{\omega,h}^2 e^{-\varphi-\psi}c(\psi) dV_\omega.
	\end{equation}
\end{theorem}

The essential idea of the following proof comes from Guan \cite{Guan2018arxiv}, but the original arguments only work for the case that $c(t)$ is increasing on $(-\infty,A)$ (compare the proof of Theorem \ref{Thm:Concave}). Therefore, we need some modifications.

\begin{proof}
	For any $t\geqslant-A$, we define $\Omega_t:=\{z\in\Omega:\psi(z)<-t\}$ and
	\begin{equation*}
		\calA_t := A^2(\Omega_t,K_\Omega\otimes E;(\det\omega)^{-1}\otimes h,e^{-\varphi-\psi}c(\psi)dV_\omega).
	\end{equation*}
	Using Theorem \ref{Thm:BasicL2Existence}, one can find an $L^2$ holomorphic section $\tilde{F}\in\calA_{-A}$ such that $\tilde{F}|_{\Omega_\alpha}-F'\in\Gamma(\Omega_\alpha, \calO(K_\Omega\otimes E)\otimes\calI(\varphi+\psi))$ (compare Remark \ref{Rmk:Coarse}). For each $t\geqslant-A$, let $F_t$ be the unique element with minimal norm in $\calA_t$ such that
	\begin{equation*}
		F_t-\tilde{F}|_{\Omega_t} \in \Gamma(\Omega_t,\calO(K_\Omega\otimes E)\otimes\calI(\varphi+\psi)).
	\end{equation*}
	For convenience, let $I(t):=\|F_t\|_{\calA_t}^2$. We may assume that $I(t)>0$ for all $t$.
	
	Let $a>R\geqslant-A$ be given, where $R\in\{-A\}\cup[-t_0,+\infty)$ and $a\geqslant-t_0$. We shall find a lower bound for $\varliminf_{t\to0^+}\frac{I(a)-I(a+t)}{t}$. It is sufficient to consider the case that $\varliminf_{t\rightarrow0^+} \frac{I(a)-I(a+t)}{t} < +\infty$. We choose a decreasing sequence $\{t_j\}_{j=1}^\infty$ of positive reals so that $t_j\searrow0$ and
	\begin{equation*}
		\varliminf_{t\rightarrow0^+} \frac{I(a) - I(a+t)}{t} = \lim_{j\rightarrow+\infty} \frac{I(a) - I(a+t_j)}{t_j}.
	\end{equation*}
	In particular, there exists a constant $C$ so that $\frac{I(a)-I(a+t_j)}{t_j} \leqslant C$ for all $j$.
	
	We apply Theorem \ref{Thm:BasicL2Existence} (if $R=-A$) or Theorem \ref{Thm:L2ExistenceSub} (if $R\geqslant-t_0$) to $\Omega_R$: for each $j\in\NN_+$, there exists a holomorphic section $f_j\in\Gamma(\Omega_R,K_\Omega\otimes E)$ such that
	\begin{equation*}
		f_j|_{\Omega_a}-F_a \in \Gamma(\Omega_a,\calO(K_\Omega\otimes E)\otimes \calI(\varphi+\psi))
	\end{equation*}
	and
	\begin{align*}
		&\, \int_{\Omega_R} |f_j-\chi_{a,t_j}(\psi)F_a|_{\omega,h}^2 e^{-\varphi-\psi}c(v_{a,t_j}(\psi)) dV_\omega \\
		\leqslant &\, \frac{1}{t_j}\int_{-a-t_j/2}^{-R}c(\tau)d\tau \int_{\{-a-t_j<\psi<-a\}} |F_a|_{\omega,h}^2 e^{-\varphi-\psi} dV_\omega.
	\end{align*}
	{\itshape Notice that, $v_{a,t_j}(t)\geqslant t$ on $(-\infty,-a)$ and $v_{a,t_j}(t)\equiv t$ on $[-a,+\infty)$. Since $-a\leqslant t_0$ and $c(t)$ is increasing on $(-\infty,t_0)$, we know $c(v_{a,t_j}(\psi)) \geqslant c(\psi)$.} Consequently,
	\begin{align*}
		&\, \int_{\Omega_R} |f_j-\chi_{a,t_j}(\psi)F_a|_{\omega,h}^2 e^{-\varphi-\psi}c(\psi) dV_\omega \\
		\leqslant &\, \int_{\Omega_R} |f_j-\chi_{a,t_j}(\psi)F_a|_{\omega,h}^2 e^{-\varphi-\psi}c(v_{a,t_j}(\psi)) dV_\omega \\
		\leqslant &\, \frac{1}{t_j} \frac{\int_{-a-t_j}^{-R}c(\tau)d\tau}{c(-a-t_j)} \int_{\{-a-t_j<\psi<-a\}} |F_a|_{\omega,h}^2 e^{-\varphi-\psi} c(\psi) dV_\omega \\
		\leqslant &\, \frac{\int_{-a-t_j}^{-R}c(\tau)d\tau}{c(-a-t_j)} \frac{I(a)-I(a+t_j)}{t_j} \leqslant C' < +\infty,
	\end{align*}
	where $C':=C\left(\int_{-a-t_1}^{-R}c(\tau)d\tau\right)/c(-a-t_1)<+\infty$. Since
	\begin{equation*}
		\int_{\Omega_R} |\chi_{a,t_j}(\psi)F_a|_{\omega,h}^2 e^{-\varphi-\psi}c(\psi) dV_\omega \leqslant \int_{\Omega_a} |F_a|_{\omega,h}^2 e^{-\varphi-\psi}c(\psi) dV_\omega < +\infty,
	\end{equation*}
	it is clear that $\{f_j\}_{j=1}^\infty$ is a bounded sequence in $\calA_R$. Since $e^{-\varphi-\psi}c(\psi)$ locally has a positive lower bound on $\Omega$, using Montel's theorem, we may assume that $\{f_j\}_{j=1}^\infty$ converges uniformly on any compact subsets of $\Omega_R$ to some holomorphic section $f\in\Gamma(\Omega_R,K_\Omega\otimes E)$. By Lemma \ref{Lemma:Coherent},
	\begin{equation*}
		f|_{\Omega_a}-F_a \in \Gamma(\Omega_a,\calO(K_\Omega\otimes E)\otimes \calI(\varphi+\psi)).
	\end{equation*}
	By the definitions of $\chi_{a,t_j}$ and $v_{a,t_j}$,
	\begin{equation*}
		\lim_{j\to+\infty} \chi_{a,t_j} = \one_{(-\infty,-a)} \quad\text{and}\quad
		\lim_{j\to+\infty} v_{a,t_j}(t) = v_a(t) :=
		\begin{cases} -a & (t<-a) \\ t & (t\geqslant-a) \end{cases}.
	\end{equation*}
	Then it follows from Fatou's lemma that
	\begin{align*}
		&\, \int_{\Omega_R} |f-\one_{\{\psi<-a\}}F_a|_{\omega,h}^2 e^{-\varphi-\psi}c(v_a(\psi)) dV_\omega \\
		\leqslant &\, \varliminf_{j\rightarrow+\infty} \int_{\Omega_R} |f_j-\chi_{a,t_j}(\psi)F_a|_{\omega,h}^2 e^{-\varphi-\psi}c(v_{a,t_j}(\psi)) dV_\omega \\
		\leqslant &\, \frac{\int_{-a}^{-R}c(\tau)d\tau}{c(-a)} \lim_{j\rightarrow+\infty}\frac{I(a)-I(a+t_j)}{t_j} \leqslant C'.
	\end{align*}
	Since $c(v_a(\psi))\geqslant c(\psi)$ and $\int_{\Omega_a}|F_a|_{\omega,h}^2e^{-\varphi-\psi}c(\psi)dV_\omega<+\infty$, we know $f\in\calA_R$.
	
	Notice that, $f|_{\Omega_a}\in \calA_a$ and $f|_{\Omega_a}-F_a \in \Gamma(\Omega_a,\calO(K_\Omega\otimes E)\otimes \calI(\varphi+\psi))$. By the minimality of $F_a\in \calA_a$, it is clear that $(f|_{\Omega_a} - F_a) \perp F_a$ in $\calA_a$ and then
	\begin{equation*}
		\|f|_{\Omega_a}\|_{\calA_a}^2 = \|f|_{\Omega_a}-F_a\|_{\calA_a}^2 + \|F_a\|_{\calA_a}^2.
	\end{equation*}
	For simplicity, we temporary write $d\mu=e^{-\varphi-\psi}c(\psi) dV_\omega$. Then
	\begin{align*}
		&\, \int_{\Omega_R} |f-\one_{\{\psi<-a\}}F_a|_{\omega,h}^2 e^{-\varphi-\psi}c(v_a(\psi)) dV_\omega \\
		= &\, \int_{\{-a\leqslant\psi<-R\}} |f|_{\omega,h}^2 e^{-\varphi-\psi}c(\psi) dV_\omega + \int_{\Omega_a} |f-F_a|_{\omega,h}^2 e^{-\varphi-\psi}c(-a) dV_\omega \\
		\geqslant &\, \int_{\{-a\leqslant\psi<-R\}} |f|_{\omega,h}^2 e^{-\varphi-\psi}c(\psi) dV_\omega + \int_{\Omega_a} |f-F_a|_{\omega,h}^2 e^{-\varphi-\psi}c(\psi) dV_\omega \\
		= &\, \int_{\{-a\leqslant\psi<-R\}} |f|_{\omega,h}^2 d\mu + \int_{\Omega_a} |f|_{\omega,h}^2 d\mu - \int_{\Omega_a} |F_a|_{\omega,h}^2 d\mu \\
		= &\, \|f\|_{\calA_R}^2 - \|F_a\|_{\calA_a}^2 \geqslant I(R) - I(a).
	\end{align*}
	In summary, for any $a>R\geqslant-A$ with $R\in\{-A\}\cup[-t_0,+\infty)$ and $a\geqslant-t_0$, we prove that
	\begin{equation*}
		I(R) - I(a) \leqslant \frac{\int_{-\infty}^{-R}c(\tau)d\tau-\int_{-\infty}^{-a}c(\tau)d\tau}{c(-a)} \varliminf_{t\to0^+}\frac{I(a)-I(a+t)}{t}.
	\end{equation*}
	Consequently,
	\begin{equation} \label{Eq:LowBdDer}
		\frac{I(R) - I(a)}{ \int_{-\infty}^{-R}c(\tau)d\tau-\int_{-\infty}^{-a}c(\tau)d\tau } \leqslant \varliminf_{t\to0^+}\frac{I(a)-I(a+t)}{ \int_{-\infty}^{-a}c(\tau)d\tau-\int_{-\infty}^{-a-t}c(\tau)d\tau }.
	\end{equation}
	
	We define $i_c(t):=\int_{-\infty}^{-t}c(\tau)d\tau$, then $i_c(t)$ is a strictly decreasing function on $[-A,+\infty)$. Let $\kappa:=i_c(-A)=\int_{-\infty}^Ac(\tau)d\tau$ and $\kappa_0:=i_c(-t_0)$. We denote by $i_c^{-1}$ the inverse function of $i_c$, then $i_c^{-1}$ is strictly decreasing on $(0,\kappa]$. Let
	\begin{equation*}
		r_2=i_c(R), \quad r_1=i_c(a) \quad\text{and}\quad r=i_c(a+t),
	\end{equation*}
	then \eqref{Eq:LowBdDer} can be reformulated as
	\begin{equation} \label{Eq:LowBdDer2}
		\frac{I(i_c^{-1}(r_2)) - I(i_c^{-1}(r_1))}{r_2-r_1} \leqslant
		\varliminf_{r\to r_1^-} \frac{I(i_c^{-1}(r_1)) - I(i_c^{-1}(r))}{r_1-r}.
	\end{equation}
	Notice that, we need to assume $R\in\{-A\}\cup[-t_0,+\infty)$ and $a\geqslant-t_0$. Therefore, the inequality \eqref{Eq:LowBdDer2} holds for any $0<r_1<r_2\leqslant\kappa$ with $r_2\in\{\kappa\}\cup(0,\kappa_0]$ and $r_1\leqslant\kappa_0$.
	
	By \eqref{Eq:LowBdDer2} and Lemma 4.7 of \cite{Guan2019}, we conclude that $J(r):=I(i_c^{-1}(r))$ is a positive concave increasing function on $(0,\kappa_0]$. By the concavity and \eqref{Eq:LowBdDer2}, it is clear that
	\begin{equation*}
		\frac{J(\kappa)-J(r_1)}{\kappa-r_1} \leqslant \varliminf_{r\to r_1^-} \frac{J(r_1)-J(r)}{r_1-r} \leqslant \frac{J(r_1)}{r_1}, \quad r_1\leqslant\kappa_0.
	\end{equation*}
	Therefore, $J(\kappa)/\kappa \leqslant J(r_1)/r_1$ for any $r_1\leqslant\kappa_0$. Equivalently,
	\begin{equation*}
		\frac{\int_\Omega|F_{-A}|_{\omega,h}^2 e^{-\varphi-\psi}c(\psi)dV_\omega}{ \int_{-\infty}^Ac(\tau)d\tau } \leqslant \frac{\int_{\Omega_a}|F_a|_{\omega,h}^2 e^{-\varphi-\psi}c(\psi)dV_\omega}{ \int_{-\infty}^{-a}c(\tau)d\tau }, \quad a\geqslant-t_0.
	\end{equation*}
	Clearly, $F_{-A}|_{\Omega_\alpha}-F'\in\Gamma(\Omega_\alpha, \calO(K_\Omega\otimes E)\otimes\calI(\varphi+\psi))$.
	Since $\|F_\alpha\|_{\calA_\alpha}^2 \leqslant \|F'\|_{\calA_\alpha}^2$ and $\alpha\geqslant-t_0$, we find that
	\begin{equation*}
		\int_\Omega|F_{-A}|_{\omega,h}^2 e^{-\varphi-\psi}c(\psi)dV_\omega \leqslant
		\frac{\int_{-\infty}^Ac(\tau)d\tau}{\int_{-\infty}^{-\alpha}c(\tau)d\tau} \int_{\Omega_\alpha}|F'|_{\omega,h}^2 e^{-\varphi-\psi}c(\psi)dV_\omega.
	\end{equation*}
	Therefore, $F_{-A}\in\Gamma(\Omega,K_\Omega\otimes E)$ satisfies all the requirements.
\end{proof}

In Theorem \ref{Thm:NewOptOpenExt}, let $\Omega=\BB^n(0;\exp(\frac{A}{2(n+k)}))$ be a ball in $\CC^n$, $\psi=2(n+k)\log|z|$ and $\varphi\equiv0$, where $k\in\NN$. Let $F'\in\calO(\Omega_\alpha)$ be a homogeneous polynomial of degree $k$, then the minimal $L^2$ extension subjects to \eqref{Eq:NewEqwrtI} is the same polynomial, and the estimate \eqref{Eq:NewOptEst} is optimal (compare Example \ref{Ex:GenOptConst}).

As an application of Theorem \ref{Thm:NewOptOpenExt}, we prove Theorem \ref{MainThm:DemExt}. By the same reason as above, the uniform estimate in Theorem \ref{MainThm:DemExt} is also optimal.

\begin{proof}[The proof of Theorem \ref{MainThm:DemExt}]
	Let $\psi=2(m+p)\log|w|$, then $\psi\in\psh(\Omega)$ and $U=\{\psi<0\}$. Clearly, $\calI(\psi)_x\subset\frakm_x^{m+1}$ for any $x\in S$. We shall apply Theorem \ref{Thm:NewOptOpenExt} with
	\begin{equation*}
		c(t)=\frac{e^t}{(1+e^{t/(m+p)})^{m+p+\eps}}.
	\end{equation*}
	By direct computations,
	\begin{equation*}
		(\log c)' = \frac{m+p-\eps e^{t/(m+p)}}{(m+p)(1+e^{t/(m+p)})}, \quad
		(\log c)'' = \frac{-(m+p+\eps)e^{t/(m+p)}}{(m+p)^2(1+e^{t/(m+p)})^2}.
	\end{equation*}
	Let $t_0:=(m+p)\log(\eps^{-1}(m+p))$, then $(\log c)'>0$ on $(-\infty,t_0)$, while $(\log c)'<0$ and $(\log c)''<0$ on $(t_0,A)$. Therefore, $c(t)$ satisfies the conditions \eqref{Eq:CondForC} and \eqref{Eq:IneqForC} (see the paragraph after \eqref{Eq:IneqForC}). Notice that, $t_0>0$ and
	\begin{align*}
		\int_{-\infty}^{+\infty}c(t)dt & = \int_0^{+\infty} \frac{(m+p)s^{m+p-1}}{(1+s)^{m+p+\eps}}ds \\
		& = \int_0^1 (m+p)\tau^{m+p-1}(1-\tau)^{\eps-1}d\tau.
	\end{align*}
	There is a similar computation for $\int_{-\infty}^0c(t)dt$. Then the theorem follows from a direct application of Theorem \ref{Thm:NewOptOpenExt}.
\end{proof}

\subsection{The Optimal \texorpdfstring{$L^2$}{L2} Extension Theorem of Guan-Zhou} \hfill

The essential idea of Theorem \ref{Thm:BasicL2Existence} is contained in Guan-Zhou's article \cite{GuanZhou2015} on optimal $L^2$ extensions. Indeed, let $a\rightarrow+\infty$ in Theorem \ref{Thm:BasicL2Existence}, we can easily obtain again the optimal $L^2$ extension theorem of Guan-Zhou. Similarly, let $a\rightarrow+\infty$ in Theorem \ref{Thm:CoarseExtThm}, we can prove a slightly different optimal $L^2$ extension theorem.

In this section, let $M$ be a complex manifold of dimension $n$, $S\subset M$ be a closed analytic subset and $(E,h)$ be a Hermitian holomorphic vector bundle over $M$.

\begin{definition} \label{Def:PolarFun}
	Given $A\in(-\infty,+\infty]$, we denote by $\#_A(M,S)$ the set of all upper semi-continuous functions $\psi:M\rightarrow[-\infty,A)$ satisfying the following conditions:
	
	(1) $\psi^{-1}(-\infty) = \{z\in M:\psi(z)=-\infty\}$ is a closed set containing $S$;
	
	(2) for any regular point $w\in S_{\reg}$, if $S$ is $(n-p)$-dimensional around $w$, then there exists a coordinate chart $(U,z=(z_1,\ldots,z_n))$ around $w$ such that $S\cap U=\{z_1=\cdots=z_p=0\}$ and
	\begin{equation*}
		\sup\nolimits_{U\backslash S}|\psi(z)-\log(|z_1|^2+\cdots+|z_p|^2)^p|<+\infty.
	\end{equation*}
	Clearly, if $\psi\in\#_A(M,S)$, then $\calI(\psi)_x\subset\frakm_x$ for any $x\in S$.
\end{definition}

Let $\psi\in\#_A(M,S)$, $b\in\RR_+$ and $dV_M$ be a continuous volume form on $M$. Following \cite{OhsawaExtV,GuanZhou2015}, we define a positive measure $dV_M^b[\psi]$ on $S_{\reg}$ as the minimal element of the partially ordered set of all positive measures $d\mu$ satisfying
\begin{equation} \label{Def:Mes1} \begin{aligned}
		\int_{S_{\reg}} g d\mu \geqslant \varlimsup_{a\rightarrow+\infty} \frac{1}{b} \int_{\{-a-b<\psi<-a\}}ge^{-\psi}dV_M, \qquad\qquad \\
		0\leqslant g\in C_c^0(M\backslash S_{\sing}).
\end{aligned}\end{equation}
{\itshape For simplicity, we drop the coefficients from the original definition!}
Similarly, we let $dV_M^*[\psi]$ be the minimal positive measure on $S_{\reg}$ such that
\begin{equation} \label{Def:Mes2}
	\int_{S_{\reg}} g dV_M^*[\psi] \geqslant \varlimsup_{a\rightarrow+\infty} e^a \int_{\{\psi<-a\}}gdV_M,
	\quad 0\leqslant g\in C_c^0(M\backslash S_{\sing}).
\end{equation}

Assume that there exists a closed subset $X\subset M$ such that
\begin{enumerate}
	\item[(a)] $X$ is locally negligible with respect to $L^2$ holomorphic functions, i.e. for any open set $U\subset M$, one has $A^2(U\setminus X)=A^2(U)$;
	\item[(b)] $M\backslash X$ is a Stein manifold that intersects with every component of $S$ and $S_{\sing}\subset X$.
\end{enumerate}
In particular, if $M$ is a Stein manifold / a projective manifold / a projective family, then we can find a closed subset $X\subset M$ satisfying the conditions (a) and (b).

Let $\omega$ be a continuous Hermitian metric on $M$ and let $dV_\omega=\frac{\omega^n}{n!}$ be the associated volume form. Let $c(t)>0$ be a smooth function on $(-\infty,A)$ satisfying the conditions \eqref{Eq:CondForC} and \eqref{Eq:IneqForC}. We further assume that
\begin{enumerate}
	\item[(i)] $\psi\in \psh(M)\cap C^\infty(M\backslash(S\cup X))$;
	\item[(ii)] $(E,he^{-\psi})$ is Nakano semi-positive on $M\backslash(S\cup X)$.
\end{enumerate}
Let $a\rightarrow+\infty$ in Theorem \ref{Thm:BasicL2Existence}, we recover the Main Theorem 2 of Guan-Zhou \cite{GuanZhou2015} (compare the proof Theorem \ref{Thm:NewMesExt} and \ref{Thm:OptJetExt}).

\begin{theorem}[see \cite{GuanZhou2015}] \label{Thm:GuanZhouExt}
	Let $(M,S,X)$, $(E,h)$, $\psi$, $\omega$, $c(t)$ be the same as above, and we define a positive measure $dV_\omega^b[\psi]$ on $S_{\reg}$ by \eqref{Def:Mes1}. For any holomorphic section $f\in\Gamma(S,(K_M\otimes E)|_S)$ satisfying $\int_{S_{\reg}} |f|_{\omega,h}^2 dV_\omega^b[\psi] < +\infty$,
	there exists a holomorphic section $F\in\Gamma(M,K_M\otimes E)$ such that $F|_S=f$ and
	\begin{equation*}
		\int_M c(\psi)|F|_{\omega,h}^2e^{-\psi}dV_\omega \leqslant \int_{-\infty}^Ac(t)dt \int_{S_{\reg}} |f|_{\omega,h}^2 dV_\omega^b[\psi].
	\end{equation*}
\end{theorem}

Similarly, let $a\rightarrow+\infty$ in Theorem \ref{Thm:CoarseExtThm}, we can prove a slightly different optimal $L^2$ extension theorem. Notice that, the asymptotically optimal estimate of Theorem \ref{Thm:CoarseExtThm} is enough to do that!

\begin{theorem}[Assume that $A=0$] \label{Thm:NewMesExt}
	Let $(M,S,X)$, $(E,h)$, $\psi$, $\omega$ be the same as above, and we define a positive measure $dV_\omega^*[\psi]$ on $S_{\reg}$ by \eqref{Def:Mes2}. Then for any holomorphic section $f\in\Gamma(S,(K_M\otimes E)|_S)$ satisfying $\int_{S_{\reg}} |f|_{\omega,h}^2 dV_\omega^*[\psi] < +\infty$,
	there exists a holomorphic section $F\in\Gamma(M,K_M\otimes E)$ such that $F|_S=f$ and
	\begin{equation*}
		\int_M |F|_{\omega,h}^2 dV_\omega \leqslant \int_{S_{\reg}}|f|_{\omega,h}^2dV_\omega^*[\psi].
	\end{equation*}
\end{theorem}

\begin{proof}
	Let $\Omega=M\backslash X$ and $S'=S\backslash X$, then $\Omega$ is Stein and $S'\subset\Omega$ is smooth. Since $\Omega$ is a Stein manifold, there exists a holomorphic section $\tilde{f}\in\Gamma(\Omega,K_M\otimes E)$ such that $\tilde{f}=f$ on $S'$. We choose a sequence $\{D^k\}_{k=0}^\infty$ of relatively compact Stein domains in $\Omega$ such that $D^k\Subset D^{k+1}\Subset\Omega$ and $\cup_kD^k=\Omega$.
	
	Let $k\in\NN$ be fixed for the moment. Clearly, $\tilde{f}\in A^2(D^k,K_M\otimes E)$. According to Theorem \ref{Thm:CoarseExtThm}, for each $a>0$, there exists an $F_{k,a}\in\Gamma(D^k,K_M\otimes E)$ such that $F_{k,a}=f$ on $S'\cap D^k$ and
	\begin{equation*}
		\int_{D^k}|F_{k,a}|_{\omega,h}^2dV_\omega \leqslant (e^a+1)\int_{\{z\in D^k:\psi(z)<-a\}}|\tilde{f}|_{\omega,h}^2dV_\omega.
	\end{equation*}
	Let $0\leqslant\eta\leqslant1$ be a compactly supported smooth function on $\Omega$ with $\eta|_{D^k}\equiv 1$, then
	\begin{align*}
		&\, \varlimsup_{a\rightarrow+\infty} (e^a+1)\int_{\{\psi<-a\}\cap D^k}|\tilde{f}|_{\omega,h}^2dV_\omega \\
		\leqslant &\, \varlimsup_{a\rightarrow+\infty} (e^a+1)\int_{\{\psi<-a\}}\eta|\tilde{f}|_{\omega,h}^2dV_\omega
		= \varlimsup_{a\rightarrow+\infty} e^a\int_{\{\psi<-a\}}\eta|\tilde{f}|_{\omega,h}^2dV_\omega \\
		\leqslant &\, \int_{S_{\reg}}\eta|\tilde{f}|_{\omega,h}^2dV_\omega^*[\psi] \leqslant \int_{S_{\reg}}|f|_{\omega,h}^2dV_\omega^*[\psi] < +\infty.
	\end{align*}
	Let $a\to+\infty$ and then $k\to+\infty$, using Montel's theorem twice, we find a holomorphic section $F\in\Gamma(\Omega,K_\Omega\otimes E)$ such that $F=f$ on $S'$ and
	\begin{equation*}
		\int_{\Omega}|F|_{\omega,h}^2dV_\omega \leqslant \int_{S_{\reg}}|f|_{\omega,h}^2dV_\omega^*[\psi].
	\end{equation*}
	Since $X$ is locally negligible with respect to $L^2$ holomorphic functions, $F$ extends to the whole manifold $M$ as a holomorphic section of $K_M\otimes E$. Clearly, $F|_S=f$ and $F$ satisfies the desired estimate on $M$.
\end{proof}

There is a few words of warning: we do not have a continuous real $(1,1)$-form $\gamma\geqslant0$ on $\Omega=M\backslash X$ such that $\ddbar\psi\geqslant\gamma$ and $\sqrt{-1}\Theta(E,h)+\gamma\otimes\Id_E \geqslant_\Nak 0$. However, since $\psi\in C^\infty(\Omega\backslash S')$, we do not need the approximation procedure for $\psi$, then the condition (ii) is exactly what we need to complete the proofs of Theorem \ref{Thm:CoarseExtThm} and \ref{Thm:BasicL2Existence}. Therefore, both theorems are available in our situation.

\begin{remark}
	Given $\psi\in\#_\infty(M,S)$ and $b\in\RR_+$, it is easy to show that
	\begin{equation*}
		dV_M^b[\psi] \leqslant \frac{e^b}{b} dV_M^*[\psi] \quad\text{and}\quad dV_M^*[\psi] \leqslant \frac{b}{1-e^{-b}} dV_M^b[\psi].
	\end{equation*}
	
	The computations of $dV_M^b[\psi]$ and $dV_M^*[\psi]$ are local. Therefore, we may assume that $M=\BB^p\times\BB^{n-p}$, $S=\{0\}\times\BB^{n-p}$ and $C:=\sup_{M\backslash S}|\psi(z)-\log|z'|^{2p}|<+\infty$. Here, we denote the coordinates of $\CC^n=\CC^p\times\CC^{n-p}$ by $z=(z',z'')$.
	
	Let $dV_M$ and $dV_S$ be continuous volume forms on $M$ and $S$, then $dV_M=ud\lambda_z$ and $dV_S=vd\lambda_{z''}$ for some $u\in C^0(M)$ and $v\in C^0(S)$. Let $h(z'') = (h_{ij}(z''))_{i,j}$ be a family of positive definite $(p\times p)$-Hermitian matrices so that $z''\mapsto h_{ij}(z'')$ are measurable functions on $S$. For each $z''\in\BB^{n-p}$, we define
	\begin{align*}
		\eta(z'') := \varliminf_{z'\rightarrow0} (\psi(z',z'') - \log|z'|_{h(z'')}^{2p}), \\
		\gamma(z'') := \varlimsup_{z'\rightarrow0} (\psi(z',z'') - \log|z'|_{h(z'')}^{2p}),
	\end{align*}
	in which $|z'|_{h(z'')}^2 := \sum_{i,j=1}^p h_{ij}(z'')z_i\overline{z_j}$.
	
	Using Fubini's theorem and Fatou's lemma, it is easy to show that
	\begin{equation*}
		dV_M^b[\psi] \leqslant \frac{\pi^p}{p!} \frac{e^{-\eta}}{\det h} \frac{(b+\gamma-\eta)u}{b v} dV_S, \quad
		dV_M^*[\psi] \leqslant \frac{\pi^p}{p!} \frac{e^{-\eta}}{\det h} \frac{u}{v} dV_S.
	\end{equation*}
	Moreover, if $\eta\equiv\gamma$ on $S$, it follows from the dominated convergence theorem that
	\begin{equation} \label{Eq:CompMes}
		dV_M^b[\psi] = dV_M^*[\psi] = \frac{\pi^p}{p!} \frac{e^{-\gamma}}{\det h} \frac{u}{v} dV_S.
	\end{equation}
\end{remark}

\begin{remark}
	Let $M$ be a complex manifold of dimension $n$ and $(E,h)$ be a Hermitian holomorphic vector bundle of rank $p$ $(p\leqslant n)$. Assume that $\sigma\in\Gamma(M,E)$ is a holomorphic section that transverses with the zero section, then $S:=\sigma^{-1}(0)$ is a closed submanifold of codimension $p$. Let $\omega$ be a continuous Hermitian metric on $M$, let $dV_M=\frac{\omega^n}{n!}$ and $dV_S=\frac{(\omega|_S)^{n-p}}{(n-p)!}$ be the associated volume forms on $M$ and $S$. Clearly, $\psi:=\log|\sigma|_h^{2p} \in\#_\infty(M,S)$.   We shall compute $dV_M^b[\psi]$ and $dV_M^*[\psi]$.
	
	Let $(U,(z_1,\ldots,z_n),(e_1,\ldots,e_p))$ be a local trivialization of $E\rightarrow M$ such that $\sigma|_U = \sum_{i=1}^p z_i\cdot e_i$. Assume that $\omega|_U = \sum_{i,j} \sqrt{-1}\omega_{ij}dz_i\wedge d\bar{z}_j$ and $h_{ij} = \inner{e_i,e_j}_h$. We write $A=(\omega_{ij})_{1\leqslant i,j\leqslant n}$, $B=(\omega_{ij})_{p+1\leqslant i,j\leqslant n}$ and $H=(h_{ij})_{1\leqslant i,j\leqslant p}$.
	
	Since $\psi = p\log\sum_{i,j}h_{ij}z_i\overline{z_j}$, it is clear that
	\begin{equation*}
		\lim_{z'\rightarrow0} \left( \psi(z',z'') - p\log\sum\nolimits_{i,j=1}^p h_{ij}(0,z'')z_i\overline{z_j} \right) = 0, \quad z''\in S\cap U.
	\end{equation*}
	On $S\cap U$, $\wedge^p(d\sigma) = dz_1\wedge\cdots\wedge dz_p\otimes e_1\wedge\cdots\wedge e_p$ is well-defined, and
	\begin{equation*}
		|dz_1 \wedge\cdots\wedge dz_p|_\omega^2 = \frac{\det B}{\det A} \quad\Rightarrow\quad
		|\wedge^p(d\sigma)|_{\omega,h}^2 = \frac{\det B\det H}{\det A}.
	\end{equation*}
	Since $dV_M = 2^n\det A\,d\lambda_z$ and $dV_S = 2^{n-p}\det B\,d\lambda_{z''}$, it follows from \eqref{Eq:CompMes} that
	\begin{equation*}
		dV_M^b[\psi] = dV_M^*[\psi] = \frac{\pi^p}{p!} \frac{1}{\det H} \frac{2^n\det A}{2^{n-p} \det B} dV_S = \frac{(2\pi)^p}{p!} \frac{dV_S}{|\wedge^p(d\sigma)|_{\omega,h}^2}.
	\end{equation*}
	Since $\wedge^p(d\sigma) \in \Gamma(S,(\wedge^{p,0}T^*M\otimes \det E)|_S)$ is globally defined, we conclude that
	\begin{equation*}
		dV_M^b[\psi] = dV_M^*[\psi] = \frac{(2\pi)^p}{p!} \frac{dV_S}{|\wedge^p(d\sigma)|_{\omega,h}^2}.
	\end{equation*}
\end{remark}

\subsection{An Optimal \texorpdfstring{$L^2$}{L2} Extension Theorem for Homogeneous Jets} \hfill

Using the method of Berndtsson-Lempert \cite{BerndtssonLempert2016}, Hosono \cite{Hosono2020} proved an optimal $L^2$ extension theorem for homogeneous jets on Stein manifolds. Indeed, using Theorem \ref{Thm:BasicL2Existence}, we can prove a general result. For related results, we refer the reader to \cite{DemaillyNonRed} and \cite{ZhouZhu}.

In this section, $\Omega$ is a complex manifold of dimension $n$ and $S\subset\Omega$ is a closed submanifold of codimension $p$. Let $\calI_S\subset\calO_\Omega$ be the ideal sheaf associated to $S$, then $\calI_S^m$ is a coherent analytic sheaf for any $m\in\NN_+$. Let $\calI_S^0:=\calO_\Omega$. For each $m\in\NN$, we define $\calJ^{(m)}:=\calI_S^m/\calI_S^{m+1}$, then $\calJ^{(m)}|_S$ is a locally free $\calO_S$-module. We denote by $J^{(m)}$ the associated holomorphic vector bundle over $S$:

Let $(U,z=(z',z''))$ be a coordinate chart of $\Omega$ such that $S\cap U=\{z'=0\}$, where $z'=(z_1,\ldots,z_p)$ and $z''=(z_{p+1},\ldots,z_n)$. Then $z''$ is a coordinate system on $S\cap U$ and $\{\mathbf{z}'_\alpha = z_1^{\alpha_1}\cdots z_p^{\alpha_p}:|\alpha|=m\}$ is a holomorphic frame of $J^{(m)}|_{S\cap U}$. Let $(V,w=(w',w''))$ be another coordinate chart of $\Omega$ so that $S\cap V=\{w'=0\}$. On $S\cap(U\cap V)$, for any $\beta\in\NN^p$ with $|\beta|=m$, we have
\begin{equation*}
	w_1^{\beta_1}\cdots w_p^{\beta_p} = (a_{11}z_1+\cdots+a_{p1}z_p)^{\beta_1}\cdots(a_{1p}z_1+\cdots+a_{pp}z_p)^{\beta_p} \mod \calI_S^{m+1},
\end{equation*}
in which $a_{ik}(z''):=\frac{\pd w_k}{\pd z_i}(0,z'')$ are holomorphic functions on $S\cap(U\cap V)$. Therefore, the transition matrix between $\{\mathbf{z}'_\alpha:|\alpha|=m\}$ and $\{\mathbf{w}'_\beta:|\beta|=m\}$ is holomorphic on $S\cap(U\cap V)$.

Let $(E,h)$ be a Hermitian holomorphic vector bundle over $\Omega$. Given a local holomorphic section $F\in\Gamma(V,\calO(E)\otimes\calI_S^m)$, for any $x\in S\cap V$, the germ $[F]_x$ determines a vector of $(E|_S\otimes J^{(m)})_x$:

Let $(U,(z',z''),(e_1,\ldots,e_r))$ be a local trivialization of $E\to\Omega$ so that $U\subset V$, $S\cap U=\{z'=0\}$ and $U\cong\BB^p\times\BB^{n-p}$. We write $F|_U=\sum_{i=1}^{r}F_i\cdot e_i$, then the holomorphic function $F_i\in\Gamma(U,\calI_S^m)$ admits a Taylor expansion on $U$,
\begin{equation*}
	F_i(z',z'') = \sum\nolimits_{|\alpha|\geqslant m}c_{i,\alpha}(z'')z_1^{\alpha_1}\cdots z_p^{\alpha_p} \quad\text{with}\quad c_{i,\alpha}(z''):=\tfrac{1}{\alpha!}\pd_{z'}^\alpha F_i(0,z'').
\end{equation*}
Therefore, for any $x=(0,z'')\in S\cap U$, the germ $[F]_x$ naturally determines a vector
\begin{equation*}
	\sum\nolimits_{i=1}^r\sum\nolimits_{|\alpha|=m} c_{i,\alpha}(z'') \cdot e_i(x)\otimes \mathbf{z}'_\alpha(x) \in (E|_S\otimes J^{(m)})_x.
\end{equation*}
This vector is independent of the choice of $(U,z,e)$. For simplicity, we identify the germ $[F]_x$ with the associated vector of $(E|_S\otimes J^{(m)})_x$. Moreover, since $c_{i,\alpha}(z'')$ are holomorphic functions, $x\mapsto[F]_x$ is a holomorphic section of $E|_S\otimes J^{(m)}$.\\

In the following, $\omega$ is a continuous Hermitian metric on $\Omega$, $dV_\Omega$ and $dV_S$ are the induced volume forms on $\Omega$ and $S$. Moreover, we assume that there exists an upper semi-continuous function $\psi$ on $\Omega$ satisfying the following condition:
\begin{quotation}
	\itshape $S=\{\psi=-\infty\}$; for any $x\in S$, there exists a coordinate chart $(U,z=(z',z''))$ around $x$ such that $S\cap U=\{z\in U:z'=0\}$ and $\sup_{U\setminus S}|\psi(z)-\log|z'|| < +\infty$. \hfill \upshape ($\sharp$)
\end{quotation}
Associated to $h,\omega$ and $\psi$, one can define a ``\textbf{singular metric}'' on $E|_S\otimes J^{(m)}$:

Let $\xi\in\Gamma(D,E|_S\otimes J^{(m)})$ be a local holomorphic section of $E|_S\otimes J^{(m)}$. Let $(U,(z',z''),(e_1,\ldots,e_r))$ be a local trivialization of $E$ so that $S\cap U=\{z'=0\}$. We assume that $S\cap U\subset D$ and $U\cong\BB^p\times\BB^{n-p}$. We take a holomorphic section $F\in\Gamma(U,\calO(E)\otimes\calI_S^m)$ such that $[F]_x = \xi(x)$ for any $x\in S\cap U$. Indeed, if
\begin{equation*}
	\xi(z'') = \sum\nolimits_{i=1}^r \sum\nolimits_{|\alpha|=m} c_{i,\alpha}(z'') \cdot e_i(z'')\otimes \mathbf{z}'_\alpha(z'')
\end{equation*}
with $c_{i,\alpha}\in \calO(S\cap U)$, then we may take
\begin{equation*}
	F(z) = \sum\nolimits_{i=1}^r \sum\nolimits_{|\alpha|=m} c_{i,\alpha}(z'')z_1^{\alpha_1}\cdots z_p^{\alpha_p} \cdot e_i(z).
\end{equation*}
Given a constant $b>0$, for any $0\leqslant\rho\in C_c^0(D)$ with $\supp\rho\subset S\cap U$, we define
\begin{equation}\begin{aligned}\label{Eq:DefHermJm}
		& \int_D \rho|\xi|_{E|_S\otimes J^{(m)}}^2 dV_S := \\
		& \qquad\qquad \varlimsup_{a\rightarrow+\infty} \frac{1}{b} \int_{\{-a-b<2(m+p)\psi<-a\}\cap U} \tilde{\rho}|F|_h^2 e^{-2(m+p)\psi} dV_\Omega,
\end{aligned}\end{equation}
where $0\leqslant\tilde{\rho}\in C_c^0(U)$ is a continuous extension of $\rho$. One can check that this definition is independent of the choices of $F$ and $\tilde{\rho}$. For general $0\leqslant\rho\in C_c^0(D)$, we define $\int_D \rho|\xi|_{E|_S\otimes J^{(m)}}^2 dV_S$ by using a partition of unity. (Also see \cite{DemaillyNonRed})

The ``singular metric'' $|\cdot|_{E|_S\otimes J^{(m)}}$ defined above is just a formal notation. If $(E,h)$ is trivial, we simply write $|\cdot|_{J^{(m)}}$.

\begin{remark}
	In certain cases, we do have an explicit formula for $|\cdot|_{E|_S\otimes J^{(m)}}$.
	Assume there exists a coordinate chart $(U,(z',z''))$ such that $S\cap U=\{z'=0\}$ and $\psi|_U=\log|z'|+\gamma$ for some continuous function $\gamma\in C^0(U)$. We write $dV_\Omega=ud\lambda_z$ and $dV_S=vd\lambda_{z''}$, where $u\in C^0(U)$ and $v\in C^0(S\cap U)$ are continuous functions. Recall that, $\{\mathbf{z}'_\alpha:|\alpha|=m\}$ is a holomorphic frame of $J^{(m)}|_{S\cap U}$. By careful computations, we can show that $|\cdot|_{J^{(m)}}$ is a Hermitian metric on $J^{(m)}$ given by
	\begin{equation*}
		\inner{\mathbf{z}'_\alpha(x), \mathbf{z}'_\beta(x)}_{J^{(m)}}
		= \frac{e^{-2(m+p)\gamma(x)}}{2(m+p)}\frac{u(x)}{v(x)} \int_{\mathbb{S}^{2p-1}} z_1^{\alpha_1}\cdots z_p^{\alpha_p} \overline{z_1^{\beta_1}\cdots z_p^{\beta_p}} dS(z'),
	\end{equation*}
	where $dS$ is the standard volume form on the unit sphere $\mathbb{S}^{2p-1}\subset\CC^p$. Moreover, the metric $|\cdot|_{E|_S\otimes J^{(m)}}$ is induced by $h|_S$ and $|\cdot|_{J^{(m)}}$.
\end{remark}

Using Theorem \ref{Thm:BasicL2Existence}, we can prove an optimal $L^2$ extension theorem for homogeneous jets on Stein manifolds.

\begin{theorem} \label{Thm:OptJetExt}
	Let $\Omega$ be a Stein manifold of dimension $n$, $S\subset\Omega$ be a closed submanifold of codimension $p$ and $(E,h)$ be a Hermitian holomorphic vector bundle over $\Omega$. Let $J^{(m)}$ be the holomorphic vector bundle determined by $(\calI_S^m/\calI_S^{m+1})|_S$. Let $\omega$ be a continuous Hermitian metric on $\Omega$, let $dV_\Omega$ and $dV_S$ be the induced volume forms on $\Omega$ and $S$. Suppose there exists a psh function $\psi$ on $\Omega$ satisfying the condition $(\sharp)$ and we define a singular metric $|\cdot|$ on $(K_\Omega\otimes E)|_S\otimes J^{(m)}$ as \eqref{Eq:DefHermJm}.
	
	Assume that $\Psi:=2(m+p)\psi<A$, where $A\in (-\infty,+\infty]$ is a constant.	Let $c(t)>0$ be a smooth function on $(-\infty,A)$ satisfying the conditions \eqref{Eq:CondForC} and \eqref{Eq:IneqForC}. Let $\varphi$ be a quasi-psh function on $\Omega$. Assume that there are continuous real (1,1)-forms $\gamma\geqslant0$ and $\rho$ so that
	\begin{equation*}
		\ddbar\Psi\geqslant\gamma, \quad \ddbar\varphi\geqslant\rho \quad\text{and}\quad \sqrt{-1}\Theta(E,h)+(\gamma+\rho)\otimes\Id_E \geqslant_\Nak 0.
	\end{equation*}
	Then for any holomorphic section $\xi\in\Gamma(S,(K_\Omega\otimes E)|_S\otimes J^{(m)})$ satisfying
	\begin{equation*}
		\int_S |\xi|^2e^{-\varphi} dV_S <+\infty,
	\end{equation*}
	there exists a holomorphic extension $F\in\Gamma(\Omega,\calO(K_\Omega\otimes E)\otimes\calI_S^m)$ of $\xi$ such that
	\begin{equation*}
		\int_\Omega c(\Psi)|F|_{\omega,h}^2e^{-\varphi-\Psi} dV_\Omega \leqslant \int_{-\infty}^Ac(t)dt \int_S |\xi|^2e^{-\varphi} dV_S.
	\end{equation*}
\end{theorem}

\begin{proof}
	For convenience, let $E':=K_\Omega\otimes E$. Since
	\begin{equation*}
		0 \rightarrow \calO(E')\otimes\calI_S^{m+1} \rightarrow \calO(E')\otimes\calI_S^{m} \rightarrow \calO(E')\otimes\calJ^{(m)} \rightarrow 0
	\end{equation*}
	is a short exact sequence of coherent analytic sheaves, then we have a long exact sequence of cohomology groups:
	\begin{equation*}
		\cdots \rightarrow \Gamma(\Omega, \calO(E')\otimes\calI_S^m) \rightarrow \Gamma(\Omega, \calO(E')\otimes\calJ^{(m)}) \rightarrow H^1(\Omega, \calO(E')\otimes\calI_S^{m+1}) \rightarrow \cdots
	\end{equation*}
	Since $\Omega$ is Stein, $H^1(\Omega, \calO(E')\otimes\calI_S^{m+1})=0$. As a consequence, the restriction map
	\begin{equation*}
		\Gamma(\Omega, \calO(E')\otimes\calI_S^m) \rightarrow \Gamma(S, E'|_S\otimes J^{(m)}) \cong \Gamma(\Omega, \calO(E')\otimes\calJ^{(m)})
	\end{equation*}
	is surjective. In particular, there exists a holomorphic section $f\in\Gamma(\Omega,\calO(E')\otimes\calI_S^m)$ such that $[f]_x = \xi(x) \in (E'|_S\otimes J^{(m)})_x$ for any $x\in S$.
	
	Since $\Omega$ is a Stein manifold, by Theorem \ref{Thm:SteinAppro}, there exists a decreasing sequence $\{\varphi_\nu\}_{\nu=1}^\infty$ of smooth functions on $\Omega$ such that $\ddbar\varphi_\nu\geqslant\rho$. Using Montel's theorem, we may simply assume that $\varphi$ is smooth. Let $\{D^k\}_{k=1}^\infty$ be a sequence of Stein domains in $\Omega$ so that $D^k\Subset D^{k+1}\Subset\Omega$ and $\cup_k D^k=\Omega$. Clearly, $\int_{D^k}|f|_{\omega,h}^2e^{-\varphi}dV_\Omega<+\infty$ for any $k\in\NN_+$.
	
	Let $b\in\RR_+$ be given. We apply Theorem \ref{Thm:BasicL2Existence} to $D^k$, $E$, $\Psi<A$ and $\varphi$: for each $k\in\NN_+$ and $a>-A$, there exists a holomorphic section $F_{k,a}\in\Gamma(D^k,K_\Omega\otimes E)$ such that $F_{k,a}-f \in \Gamma(D^k, \calO(K_\Omega\otimes E)\otimes \calI(\Psi))$ and
	\begin{align*}
		&\, \int_{D^k} |F_{k,a}-\chi_{a,b}(\Psi)f|_{\omega,h}^2 e^{-\varphi-\Psi}c(v_{a,b}(\Psi)) dV_\Omega \\
		\leqslant &\, \frac{1}{b} \int_{-a-b/2}^Ac(t)dt \int_{\{-a-b<\Psi<-a\}\cap D^k} |f|_{\omega,h}^2e^{-\varphi-\Psi} dV_\Omega.
	\end{align*}
	Let $0\leqslant\eta\leqslant1$ be a compactly supported smooth function on $\Omega$ with $\eta|_{D^k}\equiv 1$, then
	\begin{align*}
		&\, \varlimsup_{a\rightarrow+\infty} \frac{1}{b} \int_{\{-a-b<\Psi<-a\}\cap D^k} |f|_{\omega,h}^2e^{-\varphi-\Psi} dV_\Omega \\
		\leqslant &\, \varlimsup_{a\rightarrow+\infty} \frac{1}{b} \int_{\{-a-b<\Psi<-a\}} \eta|f|_{\omega,h}^2e^{-\varphi-\Psi} dV_\Omega \\
		= &\, \int_S \eta(x)|[f]_x|^2e^{-\varphi(x)} dV_S(x) \leqslant \int_S|\xi|^2e^{-\varphi}dV_S < +\infty.
	\end{align*}
	
	Since $\varliminf_{t\to-\infty}c(t)e^{-t}>0$, we find that $c(v_{a,b}(\Psi))e^{-\Psi} \geqslant c(v_{a,b}(\Psi))e^{-v_{a,b}(\Psi)}$ is bounded below on $D^k\Subset\Omega$ by a positive constant independent of $a$. Therefore, $\int_{D^k}|F_{k,a}|_{\omega,h}^2 dV_\omega$ is uniformly bounded as $a\to+\infty$. By Montel's theorem, there exists a sequence $\{a_j\}_{j=1}^\infty$ such that $a_j\to+\infty$ and $F_{k,a_j}$ converges uniformly on any compact subsets of $D^k$ to some $F_k\in\Gamma(D^k,K_\Omega\otimes E)$. Then it is clear that $F_k-f \in \Gamma(D^k, \calO(K_\Omega\otimes E)\otimes \calI(\Psi))$ and
	\begin{equation*}
		\int_{D_k} c(\Psi)|F_k|_{\omega,h}^2e^{-\varphi-\Psi} dV_\Omega \leqslant \int_{-\infty}^Ac(t)dt \int_S |\xi|^2e^{-\varphi} dV_S.
	\end{equation*}
	Let $k\rightarrow+\infty$, using Montel's theorem again, we find an $F\in\Gamma(\Omega,K_\Omega\otimes E)$ such that $F-f \in \Gamma(\Omega, \calO(K_\Omega\otimes E)\otimes \calI(\Psi))$ and
	\begin{equation*}
		\int_\Omega c(\Psi)|F|_{\omega,h}^2e^{-\varphi-\Psi} dV_\Omega \leqslant \int_{-\infty}^Ac(t)dt \int_S |\xi|^2e^{-\varphi} dV_S.
	\end{equation*}
	Since $\psi$ satisfies the condition ($\sharp$), $\calI(\Psi)_x \subset \mathfrak{m}_x^{m+1}$ for any $x\in S$. Then it is clear that $F\in\Gamma(\Omega, \calO(K_\Omega\otimes E)\otimes\calI_S^m)$ and
	\begin{equation*}
		[F]_x = [f]_x = \xi(x) \in ((K_\Omega\otimes E)|_S\otimes J^{(m)})_x, \quad x\in S. \qedhere
	\end{equation*}
\end{proof}

\begin{remark}
	If $m=0$, then Theorem \ref{Thm:OptJetExt} reduces to Theorem \ref{Thm:GuanZhouExt}. Moreover, Hosono's \cite{Hosono2020} result corresponds to the case of $A=0$ and $c(t)=\exp(\frac{t}{m+p})$. Similar to Theorem \ref{Thm:NewMesExt}, by letting $a\to+\infty$ in Theorem \ref{Thm:CoarseExtThm}, we can prove Theorem \ref{Thm:OptJetExt} in the case of $A=0$ and $c(t)\equiv e^t$, but we shall replace the definition \eqref{Eq:DefHermJm} by
	\begin{equation*}
		\int_D \rho|f|_{E|_S\otimes J^{(m)}}^2 dV_S := \varlimsup_{a\rightarrow+\infty} e^a \int_{\{2(m+p)\psi<-a\}\cap U} \tilde{\rho}|F|_h^2 dV_\Omega.
	\end{equation*}
\end{remark}

\section{Generalized Suita Conjectures with Jets and Weights}

We survey different approaches to Suita's conjecture and its various generalizations. We present a new and unified proof for generalized Suita conjectures with jets and weights, which is based on the concavity of certain minimal $L^2$ integrals and the necessary condition for linearity. Additionally, we provide some examples and counterexamples for the equalities in generalized Suita conjectures.

\subsection{Introduction} \hfill

In \cite{Suita1972}, Suita conjectured an inequality between the Bergman kernel and the logarithmic capacity of a hyperbolic Riemann surface. Later, Ohsawa \cite{Ohsawa1993Add} noticed a connection between the $L^2$ extension problem and Suita's conjecture, and he was able to prove a weaker inequality.
By proving $L^2$ extension theorems with optimal estimates, Blocki \cite{Blocki2013} and Guan-Zhou \cite{GuanZhou2012} solved Suita's conjecture.
By carefully using the optimal $L^2$ extension theorem with `gain', Guan-Zhou \cite{GuanZhou2015} also settled the equality part of the conjecture (i.e. to characterize when the equality holds).
Since then, various approaches (see \cite[etc]{Blocki2014, BerndtssonLempert2016, DongArxiv}) and generalizations (see \cite[etc]{GuanZhou2015CN, BlockiZwonek2015, BlockiZwonek2018, BlockiZwonek2020}) to Suita's conjecture have emerged.

The first purpose of this section is to survey these progress made in Suita's conjecture and its generalizations.
We also present a new approach to one dimensional generalizations with jets (see \cite{BlockiZwonek2018}) or weights (see \cite{GuanZhou2015CN, GuanZhou2015}), which is based on the concavity of certain minimal $L^2$ integrals (see \cite{Guan2019}) and the necessary condition for linearity (see Remark \ref{Rmk:ConcaveLinear}).
Actually, we prove a result unifying \cite{BlockiZwonek2018} and \cite{GuanZhou2015CN, GuanZhou2015} (see Theorem \ref{Thm:GenSuita}).
We also construct a family of counterexamples for the equality in higher order Suita conjecture (see Theorem \ref{counterexample}), which contrasts with the phenomenon observed in simply/doubly connected planar domains.

Our approach is also applicable to higher dimensional generalizations (see \cite{BlockiZwonek2015, BlockiZwonek2020}), and we obtain a necessary condition for the equality case (see Proposition \ref{Prop}).
To authors' knowledge, the only known example for the equality case is the biholomorphic image of a balanced domain (with a possible closed pluripolar set removed).
In this article, we provide a new family of examples (see Theorem \ref{Example1} and \ref{Example2}) for the equality in higher dimensional Suita conjecture.

\subsection{Capacities and Kernels on Riemann Surfaces} \hfill

In this section, $\Omega$ is a potential-theoretical hyperbolic Riemann surface, which means that $\Omega$ admits a negative non-constant subharmonic function. Then $\Omega$ has non-trivial Green's functions (see \cite{GTM71}). Recall that, the \textbf{Bergman kernel} of $\Omega$ is
\begin{equation*}
	\kappa_\Omega(z) := \sup\left\{ \sqrt{-1}F(z)\wedge\overline{F(z)}: F\in\Gamma(\Omega,K_\Omega), \int_\Omega \tfrac{\sqrt{-1}}{2} F\wedge\overline{F} \leqslant 1 \right\},
\end{equation*}
and the \textbf{exact Bergman kernel} of $\Omega$ is
\begin{equation*}
	\widetilde{\kappa}_\Omega(z) := \sup\left\{ \sqrt{-1}\pd f(z)\wedge\overline{\pd f(z)}: f\in\calO(\Omega), \int_\Omega \tfrac{\sqrt{-1}}{2} \pd f\wedge\overline{\pd f} \leqslant 1 \right\}.
\end{equation*}
Let $(V,w)$ be a coordinate chart of $\Omega$. We write $\kappa_\Omega|_V = B_\Omega|dw|^2$, $\widetilde{\kappa}_\Omega|_V = \widetilde{B}_\Omega|dw|^2$ and $c_D(z) := \sqrt{\pi\widetilde{B}_\Omega(z)}$. By definitions, $\widetilde{B}_\Omega \leqslant B_\Omega$. Recall that, the \textbf{logarithmic capacity} of $\Omega$ is locally defined by
\begin{equation*}
	c_\beta(z_0) := \lim_{z\to z_0} \exp \big(G_\Omega(z,z_0)-\log|w(z)-w(z_0)|\big),
\end{equation*}
and the \textbf{analytic capacity} of $\Omega$ is locally defined by
\begin{equation*}
	c_B(z_0) := \sup\left\{ |\tfrac{\pd f}{\pd w}(z_0)|: f\in\calO(\Omega),f(z_0)=0,\sup\nolimits_\Omega|f|\leqslant1 \right\}.
\end{equation*}
Clearly, $c_\beta|dw|$ and $c_B|dw|$ are globally defined conformal invariants. 

In the following, we collect some results on the comparison between these conformal invariants.

\begin{theorem}[see \cite{GuanZhou2015}] \label{Thm:cBcBeta}
	$c_B\leqslant c_\beta$. Moreover, $c_B(z_0)=c_\beta(z_0)$ for some $z_0\in\Omega$ if and only if there exists a holomorphic function $g\in\calO(\Omega)$ such that $\log|g|=G_\Omega(\cdot,z_0)$.
\end{theorem}

\begin{proof}
	Let $\mathcal{F}_{z_0} := \{f\in\calO(\Omega): f(z_0)=0, ~ \sup_\Omega|f|\leqslant1\}$. Since $\mathcal{F}_{z_0}$ is a normal family, there exists an $h\in\mathcal{F}_{z_0}$ with $|\frac{\pd h}{\pd w}(z_0)| =c_B(z_0)$. If $c_B(z_0)=0$, there is nothing to prove.
	In the following, we assume that $c_B(z_0)>0$. By the maximum principle, $|h|<1$ everywhere. Since $\log|h|<0$ is subharmonic on $\Omega$ and $\log|h(z)|-\log|w(z)-w(z_0)|$ is bounded near $z_0$, we know $\log|h|\leqslant G_\Omega(\cdot,z_0)$, and then
	\begin{align*}
		c_B(z_0) = |\tfrac{\pd h}{\pd w}(z_0)| & = \lim_{z\to z_0} \exp\big(\log|h(z)|-\log|w(z)-w(z_0)|\big) \\
		& \leqslant \lim_{z\to z_0} \exp\big(G_\Omega(z,z_0)-\log|w(z)-w(z_0)|\big) = c_\beta(z_0).
	\end{align*}
	Therefore, $c_B(z_0) \leqslant c_\beta(z_0)$ in general.
	
	If $c_B(z_0)>0$, then $\varphi:=\log|h|-G_\Omega(\cdot,z_0)\leqslant0$ is a subharmonic function on $\Omega$ and $\varphi(z_0)=\log\frac{c_B(z_0)}{c_\beta(z_0)}$. If $c_B(z_0)=c_\beta(z_0)$, then $\varphi(z_0)=0$. By the maximum principle, $\varphi\equiv0$, i.e. $\log|h|=G_\Omega(\cdot,z_0)$. Conversely, if there exists an $g\in\calO(\Omega)$ such that $\log|g|=G_\Omega(\cdot,z_0)$, then $g\in\mathcal{F}_{z_0}$ and $c_\beta(z_0)=|\tfrac{\pd g}{\pd w}(z_0)|\leqslant c_B(z_0)$. This implies $c_B(z_0)=c_\beta(z_0)$.
\end{proof}

\begin{theorem}[{Sakai \cite{Sakai1969}}]
	$c_D \leqslant c_B$. Moreover, $c_D(z_0)=c_B(z_0)>0$ for some $z_0\in\Omega$ if and only if $\Omega$ is conformally equivalent to the unit disc less a possible closed set which is expressed as the union of at most a countable number of compact sets of class $\mathcal{N}_B$.
\end{theorem}

\begin{theorem}[Suita \cite{Suita1972}] \label{Thm:SuitacB}
	$\pi B_\Omega\geqslant c_B^2$. Moreover, $\pi B_\Omega(z_0)=c_B(z_0)^2$ for some $z_0\in\Omega$ if and only if $\Omega$ is conformally equivalent to the unit disc less a possible closed set of inner capacity zero.
\end{theorem}

Recall that, a compact set $E$ in $\hat{\CC}=\CC\cup\{\infty\}$ is of class $\mathcal{N}_B$ if all bounded holomorphic functions on $\hat{\CC}\backslash E$ are constant, and a closed set $E$ in $\DD$ has inner capacity zero if and only if $E$ is polar.

In 1972, Suita \cite{Suita1972} conjectured that the curvature of $c_\beta|dw|$ is not greater than $-4$, i.e.
\begin{equation*}
	- \frac{4}{c_\beta^2} \frac{\pd^2 \log c_\beta}{\pd w\pd\overline{w}} \leqslant -4,
\end{equation*}
and the equality holds at some point if and only if $\Omega$ is conformally equivalent to the unit disc less a possible closed set of inner capacity zero.

According to \cite{Suita1972}, $\frac{\pd^2}{\pd w\pd\overline{w}}(\log c_\beta) = \pi B_\Omega$, then the inequality in Suita's conjecture is equivalent to
\begin{equation*}
	\pi B_\Omega(z) \geqslant c_\beta(z)^2.
\end{equation*}
For doubly connected planar domain $\Omega$ with no degenerate boundary component, Suita \cite{Suita1972} proved that $\pi B_\Omega > c_\beta^2$.

In \cite{Ohsawa1993Add}, Ohsawa observed a connection between the $L^2$ extension problem and the inequality in Suita's conjecture, and he proved that $750\pi B_\Omega(z)\geqslant c_\beta(z)^2$. Since then, there are many attempts to sharpen the estimate. In 2012, by proving $L^2$ extension theorems with optimal estimates, Blocki \cite{Blocki2013} (for planar domains) and Guan-Zhou \cite{GuanZhou2012} (for Riemann surfaces) solved the inequality part of the conjecture. Later, Guan-Zhou \cite{GuanZhou2015} also settled the equality part of the conjecture through a careful use of the optimal $L^2$ extension theorem with 'gain'.

\begin{theorem}[Blocki \cite{Blocki2013}; Guan-Zhou \cite{GuanZhou2012,GuanZhou2015}] \label{Thm:SolSuitaConj}
	$\pi B_\Omega\geqslant c_\beta^2$. Moreover, $\pi B_\Omega(z_0) = c_\beta(z_0)^2$ for some $z_0\in\Omega$ if and only if $\Omega$ is conformally equivalent to the unit disc less a possible closed polar set.
\end{theorem}

In summary, one has
\begin{equation*}
	\pi \widetilde{B}_\Omega \leqslant c_B^2 \leqslant c_\beta^2 \leqslant \pi B_\Omega,
\end{equation*}
and Theorems \ref{Thm:cBcBeta} to \ref{Thm:SolSuitaConj} also give the necessary and sufficient conditions for these inequalities to become equalities.

\subsection{Various Approaches to the Suita Conjecture} \label{Sec:VariousSuita} \hfill

After \cite{Blocki2013,GuanZhou2012,GuanZhou2015}, there are several new approaches to the Suita conjecture.
For the inequality part, Blocki \cite{Blocki2014} gave a new proof based on the tensor power trick, and Berndtsson-Lempert \cite{BerndtssonLempert2016} presented another proof based on the log-psh variation of fibrewise Bergman kernels.
Recently, Dong \cite{DongArxiv} proposed a simplified proof for the equality part by using Maitani-Yamaguchi's \cite{MaitaniYamaguchi2004} variation formula for fibrewise Bergman kernels.
In Section \ref{Sec:Suita}, we presented a slightly different proof for the equality part of Suita's conjecture, which is based on the concavity of certain minimal $L^2$ integrals and the necessary condition for linearity.

In the following, we compare these different approaches to the Suita conjecture. We shall adjust the original notations to ensure consistency.
As before, $\Omega$ is a hyperbolic Riemann surface, $(V,w)$ is a connected coordinate chart around $z_0\in\Omega$, $\kappa_\Omega=B_\Omega|dw|^2$ is the Bergman kernel, $c_\beta|dw|$ is the logarithmic capacity, and $G:=G_\Omega(\cdot,z_0)$ is the Green function.
Suita \cite{Suita1972} conjectured that $\pi B_\Omega(z_0)\geqslant c_\beta(z_0)^2$, and the equality holds if and only if $\Omega$ is conformally equivalent to $\DD$ less a possible closed polar set.

\textbf{The inequality part of Suita's conjecture.}

\begin{proof}[\itshape The approach of Blocki \cite{Blocki2013} and Guan-Zhou \cite{GuanZhou2015}.]
	As noticed by Ohsawa \cite{Ohsawa1993Add}, proving the inequality is equivalent to prove an $L^2$ extension theorem with optimal estimate, i.e. to find a holomorphic 1-form $F\in\Gamma(\Omega,K_\Omega)$ with $F(z_0) = dw$ and
	\begin{equation*}
		\int_\Omega \tfrac{\sqrt{-1}}{2} F\wedge\overline{F} \leqslant \frac{\pi}{c_\beta(z_0)^2}.
	\end{equation*}
	The existence of such $F$ would imply $B_\Omega(z_0) \geqslant (\int_\Omega \tfrac{\sqrt{-1}}{2} F\wedge\overline{F})^{-1} \geqslant \pi^{-1}c_\beta(z_0)^2$.
	By proving certain optimal $L^2$ extension theorems, Blocki and Guan-Zhou solved the inequality part of the conjecture.
\end{proof}

\begin{proof}[\itshape The approach of Blocki \cite{Blocki2014}.]
	Using Donnelly-Fefferman's $L^2$ estimates of $\bar{\pd}$, together with a tensor power trick, Blocki showed that, for pseudoconvex domain $D\Subset\CC^n$,
	\begin{equation} \label{Eq:LowBBer}
		B_D(z) \geqslant \frac{1}{e^{2na}\textup{Vol}(\{G_D(\cdot,z)<-a\})}, \quad z\in D, a\in\RR_+.
	\end{equation}
	Here, $G_D$ is the pluricomplex Green function of $D$. In dimension $1$, for $a\gg1$, the sublevel set $\{G_D(\cdot,z)<-a\}$ is almost a disc with radius $e^{-a}c_\beta(z;D)^{-1}$. The right hand side converges to $\pi^{-1}c_\beta(z;D)^2$ as $a\to+\infty$, and then $B_D(z) \geqslant \pi^{-1}c_\beta(z;D)^2$.
\end{proof}

\begin{proof}[\itshape The approach of Berndtsson-Lempert \cite{BerndtssonLempert2016}.]
	For each $t\geqslant0$, define $\Omega_t:=\{2G<-t\}$ and $B(t):=B_{\Omega_t}(z_0)$. Consider the following Stein manifold:
	\begin{equation*}
		\mathcal{X} = \{ (\tau,z)\in\CC\times\Omega: 2G(z)+\re\tau<0 \}.
	\end{equation*}
	By the log-psh variation of fibrewise Bergman kernels (see \cite{MaitaniYamaguchi2004, Berndtsson2006}), $\tau\mapsto\log B(\re\tau)$ is a psh function, then $t\mapsto\log B(t)$ is a convex function.
	By the local behavior of $G_\Omega(\cdot,z_0)$ near $z_0$, one has $B(t) \sim \pi^{-1}c_\beta(z_0)^2e^t$ as $t\to+\infty$.
	Since the convex function $k(t):=\log B(t)-t$ is bounded from above as $t\to+\infty$, $k(t)$ must be decreasing. Therefore, $k(0)\geqslant \lim_{t\to+\infty}k(t)$, which implies $B_\Omega(z_0) \geqslant \pi^{-1}c_\beta(z_0)^2$.
	
	There is a slightly different proof due to Guan \cite{Guan2019Saitoh}. Since $B(t)$ is the reciprocal of certain minimal $L^2$ integral on $\Omega_t$, by a general concavity property (see \cite[Proposition 4.1]{Guan2019}), $r\mapsto\frac{1}{B(-\log r)}$ is a concave increasing function on $(0,1]$.
	Therefore, $\frac{1}{r B(-\log r)}$ is decreasing in $r$ and $e^{k(t)}=e^{-t}B(t)$ is decreasing in $t$. As a consequence, $B_\Omega(z_0) \geqslant \pi^{-1}c_\beta(z_0)^2$.
	Notice that, in this particular case, $B(t)\leqslant Ce^t$ for some $C>0$, then the convexity of $\log B(t)$ implies the concavity of $\frac{1}{B(-\log r)}$.
\end{proof}

\textbf{The equality part of Suita's conjecture.}

\begin{proof}[\itshape The approach of Guan-Zhou \cite{GuanZhou2015}.]
	After a suitable change of coordinate, we may assume that $G(w) \equiv \log|c_\beta(z_0)w|$ on $V$. Since $\pi B_\Omega(z_0) = c_\beta(z_0)^2$, there exists a unique holomorphic 1-form $F\in\Gamma(\Omega,K_\Omega)$ such that $F(z_0)=dw$ and $\int_\Omega \frac{\sqrt{-1}}{2}F\wedge\overline{F} = \pi c_\beta(z_0)^{-2}$. Given $r_1<r_2<r_3<0$ such that $\{2G<r_3\} \Subset V$, let $d_1(t)\equiv1$ and let $d_2(t)$ be a smooth function on $(-\infty,0)$ so that $d_1\equiv d_2$ on $(-\infty,r_1)\cup(r_3,0)$, $d_1>d_2$ on $(r_1,r_2)$, $d_1<d_2$ on $(r_2,r_3)$, $d_2(t)e^t$ is increasing on $(-\infty,0)$, and $\int_{-\infty}^0 d_2(t)e^tdt = \int_{-\infty}^0 d_1(t)e^tdt = 1$.
	
	According to \cite[Theorem 2.2]{GuanZhou2015}, there exists a holomorphic 1-form $F'\in\Gamma(\Omega,K_\Omega)$ with $F'(z_0)=dw$ and $\int_\Omega \frac{\sqrt{-1}}{2} d_2(2G) F'\wedge\overline{F'} \leqslant \pi c_\beta(z_0)^{-2}$. By careful computations,
	\begin{equation*}
		\int_\Omega \tfrac{\sqrt{-1}}{2}F'\wedge\overline{F'} \leqslant \int_\Omega \tfrac{\sqrt{-1}}{2} d_2(2G)F'\wedge\overline{F'} \leqslant \pi c_\beta(z_0)^{-2}.
	\end{equation*}
	On the other hand, $\int_\Omega \tfrac{\sqrt{-1}}{2}F'\wedge\overline{F'} \geqslant B_\Omega(z_0)^{-1} = \pi c_\beta(z_0)^{-2}$. Since the minimal element is unique, one has $F\equiv F'$, and then
	\begin{equation*}
		\int_\Omega \tfrac{\sqrt{-1}}{2}F\wedge\overline{F} = \int_\Omega \tfrac{\sqrt{-1}}{2} d_2(2G)F\wedge\overline{F}.
	\end{equation*}
	By careful computations, this equality implies $F|_V \equiv dw$ (see \cite[Lemma 4.21]{GuanZhou2015}). In summary,
	\begin{center}
		\itshape if $\pi B_\Omega(z_0) = c_\beta(z_0)^2$ and $(V,w)$ is a connected coordinate chart \\ around $z_0$ such that $G|_V\equiv\log|c_\beta(z_0)w|$, then there exists a \\ global holomorphic 1-form $F\in\Gamma(\Omega,K_\Omega)$ with $F|_V\equiv dw$.
	\end{center}
	Using this fact and the theory of Riemann surfaces, Guan-Zhou constructed a holomorphic function $g\in\calO(\Omega)$ such that $G=\log|g|$. By Theorem \ref{Thm:cBcBeta} and \ref{Thm:SuitacB}, $c_B(z_0)^2 = c_\beta(z_0)^2 = \pi B_\Omega(z_0)$, and $\Omega$ is conformally equivalent to $\DD$ less a possible closed polar set.
\end{proof}

\begin{proof}[\itshape The approach of Dong \cite{DongArxiv}.]
	For each $t\geqslant0$, put $\Omega_t=\{2G<-t\}$ and $B(t)=B_{\Omega_t}(z_0)$. Let $\kappa_t(\cdot,\cdot)$ be the Bergman kernel of $\Omega_t$, i.e. $\kappa_t(x,y) = \sqrt{-1}\sum_\alpha \phi_t^\alpha(x)\wedge\overline{\phi_t^\alpha(y)}$, where $\{\phi_t^\alpha\}_\alpha$ is a complete orthonormal basis of $A^2(\Omega_t,K_\Omega)$. If we write $\kappa_t(\cdot,z_0)=K_t(\cdot)\wedge\overline{dw}$, then $K_t\in\Gamma(\Omega_t,K_\Omega)$ is the unique holomorphic 1-form with minimal $L^2$-norm such that $K_t(z_0)=B(t)dw$.
	Recall from \cite{BerndtssonLempert2016} that, $k(t):=\log B(t)-t$ is a decreasing function. If $\pi B_\Omega(z_0)=c_\beta(z_0)^2$, then $k(t)$ is constant and $B(t)\equiv B_\Omega(z_0)e^t$.
	
	Using the variation formula of Maitani-Yamaguchi \cite{MaitaniYamaguchi2004}, Dong proved that $K_0|_{\Omega_t} \equiv K_te^{-t}$ for all $t\geqslant0$ such that $K_t$ is zero-free.
	In this proof, he needed to approximate $\Omega$ by smoothly bordered Riemann surfaces while keeping the equality  $\pi B_\Omega(z_0)=c_\beta(z_0)^2$.
	
	Assume that $(V,w)$ is a connected coordinate chart around $z_0$ such that $G|_V\equiv\log|c_\beta(z_0)w|$, then $\Omega_t\Subset V$ and $K_t(w)\equiv B(t)dw$ for $t\gg1$. Consequently, $K_0|_V\equiv B_\Omega(z_0)dw$.
	Set $F:=K_0/B_\Omega(z_0)$, then $F|_V\equiv dw$. By careful analysis, Dong showed that $g:=F/(2\pd G)$ is a holomorphic function on $\Omega$ and $G=\log|g|$. By \cite[Theorem 1]{Minda1987}, $\Omega$ is conformally equivalent to $\DD$ less a possible closed polar set.
\end{proof}

\begin{proof}[\itshape The approach of Section \ref{Sec:Suita}.]
	For each $t\geqslant0$, set $\Omega_t=\{2G<-t\}$ and $B(t)=B_{\Omega_t}(z_0)$. Let $F_t\in\Gamma(\Omega_t,K_\Omega)$ be the unique holomorphic 1-form with minimal $L^2$-norm such that $F_t(z_0)=dw$, then $B(t)=\|F_t\|^{-2}$.
	By the concavity of minimal $L^2$ integrals (see \cite{Guan2019}), $r\mapsto\|F_{-\log r}\|^2$ is a concave function. If $\pi B_\Omega(z_0) = c_\beta(z_0)^2$, then $B(t)\equiv B_\Omega(z_0)e^t$ and $\|F_{-\log r}\|^2\equiv r/B_\Omega(z_0)$ is linear in $r$.
	By the necessary condition for linearity (see Remark \ref{Rmk:ConcaveLinear}), $F_0|_{\Omega_t} \equiv F_t$ for any $t\geqslant0$. If $(V,w)$ is a connected coordinate chart around $z_0$ such that $G|_V\equiv\log|c_\beta(z_0)w|$, then $F_s\equiv dw$ for $s\gg1$ and then $F_0|_V=dw$. The other part of the proof is the same as in \cite{GuanZhou2015}.
\end{proof}

\begin{remark}
	The first part of all three proofs is to find a  holomorphic 1-form $F\in\Gamma(\Omega,K_\Omega)$ such that $F|_V\equiv dw$, in which $(V,w)$ is a connected coordinate chart around $z_0$ with $G|_V\equiv\log|c_\beta(z_0)w|$. But the approaches are different. Having such an $F$, one can construct an $g\in\calO(\Omega)$ such that $\log|g|=G_\Omega(\cdot,z_0)$.
	
	We remark that, without requiring $\pi B_\Omega(z_0)=c_\beta(z_0)^2$ in advance, the existence of $g\in\calO(\Omega)$ satisfying $\log|g| = G_\Omega(\cdot,z_0)$ guarantees the rigidity.
	This fact is implicitly contained in Suita's article \cite[p213]{Suita1972}.
	It also follows from a theorem of Minda \cite{Minda1987}: if $f:X\to Y$ is a holomorphic map between hyperbolic Riemann surfaces, and $G_Y(f(a),f(b))=G_X(a,b)$ for some $a\neq b$, then $f$ is injective and $Y\backslash f(X)$ is a closed set of capacity zero.
\end{remark}

\subsection{One Dimensional Generalizations} \label{Sec:GenSuita} \hfill

In this section, $\Omega$ is a hyperbolic Riemann surface, $z_0$ is a distinguished point of $\Omega$ and $(V,w)$ is a coordinate chart around $z_0$. Let $p:\DD\to\Omega$ be a universal covering of $\Omega$. Recall that the group of deck transformations $\text{Deck}(\DD/\Omega)$ is isomorphic to the fundamental group $\pi_1(\Omega)$.
Therefore, any $\sigma\in\pi_1(\Omega)$ can be identified with an element in $\textup{Aut}(\DD)$ which we shall also denote by $\sigma$. Moreover, any such automorphism satisfies $p\circ\sigma=p$.

\begin{lemma} 
	If $f_1$ and $f_2$ are holomorphic functions on a connected complex manifold $M$ such that $|f_1|\equiv|f_2|$, then $f_1\equiv\alpha f_2$ for some $\alpha\in\CC$ with $|\alpha|=1$.
\end{lemma}

\begin{proof}
	Apply the Riemann extension theorem and the maximum principle to $f_1/f_2$.
\end{proof}

\begin{lemma}
	There exists an $g\in\calO(\DD)$ such that $\log|g|=p^* G_\Omega(\cdot,z_0)$.
\end{lemma}

\begin{proof}
	By the Weierstrass theorem for open Riemann surfaces, there is an $h\in\calO(\Omega)$ so that $h(z_0)=0$, $dh(z_0)\neq0$ and $h|_{\Omega\backslash\{z_0\}}\neq0$. Since $p^{*}(G_\Omega(\cdot,z_0)-\log|h|)$ is harmonic on $\DD$, there exists an $f\in\calO(\DD)$ such that $\re f=p^{*}(G_\Omega(\cdot,z_0)-\log|h|)$. Let $g:=p^*(h)\exp(f)$, then $g\in\calO(\DD)$ and $\log|g|=p^* G_\Omega(\cdot,z_0)$.
\end{proof}

Let $g\in\calO(\DD)$ be a holomorphic function such that $\log|g|=p^* G_\Omega(\cdot,z_0)$. For any $\sigma\in\pi_1(\Omega)$, we have $|g| = \exp(p^* G_\Omega(\cdot,z_0)) = \exp(\sigma^*p^* G_\Omega(\cdot,z_0)) = |\sigma^*g|$, which implies $\sigma^*g/g$ is a constant of modulus one. Clearly,
\begin{equation*}
	\chi_{z_0}:\sigma\in\pi_1(\Omega) \mapsto\sigma^*g/g\in\mathbb{S}^1
\end{equation*}
is a group homomorphism, which is independent of the choice of $g$.

Let $\eta$ be a harmonic function on $\Omega$, then there exists a holomorphic function $\xi\in\calO(\DD)$ so that $|\xi|=\exp(p^*\eta)$. For any $\sigma\in\pi_1(\Omega)$, we have $|\xi| = \exp(p^*\eta) = \exp(\sigma^*p^*\eta) = |\sigma^*\xi|$, and then $\sigma^*\xi/\xi$ is a constant of modulus one. Clearly,
\begin{equation*}
	\chi_\eta:\sigma\in\pi_1(\Omega) \mapsto\sigma^*\xi/\xi\in\mathbb{S}^1	
\end{equation*}
is also a group homomorphism, which is independent of the choice of $\xi$.

Given a group homomorphism $\chi\in\textup{Hom}(\pi_1(\Omega),\mathbb{S}^1)$, we define
\begin{gather*}
	\calO^\chi(\Omega) := \left\{ f\in\calO(\DD): \sigma^*f=\chi(\sigma)f \text{ for all } \sigma\in\pi_1(\Omega) \right\}, \\
	\Gamma^\chi(\Omega) := \left\{ F\in\Gamma(\DD,K_\DD): \sigma^*F=\chi(\sigma)F \text{ for all } \sigma\in\pi_1(\Omega) \right\}.
\end{gather*}
A typical element $f\in\calO^\chi(\Omega)$ (resp. $F\in\Gamma^\chi(\Omega)$) is called a multiplicative function (resp. Prym differential). Recall that, the \textbf{multiplicative Bergman kernel} (or $\chi$-Bergman kernel) of $\Omega$ is defined by
\begin{equation*}
	\kappa_\Omega^\chi(z) := \sup\left\{ \sqrt{-1}F(z)\wedge\overline{F(z)}: F\in\Gamma^\chi(\Omega), \int_\Omega \tfrac{\sqrt{-1}}{2} F\wedge\overline{F} \leqslant 1 \right\}.
\end{equation*}
Since $p_*(F\wedge\overline{F})$ is well-defined on $\Omega$, we can simply write $F\wedge\overline{F}$ on $\Omega$.

The extended Suita conjecture (see Yamada \cite{Yamada1998}) is the following:
\begin{center}
	\itshape $\pi\kappa_\Omega^\chi \geqslant c_\beta^2|dw|^2$ and the equality holds at $z_0\in\Omega$ \\
	if and only if $\chi=\chi_{z_0}$.
\end{center}
Notice that, if $\chi=\chi_{z_0}$, then $\pi\kappa_\Omega^\chi(z_0) = c_\beta(z_0)^2|dw|^2$ (see \cite[Theorem 7]{Yamada1998}).

There is an equivalent formulation in terms of \textbf{weighted Bergman kernels}. Given a harmonic function $\eta$ on $\Omega$, we define
\begin{equation*}
	\kappa_{\Omega,\eta}(z) = \sup\left\{ \sqrt{-1}F(z)\wedge\overline{F(z)}: F\in\Gamma(\Omega,K_\Omega), \int_\Omega \tfrac{\sqrt{-1}}{2}F\wedge\overline{F}e^{-2\eta} \leqslant 1 \right\}.
\end{equation*}
Then the extended Suita conjecture is equivalent to the following:
\begin{center}
	\itshape $\pi\kappa_{\Omega,\eta} \geqslant c_\beta^2e^{2\eta}|dw|^2$ and the equality holds at $z_0\in\Omega$ \\
	if and only if $\chi_\eta\chi_{z_0}=1$.
\end{center}
The inequality part of the conjecture was proved in Guan-Zhou \cite{GuanZhou2015CN} and the equality part was proved in Guan-Zhou \cite{GuanZhou2015}.

\begin{theorem}[Guan-Zhou \cite{GuanZhou2015CN,GuanZhou2015}] \label{Thm:ESC}
	$\pi\kappa_{\Omega,\eta}(z_0) \geqslant c_\beta(z_0)^2e^{2\eta(z_0)}|dw|^2$ and the equality holds if and only if $\chi_\eta\chi_{z_0}=1$.
\end{theorem}

\begin{lemma}
	Let $\eta$ be a harmonic function on $\Omega$, then $\chi_\eta\chi_{z_0}^k=1$ for some $k\in\NN$ if and only if there exists a holomorphic function $\hat{g}\in\calO(\Omega)$ such that $\log|\hat{g}|=kG_\Omega(\cdot,z_0)+\eta$.
\end{lemma}

\begin{proof}
	We choose $g,\xi\in\calO(\DD)$ such that $\log|g|=p^* G_\Omega(\cdot,z_0)$ and $|\xi|=\exp(p^*\eta)$.
	If $\chi_\eta\chi_{z_0}^k=1$, then $\sigma^*(\xi g^k) = \chi_\eta(\sigma)\xi\cdot(\chi_{z_0}(\sigma)g)^k = \xi g^k$ for any $\sigma\in\pi_1(\Omega)$. As a consequence, $\hat{g}:=p_*(\xi g^k)$ is a well-defined holomorphic function on $\Omega$. Since
	\begin{equation*}
		\log|\xi g^k| = kp^*G_\Omega(\cdot,z_0) + p^*\eta,
	\end{equation*}
	it is clear that $\log|\hat{g}| = kG_\Omega(\cdot,z_0) + \eta$.
	
	Conversely, if there is an $\hat{g}\in\calO(\Omega)$ such that $\log|\hat{g}|=kG_\Omega(\cdot,z_0)+\eta$, then $|p^*\hat{g}| = |\xi g^k|$. Hence, $\xi g^k=c\cdot p^*\hat{g}$ for some $c\in\mathbb{S}^1$. For any $\sigma\in\pi_1(\Omega)$, we have $p^*\hat{g}=\sigma^*(p^*\hat{g})$, this implies
	\begin{equation*}
		\xi g^k = \sigma^*(\xi g^k) = \chi_\eta(\sigma)\xi \cdot (\chi_{z_0}(\sigma)g)^k.
	\end{equation*}
	Therefore, $\chi_\eta(\sigma)\chi_{z_0}(\sigma)^k=1$ for all $\sigma\in\pi_1(\Omega)$.
\end{proof}

Recall that $\Omega$ is a hyperbolic Riemann surface and $(V,w)$ is a coordinate chart around $z_0\in\Omega$. Denote by $\frakm_{z_0}$ the unique maximal ideal of $\calO_{z_0}$. Given $m\in\NN$, consider the generalized Bergman kernel
\begin{equation*}
	B_\Omega^{(m)}(z_0) := \sup\left\{ \left|\frac{\pd^mf}{\pd w^m}(z_0)\right|^2:
	\begin{aligned} F\in\Gamma(\Omega,K_\Omega) \text{ with } F|_V=fdw, \\ \int_\Omega \tfrac{\sqrt{-1}}{2} F\wedge\overline{F} \leqslant 1, [f]_{z_0}\in\frakm_{z_0}^m \end{aligned} \right\}.
\end{equation*}
Clearly, $B_\Omega^{(m)}(z_0)|dw|^{2m+2}$ is independent of the choice of $(V,w)$.
For planar domain $\Omega\subset\CC_w$, following the method of \cite{Blocki2013}, Blocki-Zwonek \cite{BlockiZwonek2018} proved that
\begin{equation} \label{Eq:HighSuita}
	\pi B_\Omega^{(m)}(z_0) \geqslant m!(m+1)!c_\beta(z_0)^{2m+2}.
\end{equation}
By modifying Guan-Zhou's proof for the equality part of Suita's conjecture, Li \cite{LiPHD} obtained an equivalent condition for \eqref{Eq:HighSuita} to become an equality (also see \cite{LiXuZhou}).

\begin{theorem}[see \cite{BlockiZwonek2018} and \cite{LiPHD}] \label{Thm:HighSC}
	$\pi B_\Omega^{(m)}(z_0) \geqslant m!(m+1)!c_\beta(z_0)^{2m+2}$. Moreover, the equality holds if and only if there exists a holomorphic function $\hat{g}\in\calO(\Omega)$ such that $\log|\hat{g}|=(m+1)G_\Omega(\cdot,z_0)$.
\end{theorem}

In the following, we illustrate that the method of Section \ref{Sec:Suita} is also applicable to Theorem \ref{Thm:ESC} and \ref{Thm:HighSC}. For simplicity, it is better to consider a unified version.

\begin{theorem} \label{Thm:GenSuita}
	Let $\Omega$ be a hyperbolic Riemann surface, $\eta$ be a harmonic function on $\Omega$ and $(V,w)$ be a coordinate chart around $z_0\in\Omega$. For $m\in\NN$, we define
	\begin{equation*}
		B_{\Omega,\eta}^{(m)}(z_0) := \sup\left\{ \left|\frac{\pd^mf}{\pd w^m}(z_0)\right|^2:
		\begin{aligned} F\in\Gamma(\Omega,K_\Omega) \text{ with } F|_V=fdw, \\ \int_\Omega \tfrac{\sqrt{-1}}{2} F\wedge\overline{F}e^{-2\eta} \leqslant 1, [f]_{z_0}\in\frakm_{z_0}^m \end{aligned}
		\right\}.
	\end{equation*}
	Then
	\begin{equation} \label{Eq:GenSuita}
		\pi B_{\Omega,\eta}^{(m)}(z_0) \geqslant m!(m+1)!c_\beta(z_0)^{2m+2}e^{2\eta(z_0)}.
	\end{equation}
	Moreover, the equality holds if and only if $\chi_\eta\chi_{z_0}^{m+1}=1$, if and only if there is a holomorphic function $\hat{g}\in\calO(\Omega)$ such that $\log|\hat{g}|=(m+1)G_\Omega(\cdot,z_0)+\eta$. In this case, 
	$$ F := \pd\hat{g}-2\hat{g}\pd\eta \in \Gamma(\Omega,K_\Omega) $$
	is extremal with respect to $B_{\Omega,\eta}^{(m)}(z_0)$.
\end{theorem}

Here, a holomorphic 1-form $F\in\Gamma(\Omega,K_\Omega)$ (with $F|_V=fdw$) is said to be \textbf{extremal} with respect to $B_{\Omega,\eta}^{(m)}(z_0)$, if the following conditions are satisfied:
\begin{equation*}
	[f]_{z_0}\in\frakm_{z_0}^m, \quad B_{\Omega,\eta}^{(m)}(z_0) = \frac{1}{\int_\Omega \tfrac{\sqrt{-1}}{2} F\wedge\overline{F}e^{-2\eta}} \left|\frac{\pd^mf}{\pd w^m}(z_0)\right|^2.
\end{equation*}
Clearly, extremal 1-form always exists and it is unique up to non-zero multiplicative constants.

\begin{remark}
	Obviously, Theorem \ref{Thm:SolSuitaConj} corresponds to the case of $m=0$ and $\eta\equiv0$, Theorem \ref{Thm:ESC} is the case of $m=0$ and Theorem \ref{Thm:HighSC} is the case of $\eta\equiv0$.
	Theorem \ref{Thm:GenSuita} was announced in \cite{XuZhou2023}. We notice that, Guan-Mi-Yuan \cite{GuanMiYuanII} obtained a result generalizing this theorem: they characterized the linearity of certain minimal $L^2$ integrals on hyperbolic Riemann surfaces by using the solution of the extended Suita conjecture (i.e. Theorem \ref{Thm:ESC}). But our purpose is different, we give a new and unified proof to the inequality part and the necessity part of Theorem \ref{Thm:SolSuitaConj}, \ref{Thm:ESC} and \ref{Thm:HighSC}. For completeness, we also include a proof for the sufficiency part.
\end{remark}

Let us recall the concavity of minimal $L^2$ integrals and the necessary condition for linearity (see Remark \ref{Rmk:ConcaveLinear} and \cite[Theorem 1.3]{GuanMi2021}). For simplicity, we only focus on a special case, which is enough to prove Theorem \ref{Thm:GenSuita}.

\begin{proposition} \label{Prop:ConcaveLinear}
	Let $\Omega$ be a hyperbolic Riemann surface, $\varphi$ be a harmonic function on $\Omega$ and $\psi=2(m+1)G_\Omega(\cdot,z_0)$. For each $t\geqslant0$, let $\Omega_t:=\{\psi<-t\}$ and
	\begin{equation*}
		\calA_t := \left\{ F\in\Gamma(\Omega_t,K_\Omega): \|F\|_{\calA_t}^2 = \int_{\Omega_t} \tfrac{\sqrt{-1}}{2} F\wedge\overline{F} e^{-\varphi} < +\infty \right\}.
	\end{equation*}
	Let $F$ be a holomorphic 1-form defined in a neighbourhood of $z_0$. For each $t\geqslant0$, let $F_t\in \calA_t$ be the unique element with minimal norm that coincides with $F$ up to order $m$ at $z_0$.
	
	Set $I(t):=\|F_t\|_{\calA_t}^2$. Then $r\mapsto I(-\log r)$ is a concave increasing function on $(0,1]$ and $I(0)\leqslant I(t)e^t\leqslant I(s)e^s$ for any $0\leqslant t\leqslant s$. Moreover, if $r\mapsto I(-\log r)$ is linear on $(0,1]$, then $F_t\equiv F_0|_{\Omega_t}$ for any $t>0$. 
\end{proposition}

\begin{proof}[Proof of Theorem \ref{Thm:GenSuita}]
	Let $p:\DD\to\Omega$ be a universal covering, let $\xi\in\calO(\DD)$ and $g\in\calO(\DD)$ be holomorphic functions so that $|\xi|=\exp(p^*\eta)$ and $\log|g|=p^*G_\Omega(\cdot,z_0)$.
	Shrinking $V$ if necessary, we may assume that $V$ is connected and $p$ is biholomorphic on any connected component of $p^{-1}(V)$. Let $U$ be a component of $p^{-1}(V)$. We define $h:=p_*(g|_U)$ and $\zeta:=p_*(\xi|_U)$, then $G_\Omega(\cdot,z_0) = \log|h|$ and $|\zeta|=e^{\eta}$ on $V$.
	After a suitable change of coordinate, we further assume that $w\equiv c_\beta(z_0)^{-1}h$ on $V$.
	We will keep these notations throughout the proof.
	
	Let $\psi=2(m+1)G_\Omega(\cdot,z_0)$ and $\varphi=2\eta$. For each $t\geqslant0$, we define $\Omega_t:=\{\psi<-t\}$ and $\calA_t$ as in Proposition \ref{Prop:ConcaveLinear}. Let $F_t\in\calA_t$ be the unique element with minimal norm that coincides with $w^mdw$ up to order $m$ at $z_0$. We write $B(t) := B_{\Omega_t,\eta}^{(m)}(z_0)$, then
	\begin{equation*}
		\int_{\Omega_t} \tfrac{\sqrt{-1}}{2} F_t\wedge\overline{F_t} e^{-2\eta} = \frac{(m!)^2}{B(t)}.
	\end{equation*}
	By Proposition \ref{Prop:ConcaveLinear}, $r\mapsto \frac{1}{B(-\log r)}$ is a concave function on $(0,1]$ and
	\begin{equation*}
		B(s)e^{-s} \leqslant B(t)e^{-t} \leqslant B(0) = B_{\Omega,\eta}^{(m)}(z_0), \quad 0\leqslant t\leqslant s.
	\end{equation*}
	
	\textit{\bfseries Inequality part.}
	We may choose $s\gg1$ so that $\Omega_s\Subset V$. Recall that, $G_\Omega(\cdot,z_0)=\log|c_\beta(z_0)w|$ and $|\zeta|=e^\eta$ on $V$. Therefore, $$ \Omega_s= \left\{ |w|<c_\beta(z_0)^{-1}\exp(\tfrac{-s}{2m+2}) \right\}$$ is an open disc in $(V,w)$. Assume that $F':=udw\in\calA_s$ coincides with $w^mdw$ up to order $m$ at $z_0$. Then $v:=u/(\zeta w^m)$ is a holomorphic function on $\DD(0;r)$ satisfying $v(0)=1/\zeta(z_0)$, where $r:=c_\beta(z_0)^{-1}\exp(\frac{-s}{2m+2})$. We may expand $v$ into a power series with normal convergence on $\DD(0;r)$:
	$$ v(w) = \sum_{k=0}^\infty a_k w^k \quad\text{with}\quad a_0=\frac{1}{\zeta(z_0)}. $$
	By direct computations,
	\begin{multline*}
		\int_{\Omega_s} \tfrac{\sqrt{-1}}{2}F'\wedge\overline{F'}e^{-2\eta} = \int_{\DD(0;r)} |v(w)|^2|w|^{2m} d\lambda_w \\
		= \sum_{k=0}^\infty |a_k|^2 \int_{\DD(0;r)} |w|^{2m+2k} d\lambda_w = \sum_{k=0}^\infty \frac{\pi}{m+k+1} |a_k|^2 r^{2m+2k+2}.
	\end{multline*}
	Clearly, the above expression is minimized when $v$ is a constant function. As a consequence, for such $s\gg1$, $F_s(w)\equiv \zeta(z_0)^{-1}\zeta(w)w^mdw$ and
	\begin{align*}
		\frac{(m!)^2}{B(s)} & = \int_{\Omega_s} \tfrac{\sqrt{-1}}{2}F_s\wedge\overline{F_s}e^{-2\eta}
		= \frac{\pi}{m+1} |a_0|^2 r^{2m+2} \\
		& = \frac{\pi}{m+1} c_\beta(z_0)^{-2m-2}e^{-2\eta(z_0)}e^{-s}.
	\end{align*}
	For any $0\leqslant t\ll s$, we have
	\begin{equation}\label{Eq:ineq-part}
		\begin{aligned}
			B_{\Omega,\eta}^{(m)}(z_0) & \geqslant B(t)e^{-t} \geqslant B(s)e^{-s} \\
			& = \pi^{-1}m!(m+1)!c_\beta(z_0)^{2m+2}e^{2\eta(z_0)}.
		\end{aligned}
	\end{equation}
	this proves the inequality part of the theorem.
	
	\textit{\bfseries Equality part: necessity.}
	In the following, we assume the equality in \eqref{Eq:GenSuita} holds. According to the inequality \eqref{Eq:ineq-part},
	\begin{equation*}
		B(t)e^{-t} \equiv \pi^{-1}m!(m+1)!c_\beta(z_0)^{2m+2}e^{2\eta(z_0)} \quad (\forall t\geqslant0),
	\end{equation*}
	and then $\frac{1}{B(-\log r)}$ is a linear function of $r$. By Proposition \ref{Prop:ConcaveLinear}, $F_t \equiv F_0|_{\Omega_t}$ for any $t\geqslant0$. Since $F_s(w)\equiv\zeta(z_0)^{-1}\zeta(w) w^mdw$ for $s\gg 1$, we conclude that
	\begin{equation*}
		F_0|_V = \zeta(z_0)^{-1}\zeta(w) w^mdw.
	\end{equation*}
	Multiplying $F_0$ by a constant, we obtain a holomorphic 1-form $F\in\Gamma(\Omega,K_\Omega)$ such that $F|_V = \zeta d(h^{m+1})$ and $F$ is extremal with respect to $B_{\Omega,\eta}^{(m)}(z_0)$.
	
	According to the definitions of $\zeta$ and $h$, we know $p^*F = \xi d(g^{m+1})$ on $U$. By the uniqueness of analytic continuation, $p^*F \equiv \xi d(g^{m+1})$ on $\DD$. For any $\sigma\in\pi_1(\Omega)$, we have $p^*F = \sigma^*(p^*F)$, this implies
	\begin{equation*}
		\xi d(g^{m+1}) = \sigma^*(\xi d(g^{m+1})) = \chi_\eta(\sigma)\chi_{z_0}(\sigma)^{m+1} \cdot \xi d(g^{m+1}).
	\end{equation*}
	Therefore, $\chi_\eta(\sigma)\chi_{z_0}(\sigma)^{m+1}=1$ for all $\sigma\in\pi_1(\Omega)$, which means that $\chi_\eta\chi_{z_0}^{m+1}=1$.
	
	In this case, $\hat{g}:=p_*(\xi g^{m+1}) \in\calO(\Omega)$ and $F:=p_*(\xi d(g^{m+1})) \in\Gamma(\Omega,K_\Omega)$ are well-defined, $\log|\hat{g}|=(m+1)G_\Omega(\cdot,z_0)+\eta$, and $F$ is extremal with respect to $B_{\Omega,\eta}^{(m)}(z_0)$.
	We want a neat formula for the extremal 1-form. Notice that,
	\begin{equation*}
		\xi d(g^{m+1}) = \xi\pd(\xi^{-1}\cdot \xi g^{m+1}) = \pd(\xi g^{m+1}) - (\xi g^{m+1})\cdot\xi^{-1}\pd\xi.
	\end{equation*}
	Differentiating $|\xi|^2=\exp(2p^*\eta)$, we get
	\begin{equation*}
		\bar{\xi}\pd\xi = 2\pd(p^*\eta)\exp(2p^*\eta) = 2p^*(\pd\eta)|\xi|^2,
	\end{equation*}
	which implies $\xi^{-1}\pd\xi=2p^*(\pd\eta)$. As a consequence,
	\begin{align*}
		F = p_*\big(\xi d(g^{m+1})\big) & = p_*\big(\pd(\xi g^{m+1}) - (\xi g^{m+1})\cdot2p^*(\pd\eta)\big) \\
		& = \pd\hat{g}-2\hat{g}\pd\eta = e^{2\eta}\pd(e^{-2\eta}\hat{g}).
	\end{align*}
	Similarly, since $|g|^2=\exp(2p^*G)$, where $G:=G_\Omega(\cdot,z_0)$, we have
	\begin{align*}
		\xi d(g^{m+1}) = (m+1)(\xi g^{m+1})\cdot g^{-1}dg = (m+1)(\xi g^{m+1})\cdot 2p^*(\pd G),
	\end{align*}
	which implies $F=p_*(\xi d(g^{m+1}))=2(m+1)\hat{g}\pd G$ on $\Omega\backslash\{z_0\}$.
	
	\textit{\bfseries Equality part: sufficiency.}
	Finally, we assume that $\chi_\eta\chi_{z_0}^{m+1}=1$, then
	$$\hat{g}:=p_*(\xi g^{m+1}) \in\calO(\Omega) \quad\text{and}\quad F:=p_*(\xi d(g^{m+1})) \in\Gamma(\Omega,K_\Omega)$$
	are well-defined objects on $\Omega$.
	The above proof suggests that $F$ is extremal with respect to $B_{\Omega,\eta}^{(m)}(z_0)$. Clearly, to verify this guess, we only need to prove $\int_\Omega F \wedge\overline{F'}e^{-2\eta} = 0$ for any holomorphic 1-form $F'\in\calA_0$ (with $F'|_V=fdw$) satisfying $[f]_{z_0}\in\frakm_{z_0}^{m+1}$.
	
	Since $[\hat{g}]_{z_0}\in\frakm_{z_0}^{m+1}$ and $\hat{g}\neq0$ elsewhere, we know $F'':=F'/\hat{g}$ is a holomorphic 1-form on $\Omega$. Since $F=e^{2\eta}\pd(e^{-2\eta}\hat{g})$ and $|\hat{g}|^2=e^{2(m+1)G+2\eta}$, we have
	\begin{align*}
		\int_\Omega F\wedge\overline{F'}e^{-2\eta} & = \int_\Omega \pd(e^{-2\eta}\hat{g})\wedge\overline{\hat{g}F''} \\
		& = \int_\Omega \pd(e^{-2\eta}|\hat{g}|^2)\wedge\overline{F''} 
		= \int_\Omega \pd\left(e^{2(m+1)G}\right)\wedge\overline{F''}.
	\end{align*}
	
	We take a sequence of subdomains $D_j\ni z_0$  such that $\overline{D_j} \subset D_{j+1}$, $\Omega=\cup_jD_j$ and each $D_j$ is bounded by analytic curves (see \cite[p144]{AhlforsRS}). Then $G_j:=G_{D_j}(\cdot,z_0)$ is continuous up to $\overline{D_j}$ and $G_j\equiv0$ on $\pd D_j$. By the reflection principle, $G_j$ has a harmonic extension in some neighbourhood of $\pd D_j$.
	By the Stokes formula,
	\begin{equation*}
		\int_{D_j} \pd\left(e^{2(m+1)G_j}\right)\wedge\overline{F''} = \int_{\pd D_j} e^{2(m+1)G_j}\overline{F''} = \int_{\pd D_j} \overline{F''} = \int_{D_j} d\overline{F''} = 0.
	\end{equation*}
	
	Notice that, $e^{2(m+1)G_j}$ and $e^{2(m+1)G}$ have no singularity at $z_0$, and $G_j-G$ is a harmonic function on $D_j$. Since $D_j\nearrow\Omega$, it is clear that $G_j\searrow G$. By Harnack's theorem, $G_j-G$ decreases to $0$ uniformly on any compact subset of $\Omega$. Using the estimates on derivatives, $\pd(G_j-G)\to0$ uniformly on any compact subset of $\Omega$.
	Let $\alpha_j:=\pd(e^{2(m+1)G_j})$ and $\alpha:=\pd(e^{2(m+1)G})$. Since
	\begin{multline*}
		\pd(e^{2(m+1)G_j}) - \pd(e^{2(m+1)G}) = \pd\left[ \big(e^{2(m+1)(G_j-G)}-1\big) e^{2(m+1)G} \right] \\
		= \big(e^{2(m+1)(G_j-G)}-1\big) \pd(e^{2(m+1)G}) + 2(m+1)e^{2(m+1)G_j}\pd(G_j-G),
	\end{multline*}
	we conclude that $\alpha_j\to\alpha$ uniformly on any compact subset of $\Omega$.
	
	We recall a useful formula for Green's function $G=G_\Omega(\cdot,z_0)$:
	\begin{equation*}
		\int_\Omega \sqrt{-1}\pd G\wedge\bar{\pd}G e^{aG} = \frac{\pi}{a} \quad (\forall a>0).
	\end{equation*}
	In particular,
	\begin{equation*}
		\int_\Omega \sqrt{-1}\alpha\wedge\overline{\alpha} = 4(m+1)^2 \int_\Omega \sqrt{-1}\pd G\wedge\bar{\pd}G e^{4(m+1)G} = (m+1)\pi.
	\end{equation*}
	Similarly, $\int_{D_j} \sqrt{-1}\alpha_j\wedge\overline{\alpha_j} = (m+1)\pi $. Recall that, $F'=\hat{g}F'' \in\calA_0$, then
	\begin{gather*}
		\|F'\|_{\calA_0}^2 := \int_\Omega \sqrt{-1}F'\wedge\overline{F'}e^{-2\eta} = \int_\Omega \sqrt{-1}F''\wedge\overline{F''}e^{2(m+1)G} < +\infty.
	\end{gather*}
	We may assume that $\{2(m+1)G<-t_0\} \Subset D_1$ for some $t_0\gg1$, then
	\begin{equation*}
		\int_{\Omega\setminus D_1} \sqrt{-1}F''\wedge\overline{F''} \leqslant e^{t_0}\|F'\|_{\calA_0}^2 < +\infty.
	\end{equation*}
	Recall that $\int_{D_j} \alpha_j\wedge\overline{F''} = 0$.
	For any integers $j\geqslant k\geqslant 1$, we have
	\begin{multline*}
		|\int_\Omega \alpha\wedge\overline{F''}| = |\int_\Omega \alpha\wedge\overline{F''} - \int_{D_j} \alpha_j\wedge\overline{F''}| \\
		\leqslant |\int_{\Omega\setminus D_k} \alpha\wedge\overline{F''}| + |\int_{D_j\setminus D_k} \alpha_j\wedge\overline{F''}| + |\int_{D_k} (\alpha_j-\alpha)\wedge\overline{F''}| \\
		\leqslant \big(\|\alpha\|_{L^2(\Omega)}+\|\alpha_j\|_{L^2(D_j)}\big) \|F''\|_{L^2(\Omega\setminus D_k)} + |\int_{D_k} (\alpha_j-\alpha)\wedge\overline{F''}|.
	\end{multline*}
	Since $\|F''\|_{L^2(\Omega\setminus D_1)}<+\infty$ and $D_k\nearrow\Omega$, the first term converges to 0 as $k\to+\infty$. For fixed $k$, since $\alpha_j\to\alpha$ uniformly on $\overline{D_k}$, the second term converges to 0 as $j\to+\infty$. Let $j\to+\infty$ and then $k\to+\infty$, we conclude that $\int_\Omega \alpha\wedge\overline{F''}=0$.
	
	In summary, provided $\chi_\eta\chi_{z_0}^{m+1}=1$, we prove that $F:=p_*(\xi d(g^{m+1}))$ {\itshape is extremal with respect to} $B_{\Omega,\eta}^{(m)}(z_0)$. Recall that, $F|_V = \zeta d(h^{m+1})$ and $h\equiv c_\beta(z_0)w$ on $V$. Therefore,
	\begin{equation*}
		F|_V = (m+1)c_\beta(z_0)^{m+1}\zeta(w)w^mdw.
	\end{equation*}
	If we write $F|_V=f(w)dw$, then $[f]_{z_0}\in\frakm_{z_0}^m$ and
	\begin{equation*}
		|\tfrac{\pd^m f}{\pd w^m}(z_0)|^2 = ((m+1)!)^2 c_\beta(z_0)^{2m+2} e^{2\eta(z_0)}.
	\end{equation*}
	On the other hand, since $F=2(m+1)\hat{g}\pd G$ on $\Omega\setminus\{z_0\}$, we have
	\begin{equation*}
		\int_\Omega \tfrac{\sqrt{-1}}{2}F\wedge\overline{F}e^{-2\eta} = 4(m+1)^2\int_\Omega \tfrac{\sqrt{-1}}{2}\pd G\wedge\bar{\pd}G e^{2(m+1)G} = (m+1)\pi.
	\end{equation*}
	Since $F$ is extremal with respect to $B_{\Omega,\eta}^{(m)}(z_0)$, we conclude that
	\begin{equation*}
		B_{\Omega,\eta}^{(m)}(z_0) = \frac{((m+1)!)^2c_\beta(z_0)^{2m+2}e^{2\eta(z_0)}}{(m+1)\pi},
	\end{equation*}
	this completes the proof.
\end{proof}

\subsection{The Equality in Higher Order Suita Conjecture} \hfill

In this section, we restrict ourselves to the case of planar domains.

By the Riemann mapping theorem, any simply connected domain $\Omega\subsetneq\CC$ is conformally equivalent to the unit disk, then it is clear that $\pi B_\Omega^{(m)}(z_0)=m!(m+1)!c_\beta(z_0)^{2m+2}$ for all $z_0\in\Omega$ and $m\in\NN$.

Next, we consider a domain $\Omega\subset\CC$ of finite connectivity $n\geqslant2$. Since isolated points are removable singularities for $L^2$ holomorphic functions and upper bounded subharmonic functions, we may assume that no connected component of $\hat{\CC}\setminus\Omega$ reduces to a point.
After a conformal transformation, \textit{we assume that $\Omega$ is bounded by $n$ analytic curves $\Gamma_1,\cdots,\Gamma_n$}. Let $\omega_j$ be the harmonic measure of $\Gamma_j$ with respect to $\Omega$.
By the reflection principle, $\omega_j$ and $G_\Omega(\cdot,z_0)$ $(\forall z_0\in\Omega)$ have harmonic extensions in some neighbourhood of $\pd\Omega=\Gamma_1 \cup\cdots\cup \Gamma_n$.
It is known that the period $\int_{\Gamma_j} {^*d}G_\Omega(\cdot,z_0)$ equals to $2\pi\omega_j(z_0)$, where ${^*d}u=-\sqrt{-1}(\pd u-\bar{\pd}u)$ denotes the conjugate differential of $du$. (See \cite{AhlforsCA} for details.)

If $n=2$, then such $\Omega$ is conformally equivalent to some annulus $A_R=\{z\in\CC:1<|z|<R\}$. In a joint work \cite{LiXuZhou} with Li, by studying the multi-valued harmonic conjugate of $G_{A_R}(\cdot,z_0)$, we showed that, if $|z_0|=\exp(\frac{k}{m+1}\log R)$ for some integer $k\in[1,m]$, then there exists a holomorphic function $g\in\calO(A_R)$ such that $\log|g|=(m+1)G_{A_R}(\cdot,z_0)$. According to Theorem \ref{Thm:HighSC}, for any integers $1\leqslant k\leqslant m$, we have
\begin{equation}\label{Eq:Annulus}
	\pi B_{A_R}^{(m)}(z_0) = m!(m+1)!c_\beta(z_0;A_R)^{2m+2}, \quad |z_0|=R^{\frac{k}{m+1}}.
\end{equation}
This also follows from an explicit formula for the Green's function of $A_R$ (see \cite{JarnickiPflugBook}):
\begin{equation*}
	G_{A_R}(z,a)=\log\frac{|(1-a^{-1}z)\Pi(a,z)|}{|z|^{s(a)}}, \quad 1<a<R,
\end{equation*}
where
\begin{gather*}
	\Pi(a,z) := \frac{\prod_{\nu=1}^\infty (1-\frac{z}{a}R^{-2\nu}) (1-\frac{a}{z}R^{-2\nu})} {\prod_{\nu=1}^\infty (1-azR^{-2\nu}) (1-\frac{1}{az}R^{-2\nu+2})} \quad\text{and}\quad s(a) := 1-\frac{\log a}{\log R}.
\end{gather*}
If $a=\exp(\frac{k}{m+1}\log R)$ for some integer $k\in[1,m]$, then $s(a)=1-\frac{k}{m+1}$ and
\begin{equation*}
	g_a(z) := \frac{\big((1-a^{-1}z)\Pi(a,z)\big)^{m+1}}{z^{m+1-k}}
\end{equation*}
is a holomorphic function on $A_R$ such that $\log|g_a|=(m+1)G_{A_R}(\cdot,a)$.

As pointed out by Guan-Sun-Yuan \cite{GuanSunYuan2022}, given $z_0\in\Omega$ and $m\in\NN$, there exists a holomorphic function $g\in\calO(\Omega)$ satisfying $\log|g|=(m+1)G_\Omega(\cdot,z_0)$ if and only if $(m+1)\omega_j(z_0)\in\ZZ$ for all $1\leqslant j\leqslant n$. By Theorem \ref{Thm:HighSC}, these conditions are equivalent to $\pi B_\Omega^{(m)}(z_0)=m!(m+1)!c_\beta(z_0)^{2m+2}$. Since $0<\omega_j<1$ on $\Omega$ and $\omega_1+\cdots+\omega_n\equiv1$, in this case, one has
\begin{equation*}
	1 = \omega_1(z_0)+\cdots+\omega_n(z_0) \geqslant \tfrac{n}{m+1}.
\end{equation*}
Consequently, if the equality in $m$-order Suita conjecture holds somewhere in an $n$-connected domain, then it is necessary that $m\geqslant n-1$. (Every domain considered here is bounded by analytic curves!)
Moreover, Guan-Sun-Yuan \cite{GuanSunYuan2022} showed that, in any 3-connected domain $\Omega$, there exist some $z_0\in\Omega$ and large $m\in\NN$ such that $\pi B_\Omega^{(m)}(z_0)=m!(m+1)!c_\beta(z_0)^{2m+2}$.

In summary, if $\Omega\subsetneq\CC$ is \textit{simply connected}, then for any $m\geqslant0$, the equality in $m$-order Suita conjecture holds for every point of $\Omega$;
if $\Omega$ is \textit{doubly connected}, then for any $m\geqslant1$, the equality in $m$-order Suita conjecture holds for all points on $m$ analytic curves.

It is natural to ask, if $\Omega$ is \textit{$3$-connected}, can we find a point $z_0\in\Omega$ such that the equality in 2-order Suita conjecture holds? However, the following counterexample shows that this is impossible in general.

\begin{theorem} \label{counterexample}
	Given any integers $n\geqslant3$ and $M\gg1$, there exists a family of smoothly bounded $n$-connected domain $\Omega\subset\CC$ such that no point of $\Omega$ can satisfy the equality in $m$-order Suita conjecture, where $m=0,1,\cdots,M$.
\end{theorem}

\begin{proof}
	Let $a,\eps\in(0,1)$ be positive constants to be specified later. Define
	\begin{equation*}
		\varphi_1(z)=\frac{z+a}{1+az} \quad\text{and}\quad \varphi_2(z)=\frac{z-a}{1-az},
	\end{equation*}
	they are automorphisms of the unit disk $\DD$. Let
	\begin{equation*}
		D_1 = \{z\in\DD:|\varphi_1(z)|\leqslant\eps\} \quad\text{and}\quad D_2 = \{z\in\DD:|\varphi_2(z)|\leqslant\eps\}.
	\end{equation*}
	By the property of linear fractional transformations, $D_1$ and $D_2$ are closed disks in $\DD$:
	\begin{equation*}
		D_1,D_2 = \left\{z\in\CC: \left|z\pm\frac{a(1-\eps^2)}{1-a^2\eps^2}\right| \leqslant \frac{\eps(1-a^2)}{1-a^2\eps^2}\right\}.
	\end{equation*}
	If $\eps<a$, then $D_1$ and $D_2$ are disjoint. Let $D_3,\cdots,D_{n-1}$ be arbitrary disjoint closed disks in $\DD\setminus(D_1\cup D_2)$. Denote $\Gamma_j=\pd D_j$ and $\Gamma_n=\pd\DD$.
	Then $$\Omega:=\DD\setminus \big(\cup_{j=1}^{n-1}D_j\big)$$ is an $n$-connected domain bounded by circles $\Gamma_1,\cdots,\Gamma_n$.
	
	Let $\omega_j$ be the harmonic measure of $\Gamma_j$ with respect to $\Omega$, i.e. $\omega_j$ is a harmonic function on $\Omega$, taking boundary values 1 on $\Gamma_j$ and 0 on the other contours.
	Notice that, $\log|\varphi_1(z)|/\log\eps$ is a harmonic function on $\Omega$, taking boundary values 1 on $\Gamma_1$ and is nonnegative on the other contours. By the maximum principle,
	\begin{equation*}
		\omega_1(z) \leqslant \frac{\log|\varphi_1(z)|}{\log\eps}, \quad z\in\Omega.
	\end{equation*}
	Let $c:=\inf\{|\varphi_1(z)|:z\in\DD,\re z\geqslant0\}$. Since $\varphi_1\in\textup{Aut}(\DD)$ and $\varphi_1(-a)=0$, it clear that $0<c<1$. Therefore,
	\begin{equation*}
		\omega_1(z) \leqslant \frac{\log c}{\log\eps}, \quad z\in\Omega\cap\{z:\re z\geqslant0\}.
	\end{equation*}
	Similarly, $\omega_2\leqslant\frac{\log c}{\log\eps}$ on $\Omega\cap\{z:\re z\leqslant0\}$.
	Notice that, the constant $c$ depends only on $a$. Indeed, we can show that $c=|\varphi_1(0)|=a$. 
	
	Let $a\in(0,1)$ be arbitrarily given, we choose $\eps\ll1$ so that $\frac{\log c}{\log\eps}<\frac{1}{M+1}$, i.e. $0<\eps<c^{M+1}$. Since $\omega_1<\frac{1}{M+1}$ on $\Omega\cap\{z:\re z\geqslant0\}$, all the curves $\omega_1=q$, with $q\in \big(\frac{1}{2}\ZZ \cup\cdots\cup \frac{1}{M+1}\ZZ\big)\cap\QQ_+$, completely lie in the left half-plane. Similarly, all the curves $\omega_2=q'$, with $q'\in \big(\frac{1}{2}\ZZ \cup\cdots\cup \frac{1}{M+1}\ZZ\big)\cap\QQ_+$, completely lie in the right half-plane.
	Since these curves are disjoint, for any integer $m=0,1,\cdots,M$, we can not find a point $z_0\in\Omega$ such that $\omega_1(z_0),\omega_2(z_0)\in\frac{1}{m+1}\ZZ$. Therefore, no point of $\Omega$ can satisfy the equality in $m$-order Suita conjecture, where $m=0,1,\cdots,M$.
\end{proof}

\subsection{Higher Dimensional Generalizations} \hfill

In this section, $D$ is a bounded domain in $\CC^n$. Given a measurable function $\varphi$ on $D$ which is locally bounded from above, the weighted Bergman space is defined by
\begin{eqnarray*}
	A^2(D;e^{-\varphi}) := \left\{ f\in\calO(D): \|f\|^2=\int_D|f|^2e^{-\varphi}d\lambda < +\infty \right\},
\end{eqnarray*}
and the weighted Bergman kernel is
\begin{equation*}
	B_D(z;e^{-\varphi}) := \sup\big\{ |f(z)|^2: f\in A^2(D;e^{-\varphi}), \|f\|^2\leqslant1 \big\}.
\end{equation*}
If $\varphi\equiv0$, we shall simplify these notations as $A^2(D)$ and $B_D(z)$.

Denote by $G_D(\cdot,z)$ the pluricomplex Green function of $D$ with a pole at $z\in D$, then the \textbf{Azukawa pseudometric} of $D$ is defined by
\begin{equation*}
	A_D(z;X) := \varlimsup_{\lambda\to0} \big( G_D(z+\lambda X,z) - \log|\lambda| \big), \quad z\in D,X\in\CC^n.
\end{equation*}
Clearly, $A_D(z;\cdot)\in\textup{psh}(\CC^n)$ and $A_D(z;\tau X)=A_D(z;X)+\log|\tau|$ for any $\tau\in\CC$. Therefore, the \textbf{Azukawa indicatrix}
\begin{equation*}
	I_D(z) := \{X\in\CC^n: A_D(z;X)<0\}
\end{equation*}
is a \textit{balanced} pseudoconvex domain in $\CC^n$. (Recall that, a set $U\subset\CC^n$ is said to be balanced, if $\tau z\in U$ for every $z\in U$ and $\tau\in\CC$ with $|\tau|\leqslant1$.)

For simplicity, we assume that $z_0=0$ and $\BB^n(0;r)\subset D\subset\BB^n(0;R)$. For each $a\geqslant0$, let $D_a:=\{G_D(\cdot,0)<-a\}$. Since $\log(|z|/R)\leqslant G_D(z,0)\leqslant\log(|z|/r)$, it is easy to see that $\BB^n(0;r)\subset e^aD_a \subset\BB^n(0;R)$ and $\BB^n(0;r)\subset I_D(0) \subset\BB^n(0;R)$. Here we use the standard convention: given $U\subset\CC^n$ and $c>0$, then $cU$ is a set defined by $\{cz:z\in U\}$.

In the following, we assume that $D\subset\CC^n$ is \textit{hyperconvex}, which means that there exists a negative continuous psh exhaustion function on $D$. In this case, Zwonek \cite{Zwonek2000} proved that $A_D$ is continuous on $D\times\CC^n$ and
\begin{equation*}
	A_D(z;X) = \lim_{\substack{w\to z,w\neq z \\ (w-z)/|w-z|\to X/|X|}} \left( G_D(w,z) - \log\tfrac{|w-z|}{|X|} \right) \quad(X\neq0).
\end{equation*}
For any $0<\eps\ll1$, we can find an $a_\eps>0$ such that $e^aD_a\subset(1+\eps)I_D(0)$ for all $a>a_\eps$. Otherwise, there exist $\eps>0$, $a_j\to+\infty$ and $X_j\in\CC^n$ such that $X_j\in e^{a_j}D_{a_j}$ ($\Leftrightarrow$ $G(e^{-a_j}X_j,0)<-a_j$) and $X_j\notin(1+\eps)I_D(0)$ ($\Leftrightarrow$ $A(0;X_j)\geq\log(1+\eps)$).
Since $(1+\eps)r\leqslant|X_j|\leqslant R$, we may assume that $X_j\to X^*$ as $j\to+\infty$. Using the regularity of $A_D$, we reach a contradiction:
\begin{gather*}
	A_D(0;X^*) = \lim_{j\to+\infty} A_D(0;X_j) \geqslant \log(1+\eps), \\
	A_D(0;X^*)=\lim_{j\to+\infty} \big( G_D(e^{-a_j}X_j,0) - \log\tfrac{|e^{-a_j}X_j|}{|X^*|} \big) \leqslant 0.
\end{gather*}
Similarly, for any $0<\eps\ll1$, we can find an $a'_\eps>0$ such that $e^aD_a\supset(1-\eps)I_D(0)$ for all $a>a'_\eps$.
In summary, for any $0<\eps\ll1$, we have
\begin{equation} \label{Eq:AzuInclude}
	(1-\eps)I_D(0) \subset e^aD_a \subset (1+\eps) I_D(0) \quad \text{for all } a\gg1.
\end{equation}
As a consequence, $\lim_{a\to+\infty} e^{2na}\textup{Vol}(D_a) = \textup{Vol}(I_D(0))$. Using \eqref{Eq:LowBBer}, Blocki-Zwonek \cite{BlockiZwonek2015} obtained the following generalization to Suita's conjecture:
\begin{equation} \label{Eq:BZSuitaAzu}
	B_D(z_0) \geqslant \frac{1}{\textup{Vol}(I_D(z_0))}.
\end{equation}
Via approximation, \eqref{Eq:BZSuitaAzu} holds for general pseudoconvex domains in $\CC^n$.

Let $H(z)=\sum_{|\alpha|=m} c_\alpha z^\alpha$ be a homogeneous polynomial of degree $m$ on $\CC^n$, for any holomorphic function $f(z)$, we define
\begin{equation*}
	\pd_z^H f := \sum\nolimits_{|\alpha|=m} c_\alpha\pd_z^\alpha f = \sum\nolimits_{|\alpha|=m} c_\alpha \frac{\pd^{|\alpha|} f}{\pd z_1^{\alpha_1}\cdots\pd z_n^{\alpha_n}}.
\end{equation*}
Blocki-Zwonek \cite{BlockiZwonek2020} introduced the following generalized Bergman kernel:
\begin{equation*}
	B_D^H(w) := \sup\left\{ |\pd_z^H f(w)|^2 : f\in A^2(D),\int_D|f|^2d\lambda\leqslant1, [f]_w \in\frakm_w^m \right\}.
\end{equation*}

At first, we assume that $D\ni z_0$ is a bounded hyperconvex domain. Assume that $z_0=0$ and let $D_a:=\{G_D(\cdot,0)<-a\}$. By the monotonicity property of Bergman kernels and \eqref{Eq:AzuInclude},
\begin{equation*}
	\lim_{a\to+\infty} e^{-2(n+m)a}B_{D_a}^H(0) = \lim_{a\to+\infty} B_{e^aD_a}^H(0) = B_{I_D(0)}^H(0).
\end{equation*}
Using a tensor power trick, Blocki-Zwonek \cite{BlockiZwonek2020} proved that $$a\mapsto e^{-2(n+m)a}B_{D_a}^H(0)$$ is a decreasing function on $[0,+\infty)$, and then
\begin{equation} \label{Eq:BZSuitaH}
	B_D^H(z_0) \geqslant B_{I_D(z_0)}^H(0).
\end{equation}
Via approximation, \eqref{Eq:BZSuitaH} is true for general pseudoconvex domains.
Clearly, if $H\equiv1$, then $B_\bullet^H$ are the usual Bergman kernels and \eqref{Eq:BZSuitaH} reduces to \eqref{Eq:BZSuitaAzu}.

If $D\subset\CC$ is a planar domain, then $I_D(z_0)=\DD(0;c_\beta(z_0)^{-1})$, where $c_\beta(z_0)$ is the logarithmic capacity of $D$ at $z_0$. In this case, \eqref{Eq:BZSuitaAzu} reduces to Suita's conjecture. Let $H(z)=z^m$, then $B_D^H=B_D^{(m)}$. By direct computations, $B_{I_D(z_0)}^H(0) = \pi^{-1} m!(m+1)! c_\beta(z_0)^{2m+2}$ and then \eqref{Eq:BZSuitaH} reduces to \eqref{Eq:HighSuita}.

\quad

{\itshape In the following, we apply the approach of Section \ref{Sec:GenSuita} to prove \eqref{Eq:BZSuitaH}.}

Recall that, $D\ni z_0$ is a bounded pseudoconvex domain in $\CC^n$ and $H(z)$ is a homogeneous polynomial of degree $m$. There exists some $f\in A^2(D)$ such that $[f]_{z_0}\in\frakm_{z_0}^m$ and $\pd_z^Hf(z_0)=1$. We obtain such an $f$ by solving certain $\bar{\pd}$-equation with $L^2$ estimate. Alternatively, we can apply Corollary \ref{MainCor:Opt} directly.
Let $\psi:=2(n+m)G_D(\cdot,z_0)$. For each $t\geqslant0$, let $\calD_t:=\{\psi<-t\}$ and let $f_t\in A^2(\calD_t)$ be the unique holomorphic function with minimal $L^2$ norm such that
\begin{equation*}
	[f_t]_{z_0} \in \frakm_{z_0}^m \quad\text{and}\quad \pd_z^Hf_t(z_0) = 1.
\end{equation*}
Let $I(t):=\|f_t\|_{A^2(\calD_t)}^2$, then it is clear that $B_{\calD_t}^H(z_0)=I(t)^{-1}$.

Notice that, if $\tilde{f}$ is a holomorphic function such that $|\tilde{f}-f|^2e^{-\psi}$ is locally integrable near $z_0$, then $[\tilde{f}]_{z_0} \in \frakm_{z_0}^m$ and $\pd_z^H\tilde{f}(z_0) = 1$. Therefore, the arguments of Section \ref{Sec:LogConcave} can be applied without any change, and we conclude that (see also \cite{GuanMi2021}):

\begin{itemize}\itshape
	\item[(i)] $r\mapsto I(-\log r)$ is a concave increasing function on $(0,1]$;
	\item[(ii)] if $r\mapsto I(-\log r)$ is linear, then $f_0|_{\calD_t}\equiv f_t$ for any $t\geqslant0$.	
\end{itemize}

By the concavity, $r\mapsto I(-\log r)/r$ is decreasing on $(0,1]$, then $t\mapsto e^{-t}B_{\calD_t}^H(z_0)$ is also decreasing on $[0,+\infty)$. This monotonicity was proved in \cite{BlockiZwonek2020} by using a tensor power trick. (Recall that, $D_a:=\{G_D(\cdot,z_0)<-a\}=\calD_{2(n+m)a}$.) The remaining part of the proof is the same as \cite{BlockiZwonek2020}: if $D$ is hyperconvex, then
\begin{equation*}
	\lim_{t\to+\infty} e^{-t}B_{\calD_t}^H(z_0) = \lim_{a\to+\infty} e^{-2(n+m)a}B_{D_a}^H(z_0) = B_{I_D(z_0)}^H(0),
\end{equation*}
and the monotonicity implies the inequality \eqref{Eq:BZSuitaH}; via approximation, \eqref{Eq:BZSuitaH} is true for general pseudoconvex domains. Actually, we have something more:

\begin{proposition} \label{Prop}
	Let $D\ni z_0$ be a bounded pseudoconvex domain in $\CC^n$ and $H(z)$ be a holomorphic homogeneous polynomial of degree $m$ on $\CC^n$. For each $t\geqslant0$, let $\calD_t:=\{2(n+m)G_D(\cdot,z_0)<-t\}$ and $B(t):=B_{\calD_t}^H(z_0)$.
	
	(1) For any $N\in\NN_+$, $r\mapsto B(-\frac{1}{N}\log r)^{-N}$ is a concave function on $(0,1]$.
	
	(2) If $B_D^H(z_0) = B_{I_D(z_0)}^H(0)$, then $f_0|_{\calD_t} \equiv f_t$ for any $t\geqslant0$, where $f_t\in A^2(\calD_t)$ is the unique holomorphic function with minimal $L^2$ norm such that $[f_t]_{z_0} \in \frakm_{z_0}^m$ and $\pd_z^Hf_t(z_0)=1$.	
\end{proposition}

\begin{proof}
	We assume the same notations as above. By conclusion (i), $r\mapsto B(-\log r)^{-1}$ is a concave function on $(0,1]$. Given $N\in\NN_+$, we consider the product domain $\widetilde{D} := D\times\cdots\times D \subset \CC^{nN}$ and $\widetilde{z_0}:=(z_0,\cdots,z_0)\in\widetilde{D}$. By the same reason,
	\begin{equation*}
		r\in(0,1] \quad\mapsto\quad \left( B_{\{\Psi<\log r\}}^{H\times\cdots\times H}(\widetilde{z_0}) \right)^{-1}
	\end{equation*}
	is a concave function, where $\Psi:=2(n+m)NG_{\widetilde{D}}(\cdot,\widetilde{z_0})$. By the product properties of pluricomplex Green functions (see \cite{JarnickiPflugBook}) and generalized Bergman kernels (see \cite{BlockiZwonek2020}),
	\begin{gather*}
		\big\{\Psi<\log r\big\} = \big\{\psi<\tfrac{\log r}{N}\big\} \times \cdots \times \big\{\psi<\tfrac{\log r}{N}\big\}, \\
		B_{\{\Psi<\log r\}}^{H\times\cdots\times H}(\widetilde{z_0}) = \left( B_{\{\psi<\log r/N\}}^{H}(z_0) \right)^N = B\big(-\tfrac{1}{N}\log r\big)^N.
	\end{gather*}
	Therefore, $r\mapsto B(-\frac{1}{N}\log r)^{-N}$ is a concave function  on $(0,1]$.
	
	If $D$ is hyperconvex, we know
	\begin{equation*}
		B_D^H(z_0) \geqslant e^{-t}B_{\calD_t}^H(z_0) \geqslant B_{I_D(z_0)}^H(0), \quad t>0.
	\end{equation*}
	Via approximation, this is true for general pseudoconvex domains.
	Hence, if $B_D^H(z_0) = B_{I_D(z_0)}^H(0)$, then $B_{\calD_t}^H(z_0) = e^tB_D^H(z_0)$ for all $t$. In this case, $r\mapsto I(-\log r) = r/B_D^H(z_0)$ is linear, then it follows from the conclusion (ii) that $f_0|_{\calD_t}\equiv f_t$ for any $t\geqslant0$.
\end{proof}

If $D\subset\CC$ is a planar domain, then \eqref{Eq:BZSuitaH} reduces to \eqref{Eq:HighSuita}, and Theorem \ref{Thm:HighSC} gives a full characterization for the equality case of \eqref{Eq:HighSuita}. However, in higher dimensions, such a characterization is unknown yet. Nevertheless, the above proposition gives a necessary condition for the equality of \eqref{Eq:BZSuitaH}.

\subsection{The Equality in Higher Dimension Suita Conjecture} \hfill

In this section, we study the equality case of higher order Suita conjecture. The first example is well-known; the second example is one dimensional, it was included for completeness. As an application of Theorem \ref{Thm:GenSuita}, we also give a new family of examples for which \eqref{Eq:BZSuitaH} becomes an equality.

(1) Let $D=\{z\in\CC^n:h(z)<1\}$ be a \textit{bounded balanced pseudoconvex domain}, where $h:\CC^n\to[0,\infty)$ is homogeneous (which means that $h(\tau z)=|\tau|h(z)$ for any $\tau\in\CC$) and $\log h$ is psh. As $G_D(\cdot,0)\equiv\log h$, we know $D_a=\{G_D(\cdot,0)<-a\}=e^{-a}D$ and $I_D(0)=D$. Let $H(z)$ be a homogeneous polynomial of degree $m$ on $\CC^n$, then
\begin{equation*}
	B_D^H(0) = e^{-2(n+m)a}B_{D_a}^H(0) = B_{I_D(0)}^H(0) \quad (\forall a\geqslant0).
\end{equation*}

(2) Let $\Omega=\{z\in\CC:1<|z|<R\}$ be an \textit{annulus} and $H(z)=z^m$. We choose a point $z_0\in\Omega$ with $|z_0|=\exp(\frac{k}{m+1}\log R)$, where $k\in[1,m]$ is an integer. According to the equation \eqref{Eq:Annulus},
\begin{equation*}
	B_\Omega^H(z_0) = \pi^{-1}m!(m+1)!c_\beta(z_0;\Omega)^{2m+2} = B_{I_\Omega(z_0)}^H(0).
\end{equation*}

(3) Let $\Omega$ be a bounded domain in $\CC$ and $\eta$ be a harmonic function on $\Omega$. Let $D=\{w\in\CC^n:h(w)<1\}$ be a bounded balanced pseudoconvex domain in $\CC^n$, where $h:\CC^n\to[0,\infty)$ is homogeneous and $\log h$ is psh. We consider the following \textit{generalized Hartogs domain}:
\begin{equation*}
	\widetilde{\Omega} = \{(z,w)\in\Omega\times\CC^n: h(w)<e^{-\eta(z)} \}.
\end{equation*}
Let $\phi$ be a subharmonic exhaustion function of $\Omega$, then
\begin{equation*}
	\max \left\{ \phi(z), -\frac{1}{\log h(w)+\eta(z)} \right\}
\end{equation*}
is a psh exhaustion function of $\widetilde{\Omega}$. Hence, $\widetilde{\Omega}\subset\CC^{n+1}$ is pseudoconvex.

For any $z_0\in\Omega$, $\psi(z,w) := \max\{G_\Omega(z,z_0),\log h(w)+\eta(z)\}$ is a negative psh function on $\widetilde{\Omega}$. Clearly, $\psi(z,w)$ has a logarithmic pole at $(z_0,0)$. By the definition of pluricomplex Green functions,
\begin{equation*}
	G_{\widetilde{\Omega}}((z,w),(z_0,0)) \geqslant \max\{G_\Omega(z,z_0),\log h(w)+\eta(z)\}.
\end{equation*}

Now we estimate the Azukawa pseudometric of $\widetilde{\Omega}$ at $(z_0,0)$. For any non-zero $(X,Y)\in\CC\times\CC^n$,
\begin{align*}
	A_{\widetilde{\Omega}} & ((z_0,0);(X,Y)) := \varlimsup_{\lambda\to0} \big( G_{\widetilde{\Omega}}((z_0+\lambda X,\lambda Y), (z_0,0)) - \log|\lambda| \big) \\
	& \geqslant \varlimsup_{\lambda\to0} \big( \max\big\{ G_\Omega(z_0+\lambda X,z_0), \log h(\lambda Y)+\eta(z_0+\lambda X) \big\} - \log|\lambda| \big) \\
	& = \varlimsup_{\lambda\to0} \max\big\{ G_\Omega(z_0+\lambda X,z_0) - \log|\lambda|, \log h(Y)+\eta(z_0+\lambda X) \big\}.
\end{align*}
Notice that, if $X\neq0$, then $\lim_{\lambda\to0} \exp(G_\Omega(z_0+\lambda X,z_0) - \log|\lambda X|) = c_\beta(z_0)$, where $c_\beta(z_0)$ is the logarithmic capacity of $\Omega$ at $z_0$. Therefore,
\begin{align*}
	& A_{\widetilde{\Omega}}((z_0,0);(X,Y)) \\
	\geqslant & \max\big\{ \lim_{\lambda\to0} (G_\Omega(z_0+\lambda X,z_0) - \log|\lambda|), \log h(Y) + \lim_{\lambda\to0}\eta(z_0+\lambda X) \big\} \\
	= & \max\big\{ \log|c_\beta(z_0)X|, \log h(Y)+\eta(z_0) \big\}.
\end{align*}
Consequently,
\begin{equation} \label{Eq:AzuInd}
	I_{\widetilde{\Omega}}((z_0,0)) \subset \{ (X,Y)\in\CC\times\CC^n: |X|<c_\beta(z_0)^{-1}, h(Y)<e^{-\eta(z_0)} \}
\end{equation}
and
\begin{equation} \label{Eq:VolAzukawa}
	\Vol\big( I_{\widetilde{\Omega}}((z_0,0)) \big) \leqslant \pi c_\beta(z_0)^{-2} \times \Vol(D)e^{-2n\eta(z_0)}.
\end{equation}

Then we compute the Bergman kernel of $\widetilde{\Omega}$ at $(z_0,0)$. For any $r>0$, since $rD=\{w\in\CC^n:h(w)<r\}$ is a balanced domain, it is clear that
\begin{equation*}
	\int_{rD}|f(w)|^2d\lambda_w \geqslant \Vol(D)|f(0)|^2r^{2n}, \quad f\in \calO(rD).
\end{equation*}
Notice that, the integral on the left hand side may be infinity. For any $\tilde{f}\in A^2(\widetilde{\Omega})$, we have
\begin{align*}
	\int_{\widetilde{\Omega}} |\tilde{f}|^2 d\lambda & = \int_\Omega \left( \int_{e^{-\eta(z)}D} |\tilde{f}(z,w)|^2 d\lambda_w \right) d\lambda_z \\
	& \geqslant \Vol(D) \int_\Omega |\tilde{f}(z,0)|^2e^{-2n\eta(z)} d\lambda_z.
\end{align*}
It follows that $\tilde{f}(\cdot,0)\in A^2(\Omega;e^{-2n\eta})$ and
\begin{multline*}
	B_{\widetilde{\Omega}}((z_0,0)) = \sup\left\{ \frac{|\tilde{f}(z_0,0)|^2}{\int_{\widetilde{\Omega}}|\tilde{f}|^2d\lambda} : \tilde{f}\in A^2(\widetilde{\Omega}) \right\} \\
	\leqslant \sup\left\{ \frac{|g(z_0)|^2}{ \Vol(D)\int_\Omega|g|^2e^{-2n\eta}d\lambda} : g\in A^2(\Omega;e^{-2n\eta}) \right\} = \frac{B_{\Omega,n\eta}(z_0)}{\Vol(D)}.
\end{multline*}

On the other hand, for any $g\in A^2(\Omega;e^{-2n\eta})$, we define $\tilde{g}(z,w):=g(z)$, then
\begin{equation*}
	\tilde{g}\in A^2(\widetilde{\Omega}) \quad\text{and}\quad
	\int_{\widetilde{\Omega}} |\tilde{g}|^2 d\lambda = \Vol(D) \int_\Omega |g|^2e^{-2n\eta} d\lambda.
\end{equation*}
Then it is clear that
\begin{equation} \label{Eq:BerHartogs}
	B_{\widetilde{\Omega}}((z_0,0)) = \frac{B_{\Omega,n\eta}(z_0)}{\Vol(D)}.
\end{equation}
If $D$ is the unit ball in $\CC^n$, \eqref{Eq:BerHartogs} is also known as Ligocka's formula \cite{Ligocka}.

Combing \eqref{Eq:BerHartogs}, \eqref{Eq:BZSuitaAzu} and \eqref{Eq:VolAzukawa}, we get
\begin{equation*}
	\frac{B_{\Omega,n\eta}(z_0)}{\Vol(D)} = B_{\widetilde{\Omega}}((z_0,0)) \geqslant \frac{1}{\Vol(I_{\widetilde{\Omega}}((z_0,0)))} \geqslant \frac{c_\beta(z_0)^2e^{2n\eta(z_0)}}{\pi\Vol(D)}.
\end{equation*}
By Theorem \ref{Thm:ESC}, if $\chi_{z_0}\chi_\eta^n=1$ (equivalently, there exists some $\hat{g}\in\calO(\Omega)$ such that $\log|\hat{g}| = G_\Omega(\cdot,z_0)+n\eta$), then the above inequalities become equalities.

\begin{theorem} \label{Example1}
	Let $\Omega$ be a bounded domain in $\CC$ and $\eta$ be a harmonic function on $\Omega$. Let $D$ be a bounded balanced pseudoconvex domain in $\CC^n$ and $\widetilde{\Omega} = \{(z,w)\in\Omega\times\CC^n: w\in e^{-\eta(z)}D \}$.
	If $\chi_{z_0}\chi_\eta^n=1$ for some $z_0\in\Omega$, then 
	\begin{equation*}
		B_{\widetilde{\Omega}}((z_0,0)) = \frac{1}{\Vol(I_{\widetilde{\Omega}}((z_0,0)))}.
	\end{equation*}
\end{theorem}

It is not hard to generalize the above example to the case of \eqref{Eq:BZSuitaH}. 

\begin{theorem} \label{Example2}
	Let $\Omega$ be a bounded domain in $\CC$ and $\eta$ be a harmonic function on $\Omega$. Let $D$ be a bounded balanced pseudoconvex domain in $\CC^n$ and
	$$ \widetilde{\Omega} = \big\{ (z,w)\in\Omega\times\CC^n: w\in e^{-\eta(z)}D \big\}. $$
	Let $H_2(w)$ be a homogeneous polynomial of degree $k$ on $\CC^n$ and $H(z,w):=z^m H_2(w)$. If $\chi_{z_0}^{m+1}\chi_\eta^{n+k}=1$ for some $z_0\in\Omega$, then 
	\begin{equation*}
		B_{\widetilde{\Omega}}^H((z_0,0)) = B_W^H((0,0)), \quad \text{where } W:=I_{\widetilde{\Omega}}((z_0,0)).
	\end{equation*}
\end{theorem}

\begin{proof}
	Since $rD$ is a balanced domain, any $f\in\calO(rD)$ can be written as a compactly convergent series $f=\sum_{l=0}^\infty f_l$, where each $f_l$ is a homogeneous polynomial of degree $l$. It is clear that $[f_k]_0\in\frakm_0^k$, $\pd_w^{H_2}f_k(0) = \pd_w^{H_2}f(0)$ and
	\begin{equation*}
		\int_{rD}|f|^2d\lambda = \sum\nolimits_l \int_{rD}|f_l|^2d\lambda \geqslant \int_{rD}|f_k|^2d\lambda.
	\end{equation*}
	Notice that, these integrals may be infinite. By the definition of generalized Bergman kernels,
	\begin{equation*}
		\int_{rD}|f|^2d\lambda \geqslant \int_{rD}|f_k|^2d\lambda \geqslant \frac{|\pd_w^{H_2}f_k(0)|^2}{B_{rD}^{H_2}(0)} = \frac{|\pd_w^{H_2}f(0)|^2}{B_{D}^{H_2}(0)r^{-2(n+k)}}.
	\end{equation*}
	For any $\tilde{f}\in A^2(\widetilde{\Omega})$ with $[\tilde{f}]_{(z_0,0)} \in \frakm_{(z_0,0)}^{m+k}$, we have
	\begin{align*}
		\int_{\widetilde{\Omega}} |\tilde{f}|^2 d\lambda & = \int_\Omega \left( \int_{e^{-\eta(z)}D} |\tilde{f}(z,w)|^2 d\lambda_w \right) d\lambda_z \\
		& \geqslant \frac{1}{B_{D}^{H_2}(0)} \int_\Omega |\pd_w^{H_2}\tilde{f}(z,0)|^2e^{-2(n+k)\eta(z)} d\lambda_z.
	\end{align*}
	Therefore, $\pd_w^{H_2}\tilde{f}(\cdot,0)\in A^2(\Omega;e^{-2(n+k)\eta})$. Since $[\pd_w^{H_2}\tilde{f}(\cdot,0)]_{z_0} \in \frakm_{z_0}^m$, we know
	\begin{align*}
		B_{\widetilde{\Omega}}^H((z_0,0)) \leqslant \sup_{\tilde{f}} \frac{|\pd_z^m\pd_w^{H_2} \tilde{f}(z_0,0)|^2} {\int_\Omega |\pd_w^{H_2}\tilde{f}(z,0)|^2e^{-2(n+k)\eta(z)} d\lambda_z / B_{D}^{H_2}(0)} \\
		\leqslant B_{\Omega,(n+k)\eta}^{(m)}(z_0) \times B_D^{H_2}(0).
	\end{align*}
	Here, the supremum is taken over all $\tilde{f}\in A^2(\widetilde{\Omega})$ with $[\tilde{f}]_{(z_0,0)} \in \frakm_{(z_0,0)}^{m+k}$.
	
	It is easy to find a homogeneous polynomial $u$ of degree $k$ on $\CC^n$ such that $\int_D|u|^2d\lambda=1$ and $B_D^{H_2}(0) = |\pd_w^{H_2}u(0)|^2$. For any $g\in A^2(\Omega;e^{-2(n+k)\eta})$ with $[g]_{z_0}\in\frakm_{z_0}^m$, let $\tilde{g}(z,w):=g(z)u(w)$, then $\tilde{g}\in A^2(\widetilde{\Omega})$ and $[\tilde{g}]_{(z_0,0)} \in\frakm_{(z_0,0)}^{m+k}$. Therefore,
	\begin{align*}
		B_{\Omega,(n+k)\eta}^{(m)}(z_0) & = \sup_g \frac{|\pd_z^mg(z_0)|^2} {\int_\Omega|g|^2e^{-2(n+k)\eta}d\lambda} \\
		& = \sup_g \frac{|\pd^H\tilde{g}(z_0,0)|^2/|\pd_w^{H_2}u(0)|^2} {\int_{\widetilde{\Omega}}|\tilde{g}|^2d\lambda} \leqslant \frac{B_{\widetilde{\Omega}}^H((z_0,0))}{B_D^{H_2}(0)}.
	\end{align*}
	Here, the supremum is taken over all $g\in A^2(\Omega;e^{-2(n+k)\eta})$ with $[g]_{z_0}\in\frakm_{z_0}^m$.
	
	Denote by $W$ the Azukawa indicatrix of $\widetilde{\Omega}$ at $(z_0,0)$.
	According to \eqref{Eq:AzuInd}, $W\subset U\times V$, where $U:=\DD(0;c_\beta(z_0)^{-1})$ and $V:=e^{-\eta(z_0)}D$. By the monotonicity property and the product property of generalized Bergman kernels, 
	\begin{align*}
		B_W^H((0,0)) & \geqslant B_{U\times V}^H((0,0)) = B_U^{(m)}(0) \times B_V^{H_2}(0) \\
		& = \tfrac{m!(m+1)!}{\pi}c_\beta(z_0)^{2m+2} \times e^{2(n+k)\eta(z_0)} B_D^{H_2}(0). 
	\end{align*}
	
	Combing these results with \eqref{Eq:BZSuitaH}, we get
	\begin{align*}
		B_{\Omega,(n+k)\eta}^{(m)}(z_0) \times B_D^{H_2}(0) & = B_{\widetilde{\Omega}}^H((z_0,0)) \geqslant B_W^H((0,0)) \\
		& \geqslant \tfrac{m!(m+1)!}{\pi}c_\beta(z_0)^{2m+2} e^{2(n+k)\eta(z_0)} \times B_D^{H_2}(0).
	\end{align*}
	According to Theorem \ref{Thm:GenSuita}, if $\chi_{z_0}^{m+1}\chi_\eta^{n+k}=1$ (equivalently, there exists an $\hat{g}\in\calO(\Omega)$ such that $\log|\hat{g}|=(m+1)G_\Omega(\cdot,z_0)+(n+k)\eta$), then the above inequalities become equalities.
\end{proof}

\begin{remark}
	It is easy to find $(\Omega,\eta)$ satisfying the requirements of Theorem \ref{Example2}.
	Let $\Omega'$ be an arbitrary bounded domain in $\CC$, we choose a holomorphic function $0\not\equiv f\in\calO(\Omega')$ having at least a zero $z_0\in\Omega'$. Define $\Omega:=(\Omega'\setminus f^{-1}(0)) \cup\{z_0\}$, then $z_0$ is the only zero of $f\in\calO(\Omega)$. Let $m:=\text{ord}_{z_0}(f)-1$ and
	\begin{equation*}
		\eta:=\tfrac{1}{n+k}(\log|f|-(m+1)G_{\Omega}(\cdot,z_0)),
	\end{equation*}
	then $\eta$ is a harmonic function on $\Omega$ and $\chi_{z_0}^{m+1}\chi_\eta^{n+k}=1$.
\end{remark}

Clearly, all balanced domains are contractible. In contrast, the $\widetilde{\Omega}$ defined in Theorem \ref{Example2} has the same homotopy type as $\Omega$, which could be very complicated.

Using the transformation rule under biholomorphism and the product property of generalized Bergman kernels (see \cite{BlockiZwonek2020}), one can construct more and more examples from (1), (2) and (3). It would be interesting to find a full characterization for the equality of \eqref{Eq:BZSuitaH}.\\

\noindent\textbf{Acknowledgements} The authors would like to thank the referees for their helpful comments and suggestions.


\end{document}